\documentclass[11pt]{amsart} 
\usepackage{amsthm, amssymb, cite} 
\usepackage{amsmath, amssymb, ifthen, enumerate}
\usepackage{mathrsfs}
\usepackage{hyperref}
\usepackage[all]{xy}
\usepackage{color}

\numberwithin{equation}{section}
\newtheorem{thm}{\bf Theorem}[section]

\newtheorem{cor}[thm]{\bf Corollary}
\newtheorem{lem}[thm]{\bf Lemma}
\newtheorem{prop}[thm]{\bf Proposition}

\theoremstyle{remark}
\newtheorem{remark}[thm]{\it Remark}

\theoremstyle{definition}

\newtheorem{notation}[thm]{\bf Notation}
\newtheorem{example}[thm]{\bf Example}
\newtheorem{defn}[thm]{\bf Definition}
\newtheorem{hypothesis}[thm]{\bf Hypothesis}

\renewcommand{\leq}{\leqslant}
\renewcommand{\geq}{\geqslant}
\newcommand{\norm}{\trianglelefteq}

\newcommand{\gen}[1]{\langle #1 \rangle}

\newcommand{\foc}{\mathfrak{foc}}

\newcommand{\D}{\mathbf{D}}
\newcommand{\W}{\mathbf{W}}

\newcommand{\A}{\mathcal{A}}
\newcommand{\F}{\mathcal{F}}
\newcommand{\M}{\mathcal{M}}
\newcommand{\C}{\mathcal{C}}
\newcommand{\E}{\mathcal{E}}
\newcommand{\EE}{\mathscr{E}}
\newcommand{\K}{\mathcal{K}}
\renewcommand{\O}{\mathcal{O}}
\renewcommand{\L}{\mathcal{L}}
\renewcommand{\H}{\mathcal{H}}

\newcommand{\T}{\mathcal{T}}

\renewcommand{\phi}{\varphi}

\newcommand{\Hom}{\operatorname{Hom}}
\newcommand{\Mor}{\operatorname{Mor}}

\newcommand{\Spin}{\operatorname{Spin}}
\newcommand{\Sol}{\operatorname{Sol}}

\newcommand{\SL}{\operatorname{SL}}

\newcommand{\Aut}{\operatorname{Aut}}
\newcommand{\Out}{\operatorname{Out}}
\newcommand{\Inn}{\operatorname{Inn}}

\newcommand{\Syl}{\operatorname{Syl}}

\newcommand{\<}{\langle}
\renewcommand{\>}{\rangle}
\newcommand{\ov}{\overline}

\def\Mor{\mathrm{Mor}}
\def\Iso{\mathrm{Iso}}
\def\Aut{\mathrm{Aut}}
\def\Obj{\mathrm{Obj}}
\def\obj{\mathrm{obj}}
\def\mor{\mathrm{mor}}

\def\RR{\ensuremath{\mathbb{R}}}
\newcommand{\widebar}[1]
      {\overset{{\mskip3mu\leaders\hrule height0.4pt\hfill\mskip3mu}}{#1}
      \vphantom{#1}}
\newcommand{\ol}{\widebar}
\newcommand{\sS}{\mathscr{S}}
\renewcommand{\emptyset}{\varnothing}
\newcommand{\One}{\mathbf{1}}
\newcommand{\im}{\operatorname{Im}}

\makeatletter
\expandafter\let\csname active@char\string'\endcsname\relax
\expandafter\let\csname active@char\string!\endcsname\relax
\expandafter\let\csname active@char\string:\endcsname\relax

\def\PBtc{(PB $\times_{\amalg}$)}

\def\vp{\varphi}
\def\al{\alpha}
\def\be{\beta}
\def\ga{\gamma}
\def\De{\Delta}
\def\La{\Lambda}
\def\la{\lambda}
\def\de{\delta}
\def\si{\sigma}
\def\Ga{\Gamma}

\def\B{\ensuremath{\mathcal{B}}}
\def\C{\ensuremath{\mathcal{C}}}
\def\D{\ensuremath{\mathcal{D}}}
\def\E{\ensuremath{\mathcal{E}}}
\def\F{\ensuremath{\mathcal{F}}}
\def\H{\ensuremath{\mathcal{H}}}
\def\I{\ensuremath{\mathcal{I}}}

\def\K{\ensuremath{\mathcal{K}}}
\def\G{\ensuremath{\mathcal{G}}}
\def\L{\ensuremath{\mathcal{L}}}
\def\M{\ensuremath{\mathcal{M}}}
\def\N{\ensuremath{\mathcal{N}}}
\def\O{\ensuremath{\mathcal{O}}}
\def\P{\ensuremath{\mathcal{P}}}

\def\T{\ensuremath{\mathcal{T}}}

\def\FF{\ensuremath{\mathbb{F}}}

\def\ZZ{\ensuremath{\mathbb{Z}}}

\def\fff{\mathbf{f}}

\def\ggg{\mathbf{g}}

\def\XXX{\mathbf{X}}
\def\YYY{\mathbf{Y}}

\def\vp{\varphi}
\def\al{\alpha}
\def\be{\beta}
\def\ga{\gamma}
\def\De{\Delta}
\def\La{\Lambda}
\def\la{\lambda}
\def\de{\delta}
\def\si{\sigma}
\def\Ga{\Gamma}

\def\Top{\mathbf{\operatorname{Top}}}
\def\Out{\operatorname{Out}}
\def\Aut{\mathrm{Aut}}
\def\Hom{\mathrm{Hom}}
\def\incl{\text{incl}}

\def\op{\ensuremath{\mathrm{op}}}
\def\res{\operatorname{res}}
\def\tr{\operatorname{tr}}
\def\hocolim{\ensuremath{\mathrm{hocolim}}}

\def\Ab{\mathbf{Ab}}

\def\Inn{\operatorname{Inn}}

\def\defeq{\overset{\operatorname{def}}{=}}

\def\Fun{\operatorname{Fun}}

\def\hocolim{\operatorname{hocolim}}

\newcommand{\xto}[1]{\xrightarrow{#1}}

\newcommand{\pcomp}[1]{{#1}^\wedge_p}

\newcommand{\hhocolim}[1]{\underset{#1}{\hocolim}}

\newcommand{\modl}[1]{#1\text{-}\mathfrak{mod}}

\begin{document}
\title[]{Punctured groups for exotic fusion systems}
\author{Ellen Henke}
\author{Assaf Libman}
\author{Justin Lynd}
\thanks{J.L. was partially supported by NSF Grant DMS-1902152.}

\date{\today}

\begin{abstract}
The transporter systems of Oliver and Ventura and the localities of Chermak are
classes of algebraic structures that model the $p$-local structures of finite
groups. Other than the transporter categories and localities of finite groups,
important examples include centric, quasicentric, and subcentric linking
systems for saturated fusion systems. These examples are however not defined in
general on the full collection of subgroups of the Sylow group.  We study here
\emph{punctured groups}, a short name for transporter systems or localities on
the collection of nonidentity subgroups of a finite $p$-group.  As an
application of the existence of a punctured group, we show that the subgroup
homology decomposition on the centric collection is sharp for the fusion
system.  We also prove a Signalizer Functor Theorem for punctured groups and
use it to show that the smallest Benson-Solomon exotic fusion system at the
prime $2$ has a punctured group, while the others do not.  As for exotic fusion
systems at odd primes $p$, we survey several classes and find that in
almost all cases, either the subcentric linking system is a punctured group for
the system, or the system has no punctured group because the normalizer of some
subgroup of order $p$ is exotic. Finally, we classify punctured groups
restricting to the centric linking system for certain fusion systems on
extraspecial $p$-groups of order $p^3$.  
\end{abstract}

\maketitle

\section{Introduction}\label{S:intro}

Let $\F$ be a fusion system over the finite $p$-group $S$.  Thus, $\F$ is a
category with objects the subgroups of $S$, and with morphisms injective group
homomorphisms which contain among them the conjugation homomorphisms induced by
elements of $S$ plus one more weak axiom. A fusion system is said to be
saturated if it satisfies two stronger ``saturation'' axioms which were
originally formulated by L. Puig \cite{Puig2006} and reformulated by
Broto, Levi, and Oliver \cite{BrotoLeviOliver2003}. Those axioms hold whenever
$G$ is a finite group, $S$ is a Sylow $p$-subgroup of $G$, and $\Hom_\F(P,Q)
=\Hom_G(P,Q)$ is the set of conjugation maps $c_g$ from $P$ to $Q$ that are
induced by elements $g \in G$. The fusion system of a finite group is denoted
$\F_S(G)$. 

A saturated fusion system $\F$ is said to be \emph{exotic} if it is not of the
form $\F_S(G)$ for any finite group $G$ with Sylow $p$-subgroup $S$. The
Benson-Solomon fusion systems at $p = 2$ form one family of examples of exotic
fusion systems \cite{LeviOliver2002, AschbacherChermak2010}. They are
essentially the only known examples at the prime $2$, and they are in some
sense the oldest known examples, having been studied in the early 1970s by
Solomon in the course of the classification of finite simple groups (although
not with the more recent categorical framework in mind) \cite{Solomon1974}. In
contrast with the case $p = 2$, a fast-growing literature describes many exotic
fusion systems on finite $p$-groups when $p$ is odd.

In replacing a group by its fusion system at a prime, one retains information
about conjugation homomorphisms between $p$-subgroups, but otherwise loses
information about the group elements themselves.  It is therefore natural that
a recurring theme throughout the study of saturated fusion systems is the
question of how to ``enhance'' or ``rigidify'' a saturated fusion system to
make it again more group-like, and also to study which fusion systems have such
rigidifications.  

The study of the existence and uniqueness of centric linking systems was a
first instantiation of this theme of rigidifying saturated fusion systems. A
centric linking system is an important extension category of a fusion system
$\F$ which provides just enough algebraic information to recover a $p$-complete
classifying space. For example, it recovers the homotopy type of the
$p$-completion of $BG$ in the case where $\F = \F_S(G)$. Centric linking
systems of finite groups are easily defined, and Oliver proved that the centric
linking systems of finite groups are unique \cite{Oliver2004, Oliver2006}.
Then, Chermak proved that each saturated fusion system, possibly exotic, has a
unique associated centric linking system \cite{Chermak2013}. A proof which does
not rely on the classification of finite simple groups can be obtained through
\cite{Oliver2013,GlaubermanLynd2016}.

More generally, there are at least two frameworks for considering extensions,
or rigidifications, of saturated fusion systems: the \emph{transporter systems}
of Oliver and Ventura \cite{OliverVentura2007} and the \emph{localities} of
Chermak \cite{Chermak2013}. In particular, one can consider centric linking
systems in either setting. While centric linking systems in either setting have
a specific set of objects, the object sets in transporter systems and
localities can be any conjugation-invariant collection of subgroups which is
closed under passing to overgroups. The categories of transporter systems and
isomorphisms and of localities and isomorphisms are equivalent
\cite[Appendix]{Chermak2013} and \cite[Theorem~2.11]{GlaubermanLynd2021}.
However, depending on the intended application, it is sometimes advantageous to
work in the setting of transporter systems, and sometimes in localities. The
reader is referred to Section~\ref{S:localities} for an introduction to
localities and transporter systems.

In this paper we study \emph{punctured groups}.  These are transporter systems,
or localities, with objects the {\em nonidentity} subgroups of a finite
$p$-group $S$.  To motivate the terminology, recall that every finite group $G$
with Sylow $p$-subgroup $S$ admits a transporter system $\T_S(G)$ whose objects
are {\em all} subgroups of $S$ and $\Mor_\T(P,Q)=N_G(P,Q)$, the
transporter set consisting of all $g \in G$ which conjugate $P$ into $Q$.
Conversely, \cite[Proposition 3.11]{OliverVentura2007} shows that a transporter
system $\T$ whose set of objects consists of {all} the subgroups of $S$ is
necessarily the transporter system $\T_S(G)$ where $G=\Aut_\T(1)$, and the
fusion system $\F$ with which $\T$ is associated is $\F_S(G)$.  Thus, a
punctured group $\T$ is a transporter system whose object set is missing the
trivial subgroup, an object whose inclusion forces $\T$ to be the transporter
system of a finite group.

If we consider localities rather than transporter systems, then the punctured
group of $G$ is the \emph{locality} $\L_{\sS^*(S)}(G) \subseteq G$ consisting
of those elements $g \in G$ which conjugate a nonidentity subgroup of $S$ back
into $S$. This is equipped with the multivariable partial product $w :=
(g_1,\dots,g_n) \mapsto g_1\cdots g_n$, defined only when each initial subword
of the word $w$ conjugates some fixed nonidentity subgroup of $S$ back into
$S$.  Thus, the product is defined on words which correspond to sequences of
composable morphisms in the transporter category $\T^*_S(G)$. See
Definition~\ref{E:LDeltaM} for more details.

By contrast with the existence and uniqueness of linking systems, we will see
that punctured groups for exotic fusion systems do not necessarily exist. The
existence of a punctured group for an exotic fusion system seems to indicate
that the fusion system is ``close to being realizable'' in some sense.
Therefore, considering punctured groups might provide some insight into how
exotic systems arise. 

It is also not reasonable to expect that a punctured group is unique when it
does exist.  To give one example, the fusion systems $PSL_2(q)$ with $q \equiv
9 \pmod{16}$ all have a single class of involutions and equivalent fusion
systems at the prime $2$. On the other hand, the centralizer of an involution
is dihedral of order $2(q-1)$, so the associated punctured groups are distinct
for distinct $q$. Examples like this one occur systematically in groups of Lie
type in nondefining characteristic. Later we will give examples of realizable
fusion systems with punctured groups which do not occur as a full subcategory
of the punctured group of a finite group.

We will now describe our results in detail. To start, we present a result which
gives some motivation for studying punctured groups.

\subsection{Sharpness of the subgroup homology decomposition}
As an application of the existence of the structure of a punctured group for a
saturated fusion system $\F$, we prove that it implies the sharpness of the
subgroup homology decomposition for that system.  Recall from \cite[Definition
1.8]{BrotoLeviOliver2003} that given a $p$-local finite group $(S,\F,\L)$ its
classifying space is the Bousfield-Kan $p$-completion of the geometric
realisation of the category $\L$.  This space is denoted by $\pcomp{|\L|}$.

The orbit category of $\F$, see \cite[Definition 2.1]{BrotoLeviOliver2003}, is
the category $\O(\F)$ with the same objects as $\F$ and whose morphism sets
$\Mor_{\O(\F)}(P,Q)$ is the set of orbits of $\F(P,Q)$ under the action of
$\Inn(Q)$.  The full subcategory of the $\F$-centric subgroups is denoted
$\O(\F^c)$.  For every $j \geq 0$ there is a functor $\H^j \colon \O(\F^c)^\op
\to \modl{\ZZ_{(p)}}$:
\[
\H^j \colon P \mapsto H^j(P;\FF_p), \qquad (P \in \O(\F^c)).
\]
The stable element theorem for $p$-local finite groups \cite[Theorem B, see
also Theorem 5.8]{BrotoLeviOliver2003} asserts that for every $j \geq 0$,
\[
H^j(\pcomp{|\L|};\FF_p) \cong \underset{\O(\F^c)}{\varprojlim} \, \H^j =
\underset{P \in \O(\F^c)}{\varprojlim} \, H^j(P;\FF_p).
\]
The proof of this theorem in \cite{BrotoLeviOliver2003} is indirect and
requires heavy machinery such as Lannes's $T$-functor theory.  From the
conceptual point of view, the stable element theorem is only a shadow of a more
general phenomenon.  By \cite[Proposition 2.2]{BrotoLeviOliver2003} there is a
functor
\[
\tilde{B} \colon \O(\F^c) \to \Top
\]
with the property that $\tilde{B}(P)$ is homotopy equivalent to the classifying
space of $P$ (denoted $BP$) and moreover there is a natural homotopy
equivalence
\[
|\L| \simeq \hhocolim{\O(\F^c)} \tilde{B}.
\]
The Bousfield-Kan spectral sequence for this homotopy colimit \cite[Ch. XII,
Sec. 4.5]{BousfieldKan1972} takes the form
\[
E_2^{i,j} = \underset{\O(\F^c)^\op}{\varprojlim{}^i} \H^j \, \Rightarrow \,
H^{i+j}(\pcomp{|\L|};\FF_p)
\]
and is called the \emph{subgroup decomposition} of $(S,\F,\L)$.  We call the
subgroup decomposition \emph{sharp}, see \cite{Dwyer1998}, if the spectral sequence
collapses to the vertical axis, namely $E_2^{i,j}=0$ for all $i>0$.  When this
is the case, the stable element theorem is a direct consequence.  Indeed,
whenever $\F$ is induced from a finite group $G$ with a Sylow $p$-subgroup $S$,
the subgroup decomposition is sharp (and the stable element theorem goes back
to Cartan-Eilenberg \cite[Theorem XII.10.1]{CartanEilenberg1956}).  This
follows immediately from Dwyer's work \cite[Sec. 1.11]{Dwyer1998} and
\cite[Lemma 1.3]{BrotoLeviOliver2003b}, see for example \cite[Theorem
B]{DiazPark2015}.

It is still an open question as to whether the subgroup decomposition is sharp
for every saturated fusion system.  We will prove the following theorem.

\begin{thm}\label{T:sharpness for punctured groups}
Let $\F$ be a saturated fusion system which affords the structure of a punctured group.
Then the subgroup decomposition on the $\F$-centric subgroups is sharp.
In other words, 
\[
\underset{\O(\F^c)^\op}{\varprojlim{}^i} \H^j =0
\]
for every $i \geq 1$ and  $j \geq 0$.
\end{thm}

We will prove this theorem in Section \ref{Sec:sharpness} below.  We remark
that our methods apply to any functor $\H$ which in the language of
\cite{DiazPark2015} is the pullback of a Mackey functor on the orbit category
of $\F$ denoted $\O(\F)$ such that $\H(e)=0$ where $e \leq S$ is the trivial
subgroup.  In the absence of applications in sight for this level of generality
we have confined ourselves to the functors $\H=\H^j$.

\subsection{Signalizer functor theorem for punctured groups} 
It is natural to ask for which exotic fusion systems punctured groups exist. We
will answer this question for specific families of exotic fusion systems. As a
tool for proving the non-existence of punctured groups we define and study
signalizer functors for punctured groups thus mirroring a concept from finite
group theory.

\begin{defn}\label{D:SignalizerFunctor}
Let $(\L,\Delta,S)$ be a punctured group. If $P$ is a subgroup of $S$, write
$\I_p(P)$ for the set of elements of $P$ of order $p$. A \emph{signalizer
functor of $(\L,\Delta,S)$ on elements of order $p$} is a map $\theta$ from
$\I_p(S)$ to the set of subgroups of $\L$, which associates to each element
$a\in \I_p(S)$ a normal $p^\prime$-subgroup $\theta(a)$ of $C_\L(a)$ such that
the following two conditions hold:
\begin{itemize}
\item (Conjugacy condition) $\theta(a^g)=\theta(a)^g$ for any $g\in\L$ and
$a\in\I_p(S)$ such that $a^g$ is defined and an element of $S$.
\item (Balance condition) $\theta(a)\cap C_\L(b)\leq \theta(b)$ for all
$a,b\in\I_p(S)$ with $[a,b]=1$.
\end{itemize}
\end{defn}

Notice in the above definition that, since $(\L,\Delta,S)$ is a punctured
group, for any $a\in S$, the normalizer $N_\L(\gen{a})$ and thus also the
centralizer $C_\L(a)$ is a subgroup.

\begin{thm}[Signalizer functor theorem for punctured groups] \label{T:mainSignalizerFunctor}
Let $(\L,\Delta,S)$ be a punctured group and suppose $\theta$ is a signalizer
functor of $(\L,\Delta,S)$ on elements of order $p$. Then
\[
\widehat{\Theta}:=\bigcup_{x\in\I_p(S)}\theta(x)
\]
is a partial normal subgroup of $\L$ with $\widehat{\Theta}\cap S=1$. In
particular, the canonical projection $\rho\colon \L\rightarrow
\L/\widehat{\Theta}$ restricts to an isomorphism $S\rightarrow S^\rho$. Upon
identifying $S$ with $S^\rho$, the following properties hold:
\begin{itemize}
\item [(a)] $(\L/\widehat{\Theta},\Delta,S)$ is a locality and
$\F_S(\L/\widehat{\Theta}) = \F_S(\L)$.
\item [(b)] For each $P\in\Delta$, the projection $\rho$ restricts to an
epimorphism $N_\L(P)\rightarrow N_{\L/\widehat{\Theta}}(P)$ with kernel
$\Theta(P)$ and thus induces an isomorphism $N_\L(P)/\Theta(P)\cong
N_{\L/\widehat{\Theta}}(P)$.
\end{itemize}
\end{thm}

\subsection{Punctured groups for families of exotic fusion systems}

Let $\F$ be a saturated fusion system on the $p$-group $S$. If $\L$ is a
locality or transporter system associated with $\F$, then for each fully
$\F$-normalized object $P$ of $\L$, the normalizer fusion system $N_\F(P)$ is
the fusion system of the group $N_\L(P)$ if $\L$ is a locality, and of the
group $\Aut_\L(P)$ if $\L$ is a transporter system. This gives an easy
necessary condition for the existence of a punctured group: for each fully
$\F$-normalized nonidentity subgroup $P \leq S$, the normalizer $N_\F(P)$ is
realizable.

Conversely, there is a sufficient condition for the existence of a punctured
group: $\F$ is of characteristic $p$-type, i.e. for each fully $\F$-normalized
nonidentity subgroup $P \leq S$, the normalizer $N_\F(P)$ is constrained. This
follows from the existence of linking systems (or similarly linking localities)
of a very  general kind, a result which was shown in
\cite[Theorem~A]{Henke2019} building on the existence and uniqueness of centric
linking systems. 

The Benson-Solomon fusion systems $\F_{\Sol}(q)$ at the prime $2$ have the
property that the normalizer fusion system of each nonidentity subgroup $P$ is
realizable, and moreover, $C_\F(Z(S))$ is the fusion system at $p = 2$ of
$\Spin_7(q)$, and hence not constrained. So $\F_{\Sol}(q)$ satisfies the
obvious necessary condition for the existence of a punctured group, and does
not satisfy the sufficient one.  

Based on results of Solomon \cite{Solomon1974}, Levi and Oliver showed that
$\F_{\Sol}(q)$ is exotic \cite[Theorem~3.4]{LeviOliver2002}, i.e., it has no
locality with objects \emph{all} subgroups of a Sylow $2$-group.  In
Section~\ref{S:sol}, we show the following theorem.

\begin{thm}\label{T:sol} For any odd prime power $q$, the Benson-Solomon fusion
system $\F_{\Sol}(q)$ has a punctured group if and only if $q \equiv \pm 3
\pmod{8}$.  \end{thm}

If $l$ is the nonnegative integer with the property that $2^{l+3}$ is the
$2$-part of $q^2-1$, then $\F_{\Sol}(q) \cong \F_{\Sol}(3^{2^l})$. So the
theorem says that only the smallest Benson-Solomon system, $\F_{\Sol}(3)$, has
a punctured group, and the larger ones do not. Further details and a uniqueness
statement are given in Theorem~\ref{T:solq}. 

When showing the non-existence of a punctured group in the case $q \equiv \pm 1
\pmod{8}$, the Signalizer Functor Theorem~\ref{T:mainSignalizerFunctor} plays
an important role in showing that a putative minimal punctured group has no
nontrivial partial normal $p'$-subgroups. This is similar to the way
signalizer functor theory was used by Solomon in \cite[Section~3]{Solomon1974}.
To construct a punctured group in the case $q \equiv \pm 3 \pmod{8}$, we turn
to a procedure we call Chermak descent. It is an important tool in Chermak's
proof of the existence and uniqueness of centric linking systems
\cite[Section~5]{Chermak2013} and allows us (under some assumptions) to
``expand'' a given locality to  produce a new locality with a larger object
set. Starting with a linking locality, we use Chermak descent to construct a
punctured group $\L$ for $\F_{\Sol}(q)$ in which the centralizer of an
involution is $C_{\L}(Z(S)) \cong \Spin_7(3)$. 

It is possible that there could be other examples of punctured groups for
$\F_{\Sol}(3)$ in which the centralizer of an involution is $\Spin_7(q)$ for
certain $q = 3^{1+6a}$; a necessary condition for existence is that each prime
divisor of $q^2-1$ is a square modulo $7$.  However, given this condition, we
can neither prove or disprove the existence of an example with the prescribed
involution centralizer. 

In Section~\ref{S:pgexotic}, we survey a few families of known exotic fusion
systems at odd primes to determine whether or not they have a punctured group.
A summary of the findings is contained in Theorem~\ref{T:pgexotic}. For nearly
all the exotic systems we consider, either the system is of characteristic
$p$-type, or the normalizer of some $p$-subgroup is exotic and therefore
a punctured group does not exist. Indeed, it might be that a similar result can
be shown for all known exotic fusion systems at odd primes. At least we are not
aware of any counterexample.

In particular, when considering the family of Clelland-Parker systems
\cite{ClellandParker2010} in which each essential subgroup is special, we find
that $O^{p'}(C_\F(Z(S))/Z(S))$ is simple, exotic, and had not appeared
elsewhere in the literature as of the time of our writing.  We dedicate part of
Subsection~\ref{SS:CP} to describing these systems and to proving that they are
exotic.

Applying Theorem~\ref{T:sharpness for punctured groups} to the results of
Sections~\ref{S:sol} and \ref{S:pgexotic} establish the sharpness of the
subgroup decomposition for new families of exotic fusion systems, notably
\begin{itemize} \item Benson-Solomon's system $\F_{\Sol}(3)$
\cite{LeviOliver2002}, \item all Parker-Stroth systems \cite{ParkerStroth2015},
\item all Clelland-Parker systems \cite{ClellandParker2010} in which each
essential subgroup is abelian.  \end{itemize} It also recovers the sharpness
for certain fusion systems on $p$-groups with an abelian subgroup of index $p$,
a result that was originally established in full generality by D\i az and Park
\cite{DiazPark2015}.

\subsection{Classification of punctured groups over $p^{1+2}_+$}

In general, it seems difficult to classify all the punctured groups associated
with a given saturated fusion system. However, for fusion systems over an
extraspecial $p$-group of exponent $p$, which by \cite{RuizViruel2004} are
known to contain among them three exotic fusion systems at the prime $7$, we
are able to work out such an example. There is always a punctured group $\L$
associated to such a fusion system, and when $\F$ has one class of subgroups of
order $p$ and the full subcategory of $\L$ with objects the $\F$-centric
subgroups is the centric linking system, a classification is obtained in
Theorem~\ref{T:esclass}. Conversely, the cases we list in that theorem all
occur in an example for a punctured group. This demonstrates on the one hand
that there can be more than one punctured group associated to the same fusion
system and indicates on the other hand that examples for punctured groups are
still somewhat limited.

\subsection*{Outline of the paper and notation}

\smallskip
The paper proceeds as follows. In Section~\ref{S:localities} we recall the
definitions and basic properties of transporter systems and localities, and we
prove the Signalizer Functor Theorem in Subsection~\ref{SS:SignalizerFunctor}.
In Section~\ref{Sec:sharpness}, we prove sharpness of the subgroup
decomposition for fusion systems with associated punctured groups.
Section~\ref{S:sol} examines punctured groups for the Benson-Solomon fusion
systems, while Section~\ref{S:pgexotic} looks at several families of exotic
fusion systems at odd primes. Finally, in Section~\ref{S:RV} classifies certain
punctured groups over an extraspecial $p$-group of order $p^3$ and exponent
$p$. An Appendix~\ref{S:lie} sets notation and provides certain general results
on finite groups of Lie type that are needed in Section~\ref{S:sol}. 

The first four sections of the paper do not use the classification of the
finite simple groups (CFSG). The CFSG is always used indirectly in Sections 5
and 6 whenever we need to apply known results that certain exotic fusion
systems at odd primes are indeed exotic.  Each time this occurs (e.g.,
Proposition~\ref{P:Oliver}), the results could be stated so as to avoid
indirect use of the CFSG.  Aside from this, there are two direct applications
of the CFSG. The first occurs in the proof of Lemma~\ref{L:ClParkerExotic}(c)
when showing that fusion systems related to the Clelland-Parker systems are
exotic. The second occurs in the proof of Lemma~\ref{OneComponent}(b).

Throughout most of the paper we write conjugation like maps on the right side
of the argument and compose from left to right. There are two exceptions: when
working with transporter systems, such as in Section~\ref{Sec:sharpness}, we
compose morphisms from right to left.  Also, we apply certain maps in
Section~\ref{S:sol} on the left of their arguments (e.g. roots, when viewed as
characters of a torus). The notation for Section~\ref{S:sol} is outlined in
more detail in the appendix.

\subsection*{Acknowledgements} 
It was Andy Chermak who first asked the question in 2011 (arising out of his
proof of existence and uniqueness of linking systems) of which exotic systems
have localities on the set of nonidentity subgroups of a Sylow group.  We thank
him for comments on an earlier version of this paper and many helpful
conversations, including during a visit to Rutgers in 2014, where he and the
third author discussed the possibility of constructing punctured groups for the
Benson-Solomon systems. We are grateful to George Glauberman for communicating
to us several comments, corrections, and suggestions for improvement. We
are especially grateful to the referee for pointing out that Lemma~\ref{L:EE'F}
conflicts with a lemma of Levi and Oliver, for alerting us to errors too
numerous to mention, and for other suggestions.  We would like to thank the
Centre for Symmetry and Deformation at the University of Copenhagen for
supporting a visit of the third named author, where some of the early work on
this paper took place. Finally, the authors thank the Isaac Newton Institute
for Mathematical Sciences, Cambridge, for support and hospitality during the
programme ``Groups, representations and applications'', where work on this
paper was undertaken and supported by EPSRC grant no EP/R014604/1.

\section{Localities and transporter systems}\label{S:localities}

As already mentioned in the introduction, transporter systems as defined by
Oliver and Ventura \cite{OliverVentura2007} and localities in the sense of
Chermak \cite{Chermak2013} are algebraic structures which carry essentially the
same information. In this section, we will give an introduction to both
subjects and outline briefly the connection between localities and transporter
systems. At the end we present some signalizer functor theorems for localities.

\subsection{Partial groups}\label{S:partialgroups}

\renewcommand{\D}{\mathbf{D}}

We refer the reader to Chermak's papers \cite{Chermak2013} or \cite{Chermak2022}
for a detailed introduction to partial groups and localities. However, we will
briefly summarize the most important definitions and results here. Following
Chermak's notation, we write $\W(\L)$ for the set of words in a set $\L$, and
$\emptyset$ for the empty word. The concatenation of words
$u_1,\dots,u_k\in\W(\L)$ is denoted by $u_1\circ u_2\circ \cdots \circ u_k$.

\begin{defn}[Partial Group]\label{partial} Let $\L$ be a non-empty set, let
$\D$ be a subset of $\W(\L)$, let $\Pi \colon \D \rightarrow \L$ be a map and
let $(-)^{-1} \colon \L \rightarrow \L$ be an involutory bijection, which we
extend to a map \[ (-)^{-1} \colon \W(\L) \rightarrow \W(\L), w = (g_1, \dots,
g_k) \mapsto w^{-1} = (g_k^{-1}, \dots, g_1^{-1}).  \] We say that $\L$ is a
\emph{partial group} with product $\Pi$ and inversion $(-)^{-1}$ if the
following hold:
\begin{itemize}
\item  $\L \subseteq \D$  (i.e. $\D$ contains
all words of length 1), and \[  u \circ v \in \D \Longrightarrow u,v \in \D;\] (so
in particular, $\emptyset\in\D$.)
\item $\Pi$ restricts to the identity map on
$\L$;
\item $u \circ v \circ w \in \D \Longrightarrow u \circ (\Pi(v)) \circ w \in
\D$, and $\Pi(u \circ v \circ w) = \Pi(u \circ (\Pi(v)) \circ w)$;
\item $w \in
\D \Longrightarrow  w^{-1} \circ w\in \D$ and $\Pi(w^{-1} \circ w) = \One$ where
$\One:=\Pi(\emptyset)$.
\end{itemize}
\end{defn}

Note that any group $G$ can be regarded as a partial group with product defined
in $\D=\W(G)$ by extending the ``binary'' product to a map $\W(G)\rightarrow
G,(g_1,g_2,\dots,g_n)\mapsto g_1g_2\cdots g_n$.

\smallskip

\textbf{For the remainder of this section let $\L$ be a partial group with
product $\Pi\colon \D\rightarrow \L$ defined on the domain
$\D\subseteq\W(\L)$.}

Because of the group-like structure of partial groups, the product
$\mathcal{X}\mathcal{Y}$ of two subsets $\mathcal{X}$ and $\mathcal{Y}$ of $\L$
is naturally defined by \[\mathcal{X}\mathcal{Y}:=\{\Pi(x,y)\colon
x\in\mathcal{X},\;y\in\mathcal{Y}\mbox{ such that }(x,y)\in\D\}.\] Similarly,
there is a natural notion of conjugation, which we consider next.

\begin{defn} For every $g\in \L$ we define \[ \D(g) = \{ x\in \L \mid (g^{-1},
x, g) \in \D\}.  \] The map $c_g \colon \D(g) \rightarrow \L$, $x \mapsto x^g =
\Pi(g^{-1}, x, g)$ is the \emph{conjugation map} by $g$. If $\H$ is a subset of
$\L$ and $\H \subseteq \D(g)$, then we set \[ \H^g = \{ h^g \mid h \in \H\}.
\] Whenever we write $x^g$ (or $\H^g$), we mean implicitly that $x\in\D(g)$ (or
$\H\subseteq \D(g)$, respectively). Moreover, if $\M$ and $\H$ are subsets of
$\L$, we write $N_\M(\H)$ for the set of all $g\in\M$ such that
$\H\subseteq\D(g)$ and $\H^g=\H$. Similarly, we write $ C_\M(\H)$ for the set
of all $g\in\M$ such that $\H\subseteq\D(g)$ and $h^g=h$ for all $h\in\H$. If
$\M\subseteq\L$ and $h\in\L$, set $C_\M(h):=C_\M(\{h\})$. 
\end{defn}

\begin{defn} Let $\H$ be a non-empty subset of $\L$. The subset $\H$ is a
\emph{partial subgroup} of $\L$ if \begin{itemize}
 \item $g\in \H
\Longrightarrow g^{-1} \in \H$; and \item $w \in \D \cap \W(\H) \Longrightarrow
\Pi(w) \in \H$.  \end{itemize}
If $\H$ is a partial subgroup of $\L$ with
$\W(\H)\subseteq\D$, then $\H$ is called a \emph{subgroup} of $\L$. 

\smallskip

A partial subgroup $\N$ of $\L$ is called a \emph{partial normal subgroup} of
$\L$ (denoted $\N \unlhd \L$) if for all $g\in \L$ and $n\in\N$, \[ n \in \D(g)
\Longrightarrow n^g \in \N.\] \end{defn}

We remark that a subgroup $\H$ of $\L$ is always a group in the usual sense
with the group multiplication defined by $hg=\Pi(h,g)$ for all $h,g\in\H$. 

\subsection{Localities}

Roughly speaking, localities are partial groups with some some extra structure,
in particular with a ``Sylow $p$-subgroup'' and a set $\Delta$ of ``objects''
which in a certain sense determines the domain of the product. This is made more
precise in Definition~\ref{LocalityDefinition} below. We continue to assume
that $\L$ is a partial group with product $\Pi\colon \D\rightarrow \L$. We will
use the following notation.

\begin{notation} If $S$ is a subset of $\L$ and $g\in\L$, set \[ S_g:=\{s\in
S\cap \D(g)\colon s^g\in S\}.  \] More generally, if
$w=(g_1,\dots,g_n)\in\W(\L)$ with $n\geq 1$, define $S_w$ to be the set of
elements $s\in S$ for which there exists a sequence of elements
$s=s_0,s_1,\dots,s_n\in S$ with $s_{i-1}\in\D(g_i)$ and $s_{i-1}^{g_i}=s_i$ for
all $i=1,\dots,n$.  \end{notation}

\begin{defn}\label{LocalityDefinition} We say that $(\L,\Delta,S)$ is a
\textit{locality} if the partial group $\L$ is finite as a set, $S$ is a
$p$-subgroup of $\L$, $\Delta$ is a non-empty set of subgroups of $S$, and the
following conditions hold: \begin{itemize} \item[(L1)] $S$ is maximal with
respect to inclusion among the $p$-subgroups of $\L$.  \item[(L2)] For any word
$w=(f_1,\dots,f_n)\in\W(\L)$, we have $w\in \D$ if and only if there exist
$P_0,\dots,P_n\in\Delta$ with \begin{itemize} \item [($*$)] $P_{i-1}\subseteq
\D(f_i)$ and $P_{i-1}^{f_i}=P_i$ for all $i=1,\dots,n$.  \end{itemize}
\item[(L3)] The set $\Delta$ is closed under passing to $\L$-conjugates and
overgroups in $S$, i.e. $\Delta$ is overgroup-closed in $S$ and, for every
$P\in\Delta$ and $g\in\L$ such that $P\subseteq S_g$, we have $P^g\in\Delta$.
\end{itemize} If $(\L,\Delta,S)$ is a locality, $w=(f_1,\dots,f_n)\in\W(\L)$,
and $P_0,\dots,P_n$ are elements of $\Delta$ such that ($*$) holds, then we say
that $w\in\D$ via $P_0,\dots,P_n$ (or $w\in\D$ via $P_0$).  \end{defn}

It is argued in \cite[Remark~5.2]{Henke2019} that
Definition~\ref{LocalityDefinition} is equivalent to the definition of a
locality given by Chermak \cite[Definition 2.7]{Chermak2022} (which is
essentially the same as the one given in \cite[Definition~2.9]{Chermak2013}).

\begin{example}\label{E:LDeltaM} 
Let $M$ be a finite group and $S\in\Syl_p(M)$.  Set $\F=\F_S(M)$ and let
$\Delta$ be a non-empty collection of subgroups of $S$, which is closed under
$\F$-conjugacy and overgroup-closed in $S$. Set 
\[
\L_\Delta(M):=\{g\in G\colon S\cap S^g\in \Delta\}=\{g\in G\colon \exists
P\in\Delta\mbox{ with }P^g\leq S\} 
\] 
and let $\D$ be the set of tuples $(g_1,\dots,g_n)\in\W(M)$ such that there
exist $P_0,P_1,\dots,P_n\in\Delta$ with $P_{i-1}^{g_i}=P_i$. Then
$\L_\Delta(M)$ forms a partial group whose product is the restriction of the
multivariable product on $M$ to $\D$, and whose inversion map is the
restriction of the inversion map on the group $M$ to $\L_\Delta(M)$. Moreover,
$(\L_\Delta(M),\Delta,S)$ forms a locality.
\end{example}

In the next lemma we summarize the most important properties of localities
which we will use throughout, most of the time without reference. 

\begin{lem}[Important properties of localities]\label{L:LocalitiesProp} 
Let $(\L,\Delta,S)$ be a locality. Then the following hold: 
\begin{itemize}
\item[(a)] $N_\L(P)$ is a subgroup of $\L$ for each $P\in\Delta$.  
\item[(b)] Let $P\in\Delta$ and $g\in\L$ with $P\subseteq S_g$. Then $Q:=P^g\in\Delta$,
$N_\L(P)\subseteq \D(g)$ and \[c_g\colon N_\L(P)\rightarrow N_\L(Q),x\mapsto
x^g\] is an isomorphism of groups.  
\item[(c)] Let $w=(g_1,\dots,g_n)\in\D$ via $(X_0,\dots,X_n)$. Then
$$c_{g_1}\circ \dots \circ c_{g_n}=c_{\Pi(w)}$$ is a group isomorphism
$N_\L(X_0)\rightarrow N_\L(X_n)$.  
\item[(d)] For every $g\in\L$, we have $S_g\in\Delta$. In particular, $S_g$ is
a subgroup of $S$.  Moreover, $S_g^g=S_{g^{-1}}$ and $c_g\colon S_g\rightarrow
S,x\mapsto x^g$ is an injective group homomorphism.  
\item[(e)] For every $g\in\L$, $c_g\colon \D(g)\rightarrow \D(g^{-1}),x\mapsto
x^g$ is a bijection with inverse map $c_{g^{-1}}$.  
\item[(f)] For any $w\in\W(\L)$, $S_w$ is a subgroup of $S$ with
$S_w\in\Delta$ if and only if $w\in\D$. Moreover, $w\in\D$ implies $S_w\leq
S_{\Pi(w)}$.
\end{itemize} \end{lem}

\begin{proof} 
Properties (a),(b) and (c) correspond to the statements (a),(b) and (c) in
\cite[Lemma~2.3]{Chermak2022} except for the fact stated in (b) that
$Q\in\Delta$, which is however clearly true if one uses the definition of a
locality given above. Property (d) holds by
\cite[Proposition~2.5(a),(b)]{Chermak2022} and property (e) is stated in
\cite[Lemma~2.5(c)]{Chermak2013}.  Property (f) corresponds to
Corollary~2.6 in \cite{Chermak2022}.
\end{proof}

Let $(\L,\Delta,S)$ be a locality. Then it follows from
Lemma~\ref{L:LocalitiesProp}(d) that, for every $P\in\Delta$ and every $g\in\L$
with $P\subseteq S_g$, the map $c_g\colon P\rightarrow P^g,x\mapsto x^g$ is an
injective group homomorphism. The fusion system $\F_S(\L)$ is the fusion system
over $S$ generated by such conjugation maps. Equivalently, $\F_S(\L)$ is
generated by the conjugation maps between subgroups of $S$, or by the
conjugation maps of the form $c_g\colon S_g\rightarrow S,x\mapsto x^g$ with
$g\in\L$.  

\begin{defn}\label{D:localityoverF}
If $\F$ is a fusion system, then we say that the locality $(\L,\Delta,S)$ is a
locality over $\F$ if $\F=\F_S(\L)$.  
\end{defn}

If $(\L,\Delta,S)$ is a locality over $\F$, then notice that the set $\Delta$
is always overgroup-closed in $S$ and closed under $\F$-conjugacy. 
Definition~\ref{D:localityoverF} says precisely that every morphism in $\F$ is
a composite of conjugation maps. It is however not true in general that every
$\F$-morphism is itself a conjugation map. In the following lemma, the
assumption that $P$ is an object in $\Delta$ (rather than just a subgroup of
$S$) is therefore important.

\begin{lem}\label{L:NFPNLP}
Let $(\L,\Delta,S)$ be a locality over a fusion system $\F$ and $P\in\Delta$.
Then the following hold: 
\begin{itemize} 
\item [(a)] For every $\phi\in\Hom_\F(P,S)$, there exists $g\in\L$ such that
$P\leq S_g$ and $\phi(x)=x^g$ for all $x\in P$.  
\item [(b)] $N_\F(P)=\F_{N_S(P)}(N_\L(P))$.  
\end{itemize} 
\end{lem}
\begin{proof} 
For (a) see Lemma~5.6 in \cite{Henke2019}.  As $\F=\F_S(\L)$, one sees
that $\F_{N_S(P)}(N_\L(P))$ is a subsystem of $N_\F(P)$. Conversely, by
definition of the normalizer system, each morphism in $N_\F(P)$ extends to a
morphism with source a subgroup of $N_S(P)$ containing $P$, so $N_\F(P)$ is
generated as a fusion system by morphisms in $\F$ between objects of $\L$.
Part (a) then gives equality.
\end{proof}

Suppose $(\L^+,\Delta^+,S)$ is a locality with partial product
$\Pi^+\colon\D^+\rightarrow\L^+$. If $\Delta$ is a non-empty subset of
$\Delta^+$ which is closed  under taking $\L^+$-conjugates and overgroups in
$S$, we set \[\L^+|_\Delta:=\{f\in\L^+\colon \exists P\in\Delta\mbox{ such that
}P\subseteq \D^+(f)\mbox{ and }P^f\leq S\} \] and write $\D$ for the set of
words $w=(f_1,\dots,f_n)$ such that $w\in\D^+$ via $P_0,\dots,P_n$ for some
$P_0,\dots,P_n\in\Delta$. Note that $\D$ is a set of words in $\L^+|_\Delta$
which is contained in $\D^+$. It is easy to check that $\L^+|_\Delta$ forms a
partial group with partial product $\Pi^+|_{\D}\colon\D\rightarrow
\L^+|_{\Delta}$, and that $(\L^+|_\Delta,\Delta,S)$ forms a locality; see
\cite[Lemma~2.21]{Chermak2013} for details. We call $\L^+|_\Delta$ the
\textit{restriction} of $\L^+$ to $\Delta$.

\subsection{Projections of localities}

Throughout this subsection let $\L$ and $\L'$ be partial groups with products
$\Pi\colon\D\rightarrow \L$ and $\Pi'\colon\D'\rightarrow\L'$ respectively.

\begin{defn}\label{D:PartialHom} Let
$\beta \colon \L \rightarrow \L',g\mapsto g^\beta$ be a map. By abuse of
notation, we denote by $\beta$ also the induced map on words \[ \W(\L)
\rightarrow \W(\L'),\quad w=(f_1,\dots,f_n)\mapsto
w^\beta=(f_1^\beta,\dots,f_n^\beta) \] and set $\D^\beta=\{w^\beta\colon
w\in\D\}$. We say that $\beta$ is a homomorphism of partial groups if
\begin{enumerate} \item $\D^{\beta} \subseteq \D'$; and \item $\Pi(w)^\beta =
\Pi'(w^\beta)$ for every $w \in \D$.  \end{enumerate} If moreover $\D^\beta =
\D'$ (and thus $\beta$ is in particular surjective), then we say that $\beta$
is a \emph{projection} of partial groups. If $\beta$ is a bijective projection
of partial groups, then $\beta$ is called an \emph{isomorphism}.  \end{defn}

There is no accepted notion of a morphism of localities (at the same prime $p$)
in the literature. One could however form a category of localities with
morphisms the partial group homomorphisms -- or alternatively there are several
full subcategories of this which one might want to consider. For our purposes
it will be enough to consider the category of localities with projections of
localities as defined next.

\begin{defn} Let $(\L, \Delta, S)$ and $(\L', \Delta', S')$ be localities and
let $\beta \colon \L \rightarrow \L'$ be a projection of partial groups.  We
say that $\beta$ is a \emph{projection of localities} from $(\L,\Delta,S)$ to
$(\L',\Delta',S')$ if, setting $\Delta^\beta = \{P^\beta \mid P \in \Delta\}$,
we have $\Delta^\beta=\Delta'$ (and thus $S^\beta=S'$). 

If $\beta$ is in addition bijective, then $\beta$ is a called an
\emph{isomorphism} of localities. If $S=S'$, then an isomorphism of localities
from $(\L,\Delta,S)$ to $(\L',\Delta',S)$ is called a \emph{rigid isomorphism}
if it restricts to the identity on $S$.  \end{defn}

The notion of a rigid isomorphism will be important later on when talking about
the uniqueness of certain localities attached to a given fusion system. 

We will now describe some naturally occurring projections of localities.
Suppose $(\L,\Delta,S)$ is a locality and $\N$ is a partial normal subgroup of
$\L$. A \emph{coset} of $\N$ in $\L$ is a subset of the form \[ \N
f:=\{\Pi(n,f)\colon n\in\N\mbox{ such that }(n,f)\in\D\} \] for some $f\in\L$.
Unlike in groups, the set of cosets does not form a partition of $\L$ in
general. Instead, one needs to focus on the \emph{maximal cosets}, i.e. the
elements of the set of cosets of $\N$ in $\L$ which are maximal with respect to
inclusion. By \cite[Lemma~3.15]{Chermak2022}, the set $\L/\N$ of maximal
cosets of $\N$ in $\L$ forms a partition of $\L$. Thus, there is a natural map
\[ \beta\colon \L\rightarrow \L/\N \] sending each element $g\in\L$ to the
unique maximal coset of $\N$ in $\L$ containing $g$.  Set $\ov{\L}:=\L/\N$ and
$\ov{\D}:=\D^\beta:=\{w^\beta\colon w\in\D\}$.  By
\cite[Lemma~3.16]{Chermak2022}, there is a unique map $\ov{\Pi}\colon
\ov{\D}\rightarrow\ov{\L}$ and a unique involutory bijection
$\ov{\L}\rightarrow\ov{\L},\ov{f}\mapsto\ov{f}^{-1}$  such that $\ov{\L}$ with
these structures is a partial group, and such that $\beta$ is a projection of
partial groups. Moreover, setting $\ov{S}:=S^\beta$ and
$\ov{\Delta}:=\{P^\beta\colon P\in\Delta\}$, the triple
$(\ov{\L},\ov{\Delta},\ov{S})$ is by \cite[Corollary~4.5]{Chermak2022} a
locality, and $\beta$ is a projection from $(\L,\Delta,S)$ to
$(\ov{\L},\ov{\Delta},\ov{S})$. The map $\beta$ is called the \emph{natural
projection} from $\L\rightarrow \ov{\L}$. 

\smallskip

The notation used above suggests already that we will use a ``bar notation''
similar to the one commonly used in finite groups. Namely, if we set
$\ov{\L}:=\L/\N$, then for every subset or element $P$ of $\L$, we will denote
by $\ov{P}$ the image of $P$ under the natural projection
$\beta\colon\L\rightarrow \ov{\L}$. We conclude this section with a little
lemma needed later on. 

\begin{lem}\label{L:PreimageovS} Let $(\L,\Delta,S)$ be a locality with partial
normal subgroup $\N$. Setting $\ov{\L}:=\L/\N$, the preimage of $\ov{S}$ under
the natural projection equals $\N S$.  \end{lem}

\begin{proof} 
For every $s \in S$, the coset $\N s$ is maximal by Lemma~3.7(a) and
Proposition~3.14(c) in \cite{Chermak2022}. Thus, for every $s\in S$, we have
$\ov{s}=\N s$. Hence, the preimage of $\ov{S}=\{\ov{s}\colon s\in S\}$ equals
$\bigcup_{s\in S}\N s=\N S$.  
\end{proof}

\subsection{Transporter systems}

Throughout this section, fix a finite $p$-group $S$, a fusion system $\F$ over
$S$, and a collection $\Delta$ of nonidentity subgroups of $S$ which is
overgroup-closed in $S$ and closed under $\F$-conjugacy. As the literature
about transporter systems is written in left-hand notation, in this section, we
will also \emph{write our maps on the left hand side of the argument}.
Accordingly we will conjugate from the left.

The \emph{transporter category} $\T_S(G)$ (at the prime $p$) of a finite group
$G$ with Sylow $p$-subgroup $S$ is the category with objects the nonidentity
subgroups of $S$ and with morphisms given by the transporter sets $N_G(P,Q) =
\{g \in G \mid {}^gP \leq Q\}$. More precisely, the morphisms in $\T_S(G)$
between $P$ and $Q$ are the triples $(g,P,Q)$ with $g\in N_G(P,Q)$.  We also
write $\T_\Delta(G)$ for the full subcategory of $\T_S(G)$ with objects in
$\Delta$. 

\smallskip

Since we conjugate in this section from the left, for $P,Q\leq S$ and $g\in
N_G(P,Q)$, we write $c_g$ for the conjugation map from $P$ to $Q$ given by
\emph{left conjugation}, i.e.  
\[ 
c_g\colon P\rightarrow Q,\;x\mapsto {}^gx.
\]

\begin{defn}{{(\!\!\cite[Definition~3.1]{OliverVentura2007})}}\label{D:transportersystem}
A \emph{transporter system} associated to $\F$ is a nonempty finite category
$\T$ having object set $\Delta \subseteq \mathrm{Ob}(\F)$, together with functors 
\[ 
\T_\Delta(S) \xrightarrow{\quad \epsilon \quad} \T \xrightarrow{\quad \rho \quad} \F 
\]
which satisfy the following axioms.  
\begin{itemize} 
\item[(A1)] $\Delta$ is closed under $\F$-conjugacy and passing to
overgroups, $\epsilon$ is the identity on objects, and $\rho$ is the inclusion
on objects; 
\item[(A2)] For each $P, Q \in \Delta$, the kernel \[ E(P) := \ker(
\rho_{P,P}\colon \Aut_\T(P) \longrightarrow \Aut_\F(P) ) \] acts freely on
$\Mor_\T(P,Q)$ by right composition, and $\rho_{P,Q}$ is the orbit map for this
action. Also, $E(Q)$ acts freely on $\Mor_\T(P,Q)$ by left composition.
\item[(B)] For each $P, Q \in \Delta$, $\epsilon_{P,Q} \colon N_S(P,Q) \to
\Mor_\T(P,Q)$ is injective, and the composite $\rho_{P,Q} \circ \epsilon_{P,Q}$
sends $s \in N_S(P,Q)$ to $c_s \in \Hom_\F(P,Q)$.  
\item[(C)] For all $\phi \in \Mor_\T(P,Q)$ and all $g \in P$, the diagram 
\[
\xymatrix{ P \ar[d]_{\epsilon_{P,P}(g)} \ar[r]^{\phi} & Q
\ar[d]^{\epsilon_{Q,Q}(\rho(\phi)(g))}\\ P \ar[r]_{\phi} & Q } 
\] 
commutes in $\T$.  
\item[(I)] $\epsilon_{S,S}(S)$ is a Sylow $p$-subgroup of $\Aut_\T(S)$.
\item[(II)] Let $\phi \in \Iso_{\T}(P,Q)$, let $P \norm \bar{P} \leq S$, and
let $Q \norm \bar{Q} \leq S$ be such that $\phi \circ \epsilon_{P,P}(\bar{P})
\circ \phi^{-1} \leq \epsilon_{Q,Q}(\bar{Q})$. Then there exists $\bar{\phi}
\in \Mor_{\T}(\bar{P}, \bar{Q})$ such that $\bar{\phi} \circ
\epsilon_{P,\bar{P}}(1) = \epsilon_{Q,\bar{Q}}(1) \circ \phi$.  \end{itemize}
If we want to be more precise, we say that $(\T,\epsilon,\rho)$ is a
transporter system.
\end{defn}

Note that, by \cite[Lemmas~3.2(b) and 3.8]{OliverVentura2007}, every
morphism in a transporter system is both a monomorphism and an epimorphism.

A centric linking system in the sense of \cite{BrotoLeviOliver2003} is a
transporter system in which $\Delta$ is the set of $\F$-centric subgroups and
$E(P)$ is precisely the center $Z(P)$ viewed as a subgroup of $N_S(P)$ via the
map $\epsilon_{P,P}$.  A more general notion of linking system will be
introduced in Subsection~\ref{SS:linkingSystem}.

We next want to state the definition of an isomorphism of transporter systems
as used in \cite{GlaubermanLynd2021}. First we prove a lemma that helps to
explain that a property implicitly assumed there, and which is needed for
the definition to make sense, does in fact hold.

Given a fusion system $\F$ on a finite $p$-group $S$ and a set $\Delta$
of subgroups of $S$, we write $\F|_\Delta$ for the full subcategory of $\F$
with object set $\Delta$.

\begin{lem}\label{L:isotrans}
Let $(\T,\epsilon,\rho)$ and $(\T',\epsilon',\rho')$ be two transporter systems
having object sets $\Delta$ and $\Delta'$, associated with the fusion
systems $\F$ and $\F'$ over the $p$-groups $S$ and $S'$, respectively. Let
$\alpha \colon \T \to \T'$ be any equivalence of categories. 
Then the following hold.
\begin{itemize}
\item[(a)] $\alpha(S) = S'$, and $\alpha(P) \leq \alpha(S)$ for each $P \in \Delta$.
\item[(b)] Suppose $\alpha$ has the following two additional properties: 
\begin{itemize}
\item[(typ)] $\alpha_{P,P}(\epsilon_{P,P}(P)) =
\epsilon'_{\alpha(P),\alpha(P)}(\alpha(P))$ for each $P \in \Delta$, and
\item[(inc)] $\alpha_{P,S}(\epsilon_{P,S}(1)) = \epsilon'_{\alpha(P),\alpha(S)}(1)$ for each
$P \in \Delta$,
\end{itemize}
and set 
\[
\beta = (\epsilon'_{S',S'})^{-1} \circ \alpha_{S,S} \circ \epsilon_{S,S}.
\]
Then 
\begin{itemize}
\item[(i)] $\beta$ is an isomorphism of groups from $S$ to $S'$,
\item[(ii)] $\alpha(P) = \beta(P)$ for all $P \in \Delta$, and for each $P, Q
\in \Delta$ with $P \leq Q$, we have $\alpha(P) \leq \alpha(Q)$ and
$\alpha_{P,Q}(\epsilon_{P,Q}(1)) = \epsilon'_{\alpha(P),\alpha(Q)}(1)$,
\item[(iii)] the functor $c_{\beta}\colon \F|_\Delta \to \F'|_{\alpha(\Delta)}$
defined by $c_\beta(P) = \beta(P)$ for each $P \in \Delta$, and by
$c_{\beta}(\phi) = \beta \circ \phi \circ \beta^{-1}$ for each morphism $\phi
\in \Hom_\F(P,Q)$ with $P, Q \in \Delta$, is well-defined and an isomorphism of
categories. 
\end{itemize}
\item[(c)] If $\alpha$ satisfies (typ) and (inc), then $\alpha$ is a bijection
on objects, and hence an isomorphism of categories.
\end{itemize}
\end{lem}
\begin{proof}
\noindent
\textbf{(a)}: The subgroup $S$ (resp. $S'$) is characterized as the
unique object of $\T$ (resp. $\T'$) which receives a morphism from every
object. As $\alpha$ is an equivalence, this implies $\alpha(S) = S'$, and hence
also $\alpha(P) \leq S' = \alpha(S)$ for all $P \in \Delta$.

\smallskip
\noindent
\textbf{(b)}: By (a), $\epsilon'_{\alpha(P), \alpha(S)}(1)$ is defined,
so the right hand side of the equality in (inc) makes sense.
Assume (typ) and (inc) and set $\beta = (\epsilon'_{S',S'})^{-1} \circ
\alpha_{S,S} \circ \epsilon_{S,S}$ as above. 

\smallskip
\noindent
\textbf{(i)}: As $\alpha$ is an equivalence, it is a bijection on
morphism sets. Since $\alpha_{S,S}(\epsilon_{S,S}(S)) =
\epsilon'_{S',S'}(S')$ by (a) and (typ), and since $\epsilon$ is injective on
morphism sets (Axiom (B)), $\beta$ is thus a well-defined isomorphism of groups
from $\Aut_{\T_{\Delta}(S)}(S) = N_S(S) = S$ to $\Aut_{\T_{\Delta'}(S')}(S') =
N_{S'}(S') = S'$. 

\smallskip
\noindent
\textbf{(ii)}: Proving $\alpha(P) = \beta(P)$ for all $P \in \Delta$ means
showing $\alpha_{S,S}(\epsilon_{S,S}(P)) = \epsilon'_{S',S'}(\alpha(P))$ for
all such $P$.  For the proof fix an object $P$ and let $x\in P$.  By (typ),
there exist $y\in\alpha(S)=S'$ and $y'\in \alpha(P)$ such that
$\alpha_{S,S}(\epsilon_{S,S}(x))=\epsilon'_{S',S'}(y)$ and
$\alpha_{P,P}(\epsilon_{P,P}(x))=\epsilon'_{\alpha(P),\alpha(P)}(y')$. Note
that 
\[\epsilon_{P,S}(1)\circ
\epsilon_{P,P}(x)=\epsilon_{P,S}(x)=\epsilon_{S,S}(x)\circ\epsilon_{P,S}(1).\]
Applying $\alpha$ we obtain from this and our assumption that
\[
\epsilon'_{\alpha(P),S'}(1)\circ
\epsilon'_{\alpha(P),\alpha(P)}(y')=\epsilon'_{S',S'}(y)\circ
\epsilon'_{\alpha(P),S'}(1)
\] 
and thus $\epsilon'_{\alpha(P),S'}(y')=\epsilon'_{\alpha(P),S'}(y)$. By axiom
(B), $\epsilon'_{\alpha(P),S'}$ is injective. Thus, $y=y'\in\alpha(P)$ and thus
$\alpha_{S,S}(\epsilon_{S,S}(x))=\epsilon'_{S',S'}(y)\in\epsilon'_{S',S'}(\alpha(P))$.
This proves $\alpha_{S,S}(\epsilon_{S,S}(P))\leq\epsilon'_{S',S'}(\alpha(P))$.

\smallskip

To prove the opposite inclusion let $a\in\alpha(P)$. By the first assumption,
there exist $x\in S$ and $x'\in P$ such that
$\epsilon'_{S',S'}(a)=\alpha_{S,S}(\epsilon_{S,S}(x))$ and
$\epsilon'_{\alpha(P),\alpha(P)}(a)=\alpha_{P,P}(\epsilon_{P,P}(x'))$. Then
\begin{eqnarray*}
 \alpha_{P,S}(\epsilon_{P,S}(x'))&=&\alpha_{P,S}(\epsilon_{P,S}(1))\circ \alpha_{P,P}(\epsilon_{P,P}(x'))\\
 &=& \epsilon'_{\alpha(P),S'}(1)\circ \epsilon'_{\alpha(P),\alpha(P)}(a)\\
 &=& \epsilon'_{\alpha(P),S'}(a)\\
 &=& \epsilon'_{S',S'}(a)\circ \epsilon'_{\alpha(P),S'}(1)\\
 &=& \alpha_{S,S}(\epsilon_{S,S}(x))\circ \alpha_{P,S}(\epsilon_{P,S}(1))\\
 &=& \alpha_{P,S}(\epsilon_{P,S}(x)).
\end{eqnarray*}
As $\alpha_{P,S}$ and $\epsilon_{P,S}$ are injective (Axiom B), it follows
$x=x'\in P$ and thus $\epsilon'_{S',S'}(a)\in\alpha_{S,S}(\epsilon_{S,S}(P))$.
This shows $\alpha_{S,S}(\epsilon_{S,S}(P))=\epsilon'_{S',S'}(\alpha(P))$ and
completes the proof that $\alpha$ and $\beta$ coincide on objects.  

Now let $P \leq Q$ be an inclusion of subgroups in $\Delta$. Then $\alpha(P) =
\beta(P) \leq \beta(Q) = \alpha(Q)$. The equality $\epsilon_{Q,S}(1) \circ
\epsilon_{P,Q}(1) = \epsilon_{P,S}(1)$ holds in $\T$. Applying $\alpha$ to this
and using (inc) we have 
\[
\epsilon'_{\alpha(Q),\alpha(S)}(1) \circ \alpha_{P,Q}(\epsilon_{P,Q}(1)) = \epsilon'_{\alpha(P),\alpha(S)}(1). 
\]
On the other hand, $\epsilon'_{\alpha(P),\alpha(Q)}(1)$ is a morphism in $\T'$
since $\alpha(P) \leq \alpha(Q)$, and we have a similar equality
\[
\epsilon'_{\alpha(Q),\alpha(S)}(1) \circ \epsilon'_{\alpha(P),\alpha(Q)}(1) = \epsilon'_{\alpha(P),\alpha(S)}(1). 
\]
Since $\epsilon'_{\alpha(Q),\alpha(S)}(1)$ is a monomorphism in $\T'$, we have
$\alpha_{P,Q}(\epsilon_{P,Q}(1)) = \epsilon'_{\alpha(P),\alpha(Q)}(1)$,
completing the proof of (ii). 

\smallskip
\noindent
\textbf{(iii)}: This is more or less shown in the proof of
\cite[Proposition~2.5]{GlaubermanLynd2021}, but not all details are given
there.  In any case, our stated hypotheses here are weaker (we do not assume
$\Delta$, $\Delta'$ contain $\F^{cr}, \F'^{cr}$) and our conclusion is weaker
(we only are claiming an isomorphism of full subcategories of the fusion
systems, not of the fusion systems themselves). So we repeat the proof and add
the details.

\smallskip
\noindent
\emph{Step 1:} Let $\beta_* \colon \T_{\Delta}(S) \to \T_{\Delta'}(S')$ be the
functor defined by $\beta$ on objects and also $\beta$ on morphism sets. We
first show $\epsilon'\circ \beta_* = \alpha \circ \epsilon$.  

The functors $\epsilon' \circ \beta_*$ and $\alpha \circ \epsilon$ agree on
objects by (ii) and (A1). They also agree on morphism sets as we now show.  Let
$P,Q \in \Delta$ and $s \in N_S(P,Q)$. In $\T$ we have 
\[
\epsilon_{S,S}(s) \circ \epsilon_{P,S}(1) = \epsilon_{Q,S}(1) \circ \epsilon_{P,Q}(s).
\]
Applying $\alpha$, and using (inc) and the definition of $\beta$, we have
\[
\epsilon'_{S',S'}(\beta(s)) \circ \epsilon'_{\alpha(P),S'}(1) =
\epsilon'_{\alpha(Q),S'}(1) \circ \alpha_{P,Q}(\epsilon_{P,Q}(s))
\]
We also have the same equality when $\alpha_{P,Q}(\epsilon_{P,Q}(s))$ is replaced by
$\epsilon'_{\alpha(P),\alpha(Q)}(\beta(s))$. Since every morphism of $\T'$ is a
monomorphism, this forces 
$\alpha_{P,Q}(\epsilon_{P,Q}(s)) = \epsilon'_{\alpha(P),\alpha(Q)}(\beta(s))$. 
Thus, 
\begin{equation}
\label{leftsqcommutes}
\epsilon' \circ \beta_* = \alpha \circ \epsilon. 
\end{equation}

\smallskip
\noindent
\emph{Step 2:} We next show that for each $\phi \in \Mor_\T(P,Q)$ with $P, Q \in
\Delta$,
\[
c_{\beta}(\rho_{P,Q}(\phi)) = \rho'_{\alpha(P),\alpha(Q)}(\alpha_{P,Q}(\phi)),
\]
but let us omit subscripts on $\rho, \rho'$, and $\alpha$ in the proof to
lighten the notation, after which the equality read reads
\[
c_{\beta}(\rho(\phi)) = \rho'(\alpha(\phi)). 
\]
This then implies $\beta \circ \rho(\phi) \circ \beta^{-1} \in
\Hom(\beta(P),\beta(Q))$ is a morphism in $\F'$. Using that $\rho$ is
surjective on morphisms by (A2), $c_\beta$ as defined in (iii) is thus a
well-defined functor. 

Let $x \in P$.  By Axiom (C) for $\T$, we have
\[
\phi \circ \epsilon_{P,P}(x) = \epsilon_{Q,Q}(\rho(\phi)(x)) \circ \phi.
\]
Applying $\alpha$ and using \eqref{leftsqcommutes}, this gives 
\[
\alpha(\phi) \circ \epsilon'_{\alpha(P),\alpha(P)}(\beta(x))) = \epsilon'_{\alpha(Q),\alpha(Q)}(\beta(\rho(\phi)(x))) \circ \alpha(\phi).
\]
On the other hand, Axiom (C) for $\T'$ says we have the same equality if we
replace the instance of
$\epsilon'_{\alpha(Q),\alpha(Q)}(\beta(\rho(\phi)(x)))$ by
$\epsilon'_{\alpha(Q),\alpha(Q)}(\rho'(\alpha(\phi))(\beta(x)))$.
Since $\alpha(\phi)$ is an epimorphism in $\T'$ and
$\epsilon'_{\alpha(Q),\alpha(Q)}$ is injective, we have shown
\[
c_{\beta}(\rho(\phi))(\beta(x)) = \beta(\rho(\phi)(x)) = \rho'(\alpha(\phi))(\beta(x))
\]
for all $x \in P$, and hence $c_{\beta}(\rho(\phi)) = \rho'(\alpha(\phi))$ as
claimed. 

\smallskip
\noindent
\textit{Step 3.} We finish the proof of (iii).  By (i) and (ii), $c_{\beta}$ is
a bijection on objects, and it is also an injection on morphisms. By assumption
on $\alpha$ and axiom (A2) for $\T'$, the composite functor $\rho' \circ
\alpha$ is surjective on morphisms, so $c_\beta$ is surjective on morphisms by
Step 2. 

\smallskip
\noindent
\textbf{(c)}: 
By assumption on $\alpha$ and (i)-(ii), $\alpha$ is injective on objects,
and it remains to prove that $\alpha$ is surjective on objects. Assume this is
not the case, and choose $Q' \in \Delta'$ of minimal index in $S'$ subject to
$Q' \notin \alpha(\Delta)$. As $\alpha$ is essentially surjective, there is $P
\in \Delta$ and an isomorphism $\phi' \colon \alpha(P) \to Q'$ in $\T'$. By
Alperin's fusion theorem for transporter systems
\cite[Proposition~3.9]{OliverVentura2007}, there are subgroups $\alpha(P) =
Q_0',\dots,Q_n' = Q'$ in $\Delta'$, subgroups $T_i' \geq \gen{Q_{i-1}',
Q_{i}'}$, automorphisms $\tau_i' \in \Aut_{\T'}(T_i')$, and isomorphisms
$\phi_i' \in \Iso_{\T'}(Q_{i-1}', Q_i')$ for $i = 1,\dots,n$, such that
$\phi_i' = \tau_i'|_{Q_{i-1}',Q_i'}$ for all $i$, and $\phi' = \phi_n' \circ
\cdots \circ \phi_1'$. 

Choose the least index $m \in \{0,1,\dots,n\}$ such that $Q_{m}' \notin
\alpha(\Delta)$. Then $m > 0$ and $Q_{m-1}' \in \alpha(\Delta)$, so that
$Q_{m-1}' \neq Q_{m}'$. The subgroup $T_m'$ therefore has smaller index in $S'$
than $Q'$, and hence $T_{m}' \in \alpha(\Delta)$. Fix $T_m, Q_{m-1} \in \Delta$
and $\tau_m \in \Aut_\T(T_m)$ such that $\beta(Q_{m-1}) = \alpha(Q_{m-1}) =
Q_{m-1}'$, $\beta(T_m) = \alpha(T_m) = T_m'$ and
$\alpha_{T_m,T_m}(\tau_m) = \tau_m'$. 

Now $\rho'(\tau_m')$ is an $\F'$-automorphism of $T_m'$ sending $Q_{m-1}'$ onto
$Q_{m}'$. Thus, by (iii), $c_{\beta}^{-1}(\rho'(\tau_m')) = \beta^{-1} \circ
\rho'(\tau_m') \circ \beta$ is an $\F$-automorphism of $T_m$ sending $Q_{m-1} =
\beta^{-1}(Q_{m-1}')$ onto $\beta^{-1}(Q_{m}')$.  Since $Q_{m-1} \in \Delta$
and $\Delta$ is closed under $\F$-conjugacy by (A1), we have $\beta^{-1}(Q_m')
\in \Delta$. Hence $Q_{m}' = \alpha(\beta^{-1}(Q_{m}')) \in \alpha(\Delta)$ by
(ii), and this contradicts the choice of $m$.
\end{proof}

An equivalence satisfying (typ) is said to be \emph{isotypical}. One
satisfying (inc) is said to \emph{send inclusions to inclusions}. Given
Lemma~\ref{L:isotrans}, it is now sensible to define an isomorphism of
transporter systems to be an isotypical equivalence sending inclusions to
inclusions.

\begin{defn}\label{D:isotrans}
Let $\F$ and $\F'$ be fusion systems over the finite $p$-groups $S$ and $S'$,
and let  $(\T,\epsilon,\rho)$ and $(\T',\epsilon',\rho')$ be transporter
systems over $\F$ and $\F'$, respectively. An equivalence of categories
$\alpha\colon \T\rightarrow\T'$ is called an \emph{isomorphism} of transporter
systems if for all objects $P$ of $\T$, 
\begin{itemize}
\item[(typ)] $\alpha_{P,P}(\epsilon_{P,P}(P))=\epsilon'_{\alpha(P),\alpha(P)}(\alpha(P))$ and 
\item[(inc)] $\alpha_{P,S}(\epsilon_{P,S}(1))=\epsilon'_{\alpha(P),\alpha(S)}(1)$. 
\end{itemize}
An isomorphism $\alpha \colon \T \to \T'$ of transporter systems is said to be
\emph{rigid} if $S = S'$ and $\alpha_{S,S} \circ \epsilon_{S,S} =
\epsilon'_{S,S}$ as homomorphisms $S \to \Aut_{\T'}(S)$.
\end{defn}

Thus, an isomorphism of transporter systems is in particular an isomorphism of
categories by Lemma~\ref{L:isotrans}(c), and a rigid isomorphism is one for
which the isomorphism $\beta$ of Lemma~\ref{L:isotrans}(b) is the identity map
on $S = S'$.  

This version of the definition of isomorphism of transporter systems is
equivalent to the one given in \cite[Definition~2.3]{GlaubermanLynd2021}.  The
definition of an isotypical equivalence is the same here as there.  The authors
formulated the condition that $\alpha$ sends inclusions to inclusions by
requiring that $\alpha_{P,Q}(\epsilon_{P,Q}(1)) =
\epsilon'_{\alpha(P),\alpha(Q)}(1)$ for each pair of objects $P \leq Q$ in
\cite{GlaubermanLynd2021}.  But as was pointed out to the third author by
Julian Kaspczyk, this assumes implicitly that $P \leq Q$ implies $\alpha(P)
\leq \alpha(Q)$, which presumably need not hold for an arbitrary equivalence.
Lemma~\ref{L:isotrans}(a) however shows that (inc) makes sense when $\alpha$ is
isotypical, and then Lemma~\ref{L:isotrans}(ii) gives indeed that $\alpha(P)
\leq \alpha(Q)$ and $\alpha_{P,Q}(\epsilon_{P,Q}(1)) =
\epsilon'_{\alpha(P),\alpha(Q)}(1)$ whenever $P \leq Q$ are objects of $\T$.
Thus, (typ) and the seemingly stronger condition
\begin{enumerate} 
\item[(inc')]
$\alpha_{P,Q}(\epsilon_{P,Q}(1)) = \epsilon'_{\alpha(P),\alpha(Q)}(1)$ whenever
$P \leq Q$ are objects of $\T$ with $\alpha(P) \leq \alpha(Q)$
\end{enumerate}
are together equivalent to conditions (typ) and (inc). We will refer to
property (inc') also by saying that ``$\alpha$ sends inclusions to
inclusions''.

\subsection{The correspondence between transporter systems and localities}

Throughout this subsection let $\F$ be a fusion system over $S$.

Every locality $(\L,\Delta,S)$ over $\F$ leads to a transporter system
associated to $\F$. To see that we need to consider conjugation from the
\emph{left}. If $f,x\in\L$ such that $(f,x,f^{-1})\in\D$ (or equivalently
$x\in\D(f^{-1})$), then we set $^f\!x:=\Pi(f,x,f^{-1})=x^{f^{-1}}$. Similarly,
if $f\in\L$ and $\H\subseteq \D(f^{-1})$, then set
\[^f\!\H:=\H^{f^{-1}}:=\{^f\!x\colon x\in\H\}.\]
Define $\T_\Delta(\L)$ to be the category whose object set is $\Delta$ with the
morphism set $\Mor_{\T_\Delta(\L)}(P,Q)$ between two objects  $P,Q\in\Delta$
given as the set of triples $(f,P,Q)$ with $f\in\L$ such that $P\subseteq
\D(f^{-1})$ and $^f\!P\leq Q$. This leads to a transporter system
$(\T_\Delta(\L),\epsilon,\rho)$, where for all $P,Q\in\Delta$, $\epsilon_{P,Q}$
is the inclusion map and $\rho_{P,Q}$ sends $(f,P,Q)$ to the conjugation map
$P\rightarrow Q,\;x\mapsto {^f\!x}$.

Conversely, Chermak showed in \cite[Appendix]{Chermak2013} essentially that
every transporter system leads to a locality. More precisely, it is proved in
\cite[Theorem~2.11]{GlaubermanLynd2021} that there is an equivalence of
categories between the category of transporter systems with morphisms the
isomorphisms and the category of localities with morphisms the
isomorphisms, and such an equivalence can be chosen to preserve the rigid
isomorphisms. The definition of a locality in \cite{GlaubermanLynd2021}
differs slightly from the one given in this paper, but the two definitions can
be seen to be equivalent if one uses firstly that conjugation by $f\in\L$ from
the left corresponds to conjugation by $f^{-1}$ from the right, and secondly
that for every partial group $\L$ with product $\Pi\colon\D\rightarrow\L$ the
axioms of a partial group yield $\D=\{w\in\W(\L)\colon w^{-1}\in\D\}$.

We will consider punctured groups in either setting thus using the term
``punctured group'' slightly abusively.

\begin{defn} We call a transporter system $\T$ over $\F$ a \emph{punctured
group} if the object set of $\T$ equals the set of all non-identity subgroups.
Similarly, a locality $(\L,\Delta,S)$ over $\F$ is said to be a \emph{punctured
group} if $\Delta$ is the set of all non-identity subgroups of $S$.
\end{defn}

Observe that a transporter system over $\F$ which is a punctured group exists
if and only if a locality over $\F$ which is a punctured group exists. If it
matters it will always be clear from the context whether we mean by a punctured
group a transporter system or a locality.

\subsection{Linking localities and linking systems}\label{SS:linkingSystem}

As we have seen in the previous subsection, localities correspond to
transporter systems. Of fundamental importance in the theory of fusion systems
are (centric) linking systems, which are special cases of transporter systems.
It is therefore natural to look at localities corresponding to linking systems.
Thus, we will introduce special kinds of localities called linking localities.
We will moreover introduce a (slightly non-standard) definition of linking
systems and summarize some of the most important results about the existence
and uniqueness of linking systems and linking localities. Throughout this
subsection let $\F$ be a saturated fusion system over $S$. 

We refer the reader to \cite{AschbacherKessarOliver2011} for the definitions of
$\F$-centric and $\F$-centric radical subgroups denoted by $\F^c$ and $\F^{cr}$
respectively. Moreover, we will use the following definition which was
introduced in \cite{Henke2019}.

\begin{defn}\label{D:subcentric} A subgroup $P\leq S$ is called $\F$-\emph{subcentric}
if $O_p(N_\F(Q))$ is centric in $\F$ for every fully $\F$-normalized
$\F$-conjugate $Q$ of $P$. The set of subcentric subgroups is denoted by
$\F^s$.  
\end{defn}

Recall that $\F$ is called constrained if there is an $\F$-centric normal
subgroup of $\F$. It is shown in \cite[Lemma~3.1]{Henke2019} that a subgroup
$P\leq S$ is $\F$-subcentric if and only if for some (and thus for every) fully
$\F$-normalized $\F$-conjugate $Q$ of $P$, the normalizer $N_\F(Q)$ is
constrained. 

\begin{defn}
\begin{itemize} \item A finite group $G$ is said to be of
\textit{characteristic $p$} if $C_G(O_p(G))\leq O_p(G)$.  
\item Define a
locality $(\L,\Delta,S)$ to be of \textit{objective characteristic $p$} if, for
any $P\in\Delta$, the group $N_\L(P)$ is of characteristic $p$.  \item A
locality $(\L,\Delta,S)$ over $\F$ is called a \textit{linking locality}, if
$\F^{cr}\subseteq \Delta$ and $(\L,\Delta,S)$ is of objective characteristic
$p$.
 \item A \textit{subcentric linking locality} over $\F$ is a linking locality
 $(\L,\F^s,S)$ over $\F$. Similarly, a \textit{centric linking locality} over
 $\F$ is a linking locality $(\L,\F^c,S)$ over $\F$.
\end{itemize} \end{defn}

If $(\L,\Delta,S)$ is a centric linking locality, then it is shown in
\cite[Proposition~1]{Henke2019} that the corresponding transporter system
$\T_\Delta(\L)$ is a centric linking system.  Also, if $(\L,\Delta,S)$ is a
centric linking locality, then it is a centric linking system in the sense of
Chermak \cite{Chermak2013}, i.e. we have the property that $C_\L(P)\leq P$ for
every $P\in\Delta$. 

The term linking system is used in \cite{Henke2019} for all transporter systems
coming from linking localities, as such transporter systems have properties
similar to the ones of linking systems in Oliver's definition \cite{Oliver2010}
and can be seen as a generalization of such linking systems. We adapt this
slightly non-standard definition here.

\begin{defn} A \emph{linking system} over $\F$ is a transporter system $\T$
over $\F$ such that $\F^{cr}\subseteq\obj(\T)$ and $\Aut_\T(P)$ is of
characteristic $p$ for every $P\in\obj(\T)$. A \emph{subcentric linking system}
over $\F$ is a linking system $\T$ whose set of objects is the set $\F^s$ of
subcentric subgroups.  \end{defn}

Proving the existence
and uniqueness of centric linking systems was a long-standing open problem,
which was solved by Chermak \cite{Chermak2013}. 
Building on a basic idea in Chermak's proof, Oliver \cite{Oliver2013} gave a
new one via an earlier developed cohomological  obstruction theory. Both proofs
depend a priori on the classification of finite simple groups, but work of
Glauberman and the third author of this paper \cite{GlaubermanLynd2016} removes
the dependence of Oliver's proof on the classification. The precise theorem
proved is the following.

\begin{thm}[Chermak \cite{Chermak2013}, Oliver \cite{Oliver2013},
Glauberman--Lynd \cite{GlaubermanLynd2016}]\label{T:centric} There exists a
centric linking system associated to $\F$ which is unique up to an isomorphism
of transporter systems. Similarly, there exists a centric linking locality over
$\F$ which is unique up to a rigid isomorphism.  
\end{thm}

Using the existence and uniqueness of centric linking systems one can
relatively easily prove the following theorem.
\begin{thm}[Henke \cite{Henke2019}] \label{T:subcentric} The following hold:
\begin{itemize} 
\item [(a)] 
If $\F^{cr}\subseteq \Delta\subseteq \F^s$ such that $\Delta$ is
overgroup-closed in $S$ and closed under $\F$-conjugacy, then there exists a
linking locality over $\F$ with object set $\Delta$, and such a linking
locality is unique up to a rigid isomorphism. Similarly, there exists a linking
system $\T$ associated to $\F$ whose set of objects is $\Delta$, and such a
linking system is unique up to an isomorphism of transporter systems.
Moreover, the nerve $|\T|$ is homotopy equivalent to the nerve of a centric
linking system associated to $\F$.  \item [(b)] The set $\F^s$ is
overgroup-closed in $S$ and closed under $\F$-conjugacy. In particular, there
exists a subcentric linking locality over $\F$ which is unique up to a rigid
isomorphism, and there exists a subcentric linking system associated to $\F$
which is unique up to an isomorphism of transporter systems.  \end{itemize}
\end{thm}

The existence of subcentric linking systems stated in part (b) of the above
theorem gives often a way of proving the existence of a punctured group,
and indeed yields a punctured group directly when the fusion system is of
characteristic $p$-type. 

\begin{defn}\label{D:charptype} 
The saturated fusion system $\F$ is of \emph{characteristic $p$-type} if
$N_\F(Q)$ is constrained for every nontrivial fully $\F$-normalized subgroup
$Q$ of $S$.
\end{defn}

Note that a fusion system $\F$ is constrained if and only if the trivial
subgroup of $S$ is $\F$-subcentric, and thus if and only if every subgroup of
$S$ is $\F$-subcentric. The next lemma gives an analogous characterization for
fusion systems of characteristic $p$-type. Properties (c) and (c') of it will
be the usual way we use the characteristic $p$-type condition in
Section~\ref{S:pgexotic}, for example. 

\begin{lem}\label{L:orderp-subcent}
The following are equivalent:
\begin{itemize}
\item[(a)] $\F$ is of characteristic $p$-type. 
\item[(b)] Every nonidentity subgroup of $S$ is $\F$-subcentric. 
\item[(c)] $N_\F(Q)$ is constrained for each fully normalized $Q \leq S$ of order $p$. 
\item[(c')] $C_\F(Q)$ is constrained for each fully normalized $Q \leq S$ of order $p$. 
\end{itemize}
\end{lem}
\begin{proof}
Points (a) and (b) are equivalent by the definition of subcentric subgroup. For
the equivalence of (c) and (c') we refer to \cite[Lemma~2.13]{Henke2019}.
Finally, (b) and (c) are equivalent by Theorem~\ref{T:subcentric}(b) (i.e.,
$\F^s$ is overgroup-closed and closed under $\F$-conjugacy). 
\end{proof}
Thus, if $\F$ is of characteristic $p$-type but not constrained, then the set
$\F^s$ equals the set of all non-identity subgroups. In any case, by
Theorem~\ref{T:subcentric}(b) and Lemma~\ref{L:orderp-subcent}, there exists a
canonical punctured group associated to each $\F$ of characteristic $p$-type,
namely the subcentric linking locality (or the subcentric linking system if one
uses the language of transporter systems). 

\subsection{Partial normal $p^\prime$-subgroups}

Normal $p^\prime$-subgroups are often considered in finite group theory. We
will now introduce a corresponding notion in localities and prove some basic
properties. Throughout this subsection let $(\L,\Delta,S)$ be a locality.

\begin{defn} A \emph{partial normal $p^\prime$-subgroup of $\L$} is a partial
normal subgroup $\N$ of $\L$ such that $\N\cap S=1$. The locality
$(\L,\Delta,S)$ is said to be \emph{$p^\prime$-reduced} if there is no
non-trivial partial normal $p^\prime$-subgroup of $\L$.  \end{defn}

\begin{remark}\label{R:pprimereduced} If $(\L,\Delta,S)$ is a locality over a
fusion system $\F$, then for any $p^\prime$-group $N$, the direct product
$(\L\times N,\Delta,S)$ is a locality over $\F$ such that $N$ is a partial
normal $p^\prime$-subgroup of $\L\times N$ and $(\L\times N)/N\cong \L$; see
\cite{Henke2017} for details of the construction of direct products of
localities. Thus, if we want to prove classification theorems for localities,
it is actually reasonable to restrict attention to $p^\prime$-reduced
localities.  \end{remark}

Recall that, for a finite group $G$, the largest normal $p^\prime$-subgroup is
denoted by $O_{p^\prime}(G)$. Indeed, a similar notion can be defined for
localities. Namely, it is a special case of \cite[Theorem~5.1]{Chermak2022}
that the product of two partial normal $p^\prime$-subgroups is again a partial
normal $p^\prime$-subgroup. Thus, the following definition makes sense.

\begin{defn} 
The largest normal $p^\prime$-subgroup of $\L$ is denoted by
$O_{p^\prime}(\L)$.  
\end{defn}

We will now prove some properties of partial normal $p^\prime$-subgroups. To
start, we show two lemmas which generalize corresponding statements for groups.
The first of these lemmas gives a way of passing from an arbitrary locality to
a $p^\prime$-reduced locality.

\begin{lem}\label{L:factoroutOpprime} Set $\ov{\L}:=\L/O_{p^\prime}(\L)$. Then
$(\ov{\L},\ov{\Delta},\ov{S})$ is $p^\prime$-reduced.  \end{lem} \begin{proof}
Let $\N$ be the preimage of $O_{p^\prime}(\ov{\L})$ under the natural
projection $\L\rightarrow\ov{\L}$. Then by \cite[Proposition~4.7]{Chermak2022},
$\N$ is a partial normal subgroup of $\L$ containing $O_{p^\prime}(\L)$.
Moreover, $\ov{\N\cap S}\subseteq \ov{\N}\cap \ov{S}=1$, which implies $\N\cap
S\subseteq O_{p^\prime}(\L)$ and thus $\N\cap S\subseteq O_{p^\prime}(\L)\cap
S=1$. Thus, $\N$ is a partial normal $p^\prime$ subgroup of $\L$ and so by
definition contained in $O_{p^\prime}(\L)$. This implies
$O_{p^\prime}(\ov{\L})=\ov{\N}=1$.  \end{proof}

\begin{lem}\label{L:proveOpprime=N} Given a partial normal $p^\prime$-subgroup
$\N$ of $\L$, the image of $O_{p^\prime}(\L)$ in $\L/\N$ under the canonical
projection is a partial normal $p^\prime$-subgroup of $\L/\N$. In particular,
if $\L/\N$ is $p^\prime$-reduced, then $\N=O_{p^\prime}(\L)$.  \end{lem}

\begin{proof} 
Set $\ov{\L}:=\L/\N$. Then by \cite[Proposition~4.7]{Chermak2022},
$\ov{O_{p^\prime}(\L)}$ is a partial normal subgroup of $\ov{\L}$. By
Lemma~\ref{L:PreimageovS}, the preimage of $\ov{S}$ equals $\N S$. As
$\N\subseteq O_{p^\prime}(\L)$, the preimage of $\ov{O_{p^\prime}(\L)}\cap
\ov{S}$ is thus contained in $O_{p^\prime}(\L)\cap (\N S)$. By the Dedekind
Lemma \cite[Lemma~1.10]{Chermak2022}, we have $O_{p^\prime}(\L)\cap (\N
S)=\N(O_{p^\prime}(\L)\cap S)=\N$. Hence, $\ov{O_{p^\prime}(\L)}\cap
\ov{S}=\One$ and $\ov{O_{p^\prime}(\L)}$ is a normal $p^\prime$-subgroup of
$\ov{\L}$. If $\ov{\L}=\L/\N$ is $p^\prime$-reduced, it follows that
$\ov{O_{p^\prime}(\L)}=\One$ and so $O_{p^\prime}(\L)=\N$.
\end{proof}

We now proceed to prove some technical results which are needed in the next
subsection.

\begin{lem}\label{L:proveOpprimeTrivial} 
If $\N$ is a partial normal $p^\prime$-subgroup of $\L$, then $f\in C_\N(S_f)$
for every $f\in\N$.  
\end{lem} 
\begin{proof} 
Let $f\in\N$, set $P:=S_f$ and let $s\in P$. Then $P^f\leq S$ and thus
$P^{fs}\leq S$. Moreover, $P^s=P$. Thus, $w=(s^{-1},f^{-1},s,f)\in\D$ via
$P^{fs}$. Now $\Pi(w)=(f^{-1})^sf=s^{-1}s^f\in \N\cap S=\One$ and hence
$s^f=s$.  As $s\in P$ was arbitrary, this proves $f\in C_\N(P)$.  
\end{proof}

\begin{lem}\label{C:proveOpprimeTrivial} If $\N$ is a non-trivial partial
normal $p^\prime$ subgroup of $\L$, then there exists $P\in\Delta$ such that
$N_\N(P)=C_\N(P)\neq 1$. In particular, if $O_{p^\prime}(N_\L(P))=1$ for all
$P\in\Delta$, then $O_{p^\prime}(\L)=1$.  \end{lem}

\begin{proof} 
Let $\N$ be a non-trivial partial normal $p^\prime$-subgroup and pick $1\neq
f\in\N$. Then $P:=S_f\in\Delta$ by Lemma~\ref{L:LocalitiesProp}(d), and it
follows from Lemma~\ref{L:proveOpprimeTrivial} that $1\neq f\in C_\N(P)$. As
$N_\N(P)$ is a normal $p^\prime$-subgroup of $N_\L(P)$ and $P$ is a normal
$p$-subgroup of $N_\L(P)$, we have $C_\N(P)=N_\N(P)$. Hence,
$C_\N(P)=N_\N(P)\neq 1$ is a normal $p^\prime$-subgroup of $N_\L(P)$ and the
assertion follows.  \end{proof}

\begin{cor}\label{C:LinkingLocOpprimeTrivial} If $(\L,\Delta,S)$ is a linking
locality or, more generally, a locality of objective characteristic $p$, then
$O_{p^\prime}(\L)=1$.  \end{cor} \begin{proof} If, for every $P\in\Delta$, the
group $N_\L(P)$ is of characteristic $p$, then it is in particular
$p^\prime$-reduced. Thus, the assertion follows from
Corollary~\ref{C:proveOpprimeTrivial}.  \end{proof}

\subsection{A signalizer functor theorem for punctured groups} \label{SS:SignalizerFunctor}

In this section we provide some tools for showing that a locality has a
non-trivial partial normal $p^\prime$-subgroup. Corresponding problems for
groups are typically treated using signalizer functor theory. A similar
language will be used here for localities. We will start by investigating how a
non-trivial partial normal $p^\prime$-subgroup can be produced if some
information is known on the level of normalizers of objects. We will then use
this to show a theorem for punctured groups which looks similar to the
signalizer functor theorem for finite groups, but is much more elementary to
prove. Throughout this subsection let $(\L,\Delta,S)$ be a locality.

\begin{defn}\label{D:SignalizerFunctorObjects} A \emph{signalizer functor of
$(\L,\Delta,S)$ on objects} is a map from $\Delta$ to the set of subgroups of
$\L$, which associates to $P\in\Delta$ a normal $p^\prime$-subgroup $\Theta(P)$
of $N_\L(P)$ such that the following conditions hold: \begin{itemize} \item
(Conjugacy condition) $\Theta(P)^g=\Theta(P^g)$ for all $P\in\Delta$ and all
$g\in\L$ with $P\leq S_g$.  \item (Balance condition) $\Theta(P)\cap
C_\L(Q)=\Theta(Q)$ for all $P,Q\in\Delta$ with $P\leq Q$.  \end{itemize}
\end{defn}

As seen in Lemma~\ref{C:proveOpprimeTrivial}, given a locality $(\L,\Delta,S)$
with $O_{p^\prime}(\L)\neq 1$, there exists $P\in\Delta$ with
$O_{p^\prime}(N_\L(P))\neq 1$. The next theorem says basically that, under
suitable extra conditions, the converse holds. 

\begin{prop}\label{P:SignalizerFunctorObjects} If $\Theta$ is a signalizer
functor of $(\L,\Delta,S)$ on objects, then  \[
\widehat{\Theta}:=\bigcup_{P\in\Delta}\Theta(P) \] is a partial normal
$p^\prime$-subgroup of $\L$. In particular, the canonical projection
$\rho\colon \L\rightarrow \L/\widehat{\Theta}$ restricts to an isomorphism
$S\rightarrow S^\rho$. Upon identifying $S$ with $S^\rho$, the following
properties hold: \begin{itemize} \item [(a)] $(\L/\widehat{\Theta},\Delta,S)$
is a locality and $\F_S(\L/\widehat{\Theta}) = \F_S(\L)$.  \item [(b)] For each
$P\in\Delta$, the restriction $N_\L(P)\rightarrow N_{\L/\widehat{\Theta}}(P)$
of $\rho$ is an epimorphism with kernel $\Theta(P)$. In particular,
$N_\L(P)/\Theta(P)\cong N_{\L/\widehat{\Theta}}(P)$.  \end{itemize} \end{prop}
\begin{proof} We proceed in three steps, where in the first step, we prove a
technical property, which allows us in the second step to show that
$\widehat{\Theta}$ is a partial normal $p^\prime$-subgroup, and in the third
step to conclude that the remaining properties hold.

\smallskip
\noindent{\em Step 1:} We show $x\in \Theta(S_x)$ for any
$x\in\widehat{\Theta}$. Let $x\in\widehat{\Theta}$. Then by definition of
$\widehat{\Theta}$, the element $x$ lies in $\Theta(P)$ for some $P\in\Delta$.
Choose such $P$ maximal with respect to inclusion. Notice that $[P,x]=1$. In
particular, $P\leq S_x$ and $[N_{S_x}(P),x]\leq \Theta(P)\cap N_S(P)=\One$.
Hence, using the balance condition, we conclude $x\in \Theta(P)\cap
C_\L(N_{S_x}(P))=\Theta(N_{S_x}(P))$. So the maximality of $P$ yields
$P=N_{S_x}(P)$ and thus $P=S_x$. Hence, $x\in\Theta(S_x)$ as required.

\smallskip
\noindent{\em Step 2:} We show that $\widehat{\Theta}$ is a partial normal
$p^\prime$-subgroup of $\L$. Clearly $\widehat{\Theta}$ is closed under
inversion, since $\Theta(P)$ is a group for every $P\in\Delta$. Note also that
$\Pi(\emptyset)=\One\in\widehat{\Theta}$ as $\One\in\Theta(P)$ for any
$P\in\Delta$.  Let now $(x_1,\dots,x_n)\in\D\cap\W(\widehat{\Theta})$
with $n\geq 1$.  Then $R:=S_{(x_1,\dots,x_n)}\in\Delta$ by
Lemma~\ref{L:LocalitiesProp}(f).  Induction on $i$ together with the balance
condition and Step~1 show $R\leq S_{x_i}$ and $x_i\in\Theta(S_{x_i})\leq
\Theta(R)\leq C_\L(R)$ for each $i=1,\dots,n$.  Hence, $\Pi(x_1,x_2,\dots
,x_n)\in\Theta(R)\subseteq\widehat{\Theta}$. Thus, $\widehat{\Theta}$ is a
partial subgroup of $\L$. 

Let $x\in\widehat{\Theta}$ and $f\in\L$ with $(f^{-1},x,f)\in\D$. Then
$X:=S_{(f^{-1},x,f)}\in\Delta$ by Lemma~\ref{L:LocalitiesProp}(f).  Moreover,
$X^{f^{-1}}\leq S_x$. By Step~1, we have $x\in\Theta(S_x)$, and then by the
balance condition, $x\in\Theta(X^{f^{-1}})$. It follows now from the conjugacy
condition that $x^f\in\Theta(X^{f^{-1}})^f=\Theta(X)\subseteq\widehat{\Theta}$.
Hence, $\widehat{\Theta}$ is a partial normal subgroup of $\L$. Notice that
$\widehat{\Theta}\cap S=\One$, as $\Theta(P)\cap S=\Theta(P)\cap N_S(P)=\One$
for each $P\in\Delta$. Thus, $\widehat{\Theta}$ is a partial normal
$p^\prime$-subgroup of $\L$.

\smallskip
\noindent{\em Step 3:} We are now in a position to complete the proof.  By
\cite[Corollary~4.5]{Chermak2022}, the quotient map $\rho\colon \L\rightarrow
\L/\widehat{\Theta}$ is a projection of partial groups with
$\ker(\rho)=\widehat{\Theta}$. Moreover, by the same result, setting
$\Delta^\rho:=\{P^\rho\colon P\in\Delta\}$, the triple
$(\L/\widehat{\Theta},\Delta^\rho,S^\rho)$ is a locality. Notice that
$\rho|_S\colon S\rightarrow S^\rho$ is a homomorphism of groups with kernel
$S\cap\widehat{\Theta}=1$ and thus an isomorphism of groups.  Upon identifying
$S$ with $S^\rho$, it follows now that $(\L/\widehat{\Theta},\Delta,S)$ is a
locality. Moreover, by \cite[Theorem~5.7(b)]{Henke2019}, we have
$\F_S(\L)=\F_S(\L/\widehat{\Theta})$.  So (a) holds. 

Let $P\in\Delta$. By \cite[Theorem~4.3(c)]{Chermak2022}, the restriction of
$\rho$ to a map $N_\L(P)\rightarrow N_{\L/\widehat{\Theta}}(P)$ is an
epimorphism with kernel $N_\L(P)\cap\widehat{\Theta}$. For any $x\in
N_\L(P)\cap\widehat{\Theta}$, we have $P\leq S_x$ and then
$x\in\Theta(S_x)\leq\Theta(P)$ by the balance condition and Step~1. This shows
$N_\L(P)\cap\widehat{\Theta}=\Theta(P)$ and so (b) holds.  
\end{proof}

The property stated in Proposition~\ref{P:SignalizerFunctorObjects}(a)
holds indeed for every partial normal $p^\prime$-subgroup $\widehat{\Theta}$ of
$\L$. More generally, for any partial normal subgroup $\N$ of $\L$, setting
$\ov{\L}:=\L/\N$ and $\ov{\Delta}:=\{\ov{P}\colon P\in\Delta\}$, the triple
$(\ov{\L},\ov{\Delta},\ov{S})$ is a locality with $\F_{\ov{S}}(\ov{\L})\cong
\F_S(\L)/(S\cap \N)$ (see \cite[Corollary~4.5]{Chermak2022} and
\cite[Theorem~5.7]{Henke2019}).

\begin{remark}\label{R:Balance} 
If $P,Q,R\in\Delta$ such that $P\leq Q\leq R$ and the balance condition in
Definition~\ref{D:SignalizerFunctorObjects} holds for the pair $P\leq
Q$, then 
\[
\Theta(Q) \cap C_\L(R) = 
 \Theta(P) \cap C_\L(Q) \cap C_\L(R) = \Theta(P) \cap C_\L(R)
\]
so $\Theta(R) = \Theta(Q) \cap C_{\L}(R)$ if and only if $\Theta(R) = \Theta(P)
\cap C_\L(R)$. Hence, in this situation balance holds for the pair $Q \leq R$
if and only if balance holds for the pair $P \leq R$.
\end{remark}

\begin{defn}
Let $G$ be a finite group. Then $G$ is said to be \emph{$p$-constrained} if
$G/O_{p^\prime}(G)$ is of characteristic $p$. The group $G$ is called
\emph{Sylow $p$-constrained}, if $C_T(O_p(G))\leq O_p(G)$ for some (and thus
for every) Sylow $p$-subgroup $T$ of $G$. 
\end{defn}

The following proposition is essentially a restatement of
\cite[Proposition~6.4]{Henke2019}, but we will give an independent proof
building on the previous proposition.

\begin{prop}\label{P:GetLocalityObjectiveCharp} 
Let $(\L,\Delta,S)$ be a locality such that $N_\L(P)$ is $p$-constrained
for all $P \in \Delta$. For each $P\in\Delta$, set
\[
\Theta(P):=O_{p^\prime}(N_\L(P)). 
\]
Then the assignment $\Theta$ is a
signalizer functor of $(\L,\Delta,S)$ on objects and $O_{p^\prime}(\L)$ equals
$\widehat{\Theta}:=\bigcup\{\Theta(P)\colon P\in\Delta\}$.  In particular, the
canonical projection $\rho\colon \L\rightarrow \L/\widehat{\Theta}$ restricts
to an isomorphism $S\rightarrow S^\rho$. Upon identifying $S$ with $S^\rho$,
the following properties hold: \begin{itemize} \item [(a)]
$(\L/\widehat{\Theta},\Delta,S)$ is a locality of objective characteristic $p$.
\item [(b)] $\F_S(\L/\widehat{\Theta})=\F_S(\L)$.  \item [(c)] For every
$P\in\Delta$, the restriction $N_\L(P)\rightarrow N_{\L/\widehat{\Theta}}(P)$
of $\rho$ is an epimorphism with kernel $\Theta(P)$. In particular, 
$N_\L(P)/\Theta(P)\cong N_{\L/\widehat{\Theta}}(P)$.  
\end{itemize} 
\end{prop}

\begin{proof} We remark first that, as any normal $p^\prime$-subgroup of
$N_\L(P)$ centralizes $P$ and $O_{p^\prime}(C_\L(P))$ is characteristic in
$C_\L(P)\unlhd N_\L(P)$, we have $\Theta(P)=O_{p^\prime}(C_\L(P))$ for every
$P\in\Delta$.

\smallskip

We show now that the assignment $\Theta$ is a signalizer functor of $\L$ on
objects. It follows from Lemma~\ref{L:LocalitiesProp}(b) that the conjugacy
condition holds. Thus, it remains to show the balance condition, i.e. that
$\Theta(Q)=\Theta(P)\cap C_\L(Q)$ for any $P,Q\in\Delta$ with $P\leq Q$. For
the proof note that $P$ is subnormal in $Q$. So by induction on the subnormal
length and by Remark~\ref{R:Balance}, we may assume that $P\unlhd  Q$. Set
$G:=N_\L(P)$. Then $Q\leq G$ and $C_\L(Q)=C_{G}(Q)$. As $G$ is $p$-constrained,
it follows from \cite[8.2.12]{KurzweilStellmacher2004} that
$O_{p^\prime}(N_G(Q))=O_{p^\prime}(G)\cap N_G(Q)=O_{p^\prime}(G)\cap C_G(Q)$.
Hence,
$\Theta(Q)=O_{p^\prime}(C_\L(Q))=O_{p^\prime}(C_{G}(Q))=O_{p^\prime}(N_G(Q))=O_{p^\prime}(G)\cap
C_G(Q)=\Theta(P)\cap C_\L(Q)$. This proves that the assignment $\Theta$ is a
signalizer functor of $(\L,\Delta,S)$ on objects. In particular, by
Proposition~\ref{P:SignalizerFunctorObjects}, the subset \[
\widehat{\Theta}:=\bigcup_{P\in\Delta}\Theta(P) \] is a partial normal
$p^\prime$-subgroup of $\L$. Moreover, upon identifying $S$ with its image in
$\L/\widehat{\Theta}$, the triple $(\L/\widehat{\Theta},\Delta,S)$ is a locality
and properties (b) and (c) hold.  Part (c) and our assumption yield (a). Hence,
by Corollary~\ref{C:LinkingLocOpprimeTrivial}, we have
$O_{p^\prime}(\L/\widehat{\Theta})=\One$. So by Lemma~\ref{L:proveOpprime=N},
we have $\widehat{\Theta}=O_{p^\prime}(\L)$ and the proof is complete.
\end{proof}

\begin{lem}\label{L:pconstrained} If $G$ is a Sylow $p$-constrained finite group,
then $G$ is $p$-constrained.
\end{lem}
\begin{proof}
Write $T$ for a Sylow $p$-subgroup of $G$, and set $P:=O_p(G)$.
Then the centralizer $C_T(P)$ equals $Z(P)$ and is thus a central Sylow
$p$-subgroup of $C_G(P)$. So e.g. by the Schur-Zassenhaus Theorem
\cite[6.2.1]{KurzweilStellmacher2004}, we have $C_G(P)=Z(P)\times
O_{p^\prime}(C_G(P)) \leq Z(P)\times O_{p^\prime}(G)$. Set
$\ov{G}=G/O_{p^\prime}(G)$, write $C$ for the preimage of $C_{\ov{G}}(\ov{P})$
in $G$. As $P$ is normal in $G$, it follows that $[P,C]\leq P\cap
O_{p^\prime}(G)=1$. So $C=C_G(P)$ and thus $\ov{C}= \ov{C_G(P)}\leq \ov{P}$.
Thus, $\ov{G}$ has characteristic $p$ and $G$ is $p$-constrained.
\end{proof}

We now turn attention to the case that $(\L,\Delta,S)$ is a punctured group and
we are given a signalizer functor \emph{on elements of order $p$} in the sense
of Definition~\ref{D:SignalizerFunctor} in the introduction.  We show first
that, if $\theta$ is such a signalizer functor on elements of order $p$ and
$a\in\I_p(S)$, the subgroup $\theta(a)$ depends only on $\<a\>$.

\begin{lem}\label{L:theta(a)} Let $(\L,\Delta,S)$ be a punctured group and let
$\theta$ be a signalizer functor of $(\L,\Delta,S)$ on elements of order $p$.
Then $\theta(a)=\theta(b)$ for all $a,b\in\I_p(S)$ with $\<a\>=\<b\>$.
\end{lem}

\begin{proof} If $\<a\>=\<b\>$, then $[a,b]=1$ and $\theta(a)\subseteq
C_\L(a)=C_\L(b)$. So the balance condition implies $\theta(a)=\theta(a)\cap
C_\L(b) \subseteq \theta(b)$. A symmetric argument gives the opposite
inclusion $\theta(b)\subseteq\theta(a)$, so the assertion holds.  
\end{proof}

Theorem~\ref{T:mainSignalizerFunctor} in the introduction follows directly from
the following theorem, which explains at the same time how a signalizer functor
on objects can be constructed from a signalizer functor on elements of order
$p$.

\begin{thm}[Signalizer functor theorem for punctured groups]\label{T:sigfun}
Let $(\L,\Delta,S)$ be a punctured group and suppose $\theta$ is a signalizer
functor of $(\L,\Delta,S)$ on elements of order $p$. Then a signalizer functor
$\Theta$ of $(\L,\Delta,S)$ on objects is defined by \[
\Theta(P):=\left(\bigcap_{x\in\I_p(P)}\theta(x)\right)\cap C_\L(P)\mbox{ for
all }P\in\Delta.  \] In particular, \[
\widehat{\Theta}:=\bigcup_{P\in\Delta}\Theta(P)=\bigcup_{x\in\I_p(S)}\theta(x)
\] is a partial normal $p^\prime$ subgroup of $\L$ and the other conclusions in
Proposition~\ref{P:SignalizerFunctorObjects} hold.  \end{thm} \begin{proof}
Since $\theta(x)$ is a $p^\prime$-subgroup for each $x\in\I_p(S)$, the subgroup
$\Theta(P)$ is a $p^\prime$-subgroup for each object $P\in\Delta$.  Moreover,
it follows from the conjugacy condition for $\theta$ (as stated in
Definition~\ref{D:SignalizerFunctor}) that $\Theta(P)$ is a normal subgroup of
$N_\L(P)$, and that the conjugacy condition stated in
Definition~\ref{D:SignalizerFunctorObjects} holds for $\Theta$; to obtain the
latter conclusion notice that Lemma~\ref{L:LocalitiesProp}(b) implies
$C_\L(P)^g=C_\L(P^g)$ for every $P\in\Theta$ and every $g\in\L$ with $P\leq
S_g$. 

To prove that $\Theta$ is a signalizer functor on objects, it remains to show
that the balance condition $\Theta(P)\cap C_\L(Q)=\Theta(Q)$ holds for every
pair $P\leq Q$ with $P\in\Delta$. Notice that $P$ is subnormal in $Q$ whenever
$P\leq Q$. Hence, if the balance condition for $\Theta$ fails for some pair
$P\leq Q$ with $P\in\Delta$, then by Remark~\ref{R:Balance}, it fails for some
pair $P\unlhd Q$ with $P\in\Delta$. Suppose this is the case. Among all pairs
$P\unlhd Q$ such that $P\in\Delta$ and the balance condition fails, choose one
such that $Q$ is of minimal order. 

Notice that $P<Q$, as the balance condition would otherwise trivially hold. So
as $1\neq P_0:=C_P(Q)\in\Delta$ and $P_0\unlhd P$, the minimality of $|Q|$
yields that the balance condition holds for the pair $P_0\leq P$.
If the balance condition holds also for the pair $P_0\unlhd Q$, then 
the balance condition holds for the pair $P\leq Q$ by Remark~\ref{R:Balance},
contradicting our assumption. So the balance condition does not hold for
$P_0\leq Q$. Therefore, replacing $P$ by $P_0$, we can and will assume from now
on that $P\leq Z(Q)$.

It is clear from the definition that $\Theta(Q)\leq \Theta(P)\cap C_\L(Q)$.
Hence it remains to prove the opposite inclusion. By definition of
$\Theta(Q)$, this means that we need to show $\Theta(P)\cap C_\L(Q)\leq
\theta(b)$ for all $b\in\I_p(Q)$. To show this fix $b\in \I_p(Q)$. As
$P\in\Delta$, we have $P\neq 1$ and so we can pick $a\in\I_p(P)$. Since $P\leq
Z(Q)$, the elements $a$ and $b$ commute. Hence, the balance condition for
$\theta$ yields \[ \Theta(P)\cap C_\L(Q)\leq \theta(a)\cap C_\L(b)\leq
\theta(b).  \] This completes the proof that $\Theta$ is a signalizer functor
on objects. 

Given $P\in\Delta$, we can pick any $x\in\I_p(P)$ and have
$\Theta(P)\subseteq\theta(x)$. Hence,
$\widehat{\Theta}:=\bigcup_{P\in\Delta}\Theta(P)$ is contained in
$\bigcup_{x\in\I_p(S)}\theta(x)$. The opposite inclusion holds as well,
as Lemma~\ref{L:theta(a)} implies $\theta(x)=\Theta(\<x\>)$ for every
$x\in\I_p(S)$. The assertion follows now from
Proposition~\ref{P:SignalizerFunctorObjects}.  
\end{proof}

\section{Sharpness of the subgroup decomposition}\label{Sec:sharpness}

\subsection{Additive extensions of categories}

Let $\C$ be a (small) category.  Define a category $\C_{\amalg}$ as follows,
see \cite[Sec. 4]{JackowskiMcClure1992}.  The objects of $\C_{\amalg}$ are
pairs $(I,\XXX)$ where $I$ is a finite set and $\XXX \colon I \to \obj(\C)$ is
a function.  A morphisms $(I,\XXX) \to (J,\YYY)$ is a pair $(\si, \fff)$ where
$\si \colon I \to J$ is a function and $\fff \colon I \to \mor(\C)$ is a
function such that $\fff(i) \in \C(\XXX(i), \YYY(\si(i)))$.  We leave it to the
reader to check that this defines a category.

There is a fully faithful inclusion $\C \subseteq \C_{\amalg}$ by sending $X
\in \C$ to the function $\XXX \colon \{\emptyset\} \to \obj(\C)$ with
$\XXX(\varnothing) = X$.  We will write $X$ (not boldface) to denote these
objects in $\C_\amalg$.

The category $\C_{\amalg}$ has a monoidal structure $\coprod$ where $(I,\XXX)
\coprod (J,\YYY) \defeq (I \coprod J, \XXX \coprod \YYY)$.  One checks that
this is the categorical coproduct in $\C_{\amalg}$.  For this reason we will
often write objects of $\C_{\amalg}$ in the form $\coprod_{i \in I} X_i$ where
$X_i \in \C$.  Also, when the indexing set $I$ is understood we will simply
write $\XXX$ instead of $(I,\XXX)$.

When $(I,\XXX)$ is an object and $J \subseteq I$ we will refer to $(J,\XXX|_J)$
as a ``subobject'' of $(I,\XXX)$ and we leave it to the reader to check that
the inclusion is a monomorphism, namely for any two morphisms $\fff,\ggg \colon
\YYY \to \XXX|_J$, if $\incl_{\XXX|_J}^{\XXX} \circ \fff=
\incl_{\XXX|_J}^{\XXX} \circ \ggg$ then $\fff=\ggg$.  One also checks that 

\begin{eqnarray}\label{Eq:morphisms of Ccoprod}
&&\C_{\amalg}(\coprod_{i \in I} X_i,Y) = \prod_{i \in I} \C(X_i,Y), \\
\nonumber
&&\C_{\amalg} (X,\coprod_{i \in I} Y_i) = \coprod_{i \in I} \C(X,Y_i).
\end{eqnarray}

\begin{defn}[{Compare \cite[p. 123]{JackowskiMcClure1992}}]
\label{D:PBtc}
We say that $\C$ satisfies \PBtc~ if the product of each pair of objects in
$\C$ exists in $\C_{\amalg}$ and if the pullback of each diagram $c \to e
\leftarrow d$ of objects in $\C$ exists in $\C_{\amalg}$.
\end{defn}

\begin{defn}[{Compare \cite[p. 124 and Lemma 5.13]{JackowskiMcClure1992}}]
\label{D:proto Mackey}
Assume that $\C$ is a small category satisfying \PBtc.  A functor $M  \colon
\C^\op \to \Ab$ is called a \emph{proto-Mackey functor} if there is a functor
$M_*\colon \C \to \Ab$ such that the following hold.
\begin{itemize}
\item[(a)] $M(C)=M_*(C)$ for any $C \in \obj(\C)$.
\item[(b)] For any isomorphism $\phi \in \C$, $M_*(\phi)=M(\phi^{-1})$.
\item[(c)] By applying $M$ and $M_*$ to a pullback diagram in $\C_{\amalg}$ of the form
\[
\xymatrix{
\coprod_{i \in I} B_i \ar[r]^(0.65){\sum_i \phi_i} \ar[d]_{\sum_i \psi_i} &
D \ar[d]^{\be}
\\
C \ar[r]^{\al} & E
}
\]
where $B_i,C,E \in \C$, there results the following commutative square in $\Ab$
\[
\xymatrix{
\bigoplus_{i \in I} M(B_i) \ar[rr]^(0.6){\sum_i M_*(\phi_i)} & &
M(D)
\\
M(C) \ar[u]^{\oplus_i M(\psi_i)} \ar[rr]^{M_*(\al)} & &
M(E) \ar[u]_{M(\be)}
}
\]
\end{itemize}
\end{defn}

We remark that every pullback diagram in $\C_{\amalg}$ defined by objects in
$\C$ is isomorphic in $\C_{\amalg}$ to a commutative square as in (c) in this
definition.

Given a small category $\D$ and a functor $M \colon \D \to \Ab$, we write
\[
H^*(\D;M) \defeq \underset{\D}{\varprojlim{}^*} M.
\]
for the derived functors of $M$.
We say that $M$ is \emph{acyclic} if $H^i(\D;M)=0$ for all $i>0$.

\begin{prop}[{See \cite[Corollary 5.16]{JackowskiMcClure1992}}] 
\label{P:JM acyclicity}
Fix a prime $p$.
Let $\C$ be a small category which satisfies \PBtc~ and in addition 
\begin{itemize}
\item[(B1)] $\C$ has finitely many isomorphism classes of objects, all morphism
sets are finite and all self maps in $\C$ are isomorphisms.
\item[(B2)] For every object $C \in \C$ there exists an object $D$ such that
$|\C(C,D)| \neq 0 \mod p$.
\end{itemize}
Then any proto-Mackey functor $M \colon \C^{\op} \to \modl{\ZZ_{(p)}}$ is
acyclic, namely $H^i(\C^{\op},M)=0$ for all $i>0$.
\end{prop}

\subsection{Transporter categories}

Let $\F$ be a saturated fusion system over $S$ and let $\T$ be a transporter
system associated with $\F$ (Definition \ref{D:transportersystem}).  By
\cite[Lemmas 3.2(b) and 3.8]{OliverVentura2007} every morphism in $\T$ is both
a monomorphism and an epimorphism.  For any  $P,Q \in \obj(\T)$ such that $P
\leq Q$ denote $\iota_P^Q=\epsilon_{P,Q}(e) \in \Mor_{\T}(P,Q)$.  We think of these as
``inclusion'' morphisms in $\T$.  We obtain a notion of ``extension'' and
``restriction'' of morphisms in $\T$ as follows.  Suppose $\vp \in
\Mor_\T(P,Q)$ and $P' \leq P$ and $Q' \leq Q$ and $\psi \in \Mor_\T(P',Q')$
are such that $\vp \circ \iota_{P'}^P = \iota_{Q'}^Q \circ \psi$. Then we say
that $\psi$ is a restriction of $\vp$ and that $\vp$ is an extension of $\psi$.
Notice that since $\iota_{Q'}^Q$ is a monomorphism, given $\vp$ then its
restriction $\psi$ if it exists, is unique and we will write
$\psi=\vp|_{P'}^{Q'}$.  Similarly, since $\iota_{P'}^P$ is an epimorphism,
given $\psi$, if an extension $\vp$ exists then it is unique.  By \cite[Lemma
3.2(c)]{OliverVentura2007}, given $\vp \in \Mor_\T(P,Q)$ and subgroups $P' \leq P$
and $Q' \leq Q$ such that $\rho(\vp)(P') \leq Q'$, then $\vp$ restricts to a
(unique morphism) $\vp' \in \Mor_\T(P',Q')$.  We will use this fact repeatedly.  In
particular, every morphism $\vp \colon P \to Q$ in $\T$ factors uniquely $P
\xto{\bar{\vp}} \bar{P} \xto{\iota_{\bar{P}}^Q} Q$ where $\bar{\vp}$ is an
isomorphism in $\T$ and $\bar{P}=\rho(\vp)(P)$.  

For any $P,Q \in \obj(\T)$ set
\[
K_{P,Q} = \{(A,\al) \, \colon \, A \leq P, \, A \in \obj(\T), \, \al \in \Mor_\T(A,Q)\}.
\]
This set is partially ordered where $(A,\al) \preceq (B,\be)$ if $A \leq B$ and
$\al=\be|_A$.  Since $K_{P,Q}$ is finite we may consider the set
$K_{P,Q}^{\max}$ of the maximal elements.

For any $x \in N_S(P,Q)$ we write $\widehat{x}$ instead of $\epsilon_{P,Q}(x)$.
There is an action of $Q \times P$ on $K_{P,Q}$ defined by
\[
(y,x) \cdot (A,\al) = (x A x^{-1}, \widehat{y} \circ \al \circ \widehat{x}^{-1}), \qquad (x \in P, \, y \in Q).
\]
This action is order preserving and therefore it leaves $K_{P,Q}^{\max}$
invariant.  
We will write $\overline{K_{P,Q}^{\max}}$ for the set of orbits.
For any $P,Q \in \T$ we will choose \emph{once and for all} a subset 
\[
\K_{P,Q}^{\max} \subseteq K_{P,Q}^{\max}
\] 
of representatives for the orbits of $Q \times P$ on $K_{P,Q}^{\max}$.

\begin{lem}\label{L:max elements in T}
For any $(A,\al) \in K_{P,Q}$ there exists a unique $(B,\be) \in
K_{P,Q}^{\max}$ such that $(A,\al) \preceq (B,\be)$.
\end{lem}

\begin{proof}
We use induction on $[S:A]$.  Fix $(A,\al)\in K_{P,Q}$  and $(B_1,\be_1)$ and
$(B_2,\be_2)$ in $K_{P,Q}^{\max}$ such that $(A,\al) \preceq (B_i,\be_i)$.
Thus, $\be_1|_A=\al=\be_2|_A$.  We may assume that $A < B_i$ since if
say $A=B_1$ then $(B_1,\be_1) \preceq (B_2,\be_2)$ and maximality implies
$(B_1,\be_1)=(B_2,\be_2)$.

For $i=1,2$ set $N_i=N_{B_i}(A)$.  Then $N_i$ contain $A$ properly and we set
$D=\langle N_1,N_2\rangle$.  Then $A \trianglelefteq D$.  Set $T=\al(A)$ and
$\overline{T}=N_Q(T)$.  For $i=1,2$, if $x \in N_i$ then Axiom (C) of
Definition~\ref{D:transportersystem} applied to $\be_i$ yields
\[
\al \circ \widehat{x}|_{A}^{A} = ((\be_i|_{N_i}) \circ
\widehat{x}|_{N_i}^{N_i})|_A = \widehat{\be_i(x)}|_{Q}^{Q} \circ \be_i|_A =
\widehat{\be_i(x)}|_{Q}^{Q} \circ \al.
\]
Notice that $\be_i(x) \in N_Q(T)$, so Axiom (II) of Definition
\ref{D:transportersystem} implies that $\al$ extends
to $\de \in \Mor_\T(D,Q)$.  Since for $i=1,2$ the morphisms $\iota_A^{N_i}
\colon A \to N_i$ are epimorphisms in $\T$, the equality $\be_i|_{N_i} \circ
\iota_A^{N_i} = \al = \de|_A = (\de|_{N_i}) \circ \iota_A^{N_i}$ shows that
$\de|_{N_i}=\be_i|_{N_i}$.  Now we have $(N_i,\be_i|_{N_i}) \preceq (D,\de)$
and $(N_i,\be_i|_{N_i}) \preceq (B_i,\be_i)$ in $K_{P,Q}$.  Since $|A|<|N_i|$
the induction hypothesis implies that $(B_i,\be_i)$ is the unique maximal
extension of $(N_i,\be_i|_{N_i})$ for each $i=1,2$, and both must coincide with
the unique maximal extension of $(D,\de)$.  It follows that
$(B_1,\be_1)=(B_2,\be_2)$.
\end{proof}


The \emph{orbit category} of $\T$ is a category $\O\T$ with the same set of
objects as $\T$.  For any $P,Q \in \O\T$ the morphism set
$\Mor_{\O\T}(P,Q)$ is the set of orbits of $\Mor_\T(P,Q)$ under the action of
$\widehat{Q} = \epsilon_{Q,Q}(Q) \subseteq \Mor_\T(Q,Q)$ via postcomposition.  See
\cite[Section 4, p. 1010]{OliverVentura2007}.  Axiom (C) guarantees that
composition in $\O\T$ is well defined.  
Given $\vp \in \Mor_\T(P,Q)$ we will denote its image in $\Mor_{\O\T}(P,Q)$ by $[\vp]$.

We notice that every morphism in $\O\T$ is an epimorphism, namely for every
$[\al] \in \Mor_{\O\T}(P,Q)$ and $[\be],[\ga] \in
\Mor_{\O\T}(Q,R)$, if $[\be] \circ [\al]=[\ga] \circ[\al]$ then
$[\be]=[\ga]$.  This follows from the fact that every morphism in $\T$ is an
epimorphism.

Consider $P,Q \in \obj(\O\T)$ such that $P \trianglelefteq Q$.  Precomposition
with $[\iota_P^Q]$ gives a ``restriction'' map
\[
\Mor_{\O\T}(Q,S) \xto{\res} \Mor_{\O\T}(P,S).
\]
Observe that $Q$ acts on $\Mor_{\O\T}(P,S)$ by precomposing morphisms
with $[\widehat{x}|_P^P]$ for any $x \in Q$.  This action has $P$ in its kernel
by Axiom (C) of transporter systems.

\begin{lem}\label{L:morphisms in OT mod p}
Let $\T$ be a transporter system and $\O\T$ its orbit category.
\begin{itemize}
\item[(a)]
For any $P,Q \in \obj(\O\T)$ such that $P \trianglelefteq Q$ the map
$\Mor_{\O\T}(Q,S)\to \Mor_{\O\T}(P,S)$ induced by the restriction $[\vp]
\mapsto [\vp|_P]$, gives rise to a bijection
\begin{equation}\label{E:resPQ-fixed-points}
\res \colon \Mor_{\O\T}(Q,S) \to \Mor_{\O\T}(P,S)^{Q/P}
\end{equation}
\item[(b)]
For any $P \in \O\T$ we have $|\Mor_{\O\T}(P,S)| \neq 0 \mod p$.
\end{itemize}
\end{lem}
\begin{proof}
(a)
First, observe that if $[\vp] \in \Mor_{\O\T}(Q,S)$ then $[\vp|_P]$ is fixed by
$Q/P$ by Axiom (C), hence the image of $\res$ is contained in
$\Mor_{\O\T}(P,S)^{Q/P}$.  Now suppose that $[\vp] \in
\Mor_{\O\T}(P,S)^{Q/P}$ and set $\bar{P}=\rho(\vp)(P)$.  Since $[\vp]$
is fixed by $Q/P$ this exactly means that for every $x \in Q$ there exists $y
\in N_S(\bar{P})$ such that $\vp \circ \widehat{x}|_P^P =
\widehat{y} \circ \vp$ and Axiom (II) implies
that $\vp$ extends to a morphism $\psi \in \Mor_{\T}(Q,S)$.  This shows
that the map $\res$ in \eqref{E:resPQ-fixed-points} is onto
$\Mor_{\O\T}(P,S)^{Q/P}$.  It is injective because $[\iota_P^Q]$ is an
epimorphism in $\O\T$.

(b) 
Use induction on $[S:P]$.  If $P=S$ then $\epsilon_{S,S}(S)$ is a Sylow
$p$-subgroup of $\Aut_\T(S)=\Mor_{\T}(S,S)$ and therefore $|\Mor_{\O\T}(S,S)|
\neq 0 \mod p$.  Suppose $P<S$ and set $Q=N_S(P)$.  Then $Q > P$ and since
$Q/P$ is a finite $p$-group,
$|\Mor_{\O\T}(P,S)|=|\Mor_{\O\T}(P,S)^{Q/P}| \mod p$.  It follows
from part (a) and the induction hypothesis on $[S:Q]$ that
$|\Mor_{\O\T}(P,S)| \neq 0 \mod p$.
\end{proof}


In the remainder of this subsection we will prove that $\O\T$ satisfies \PBtc, 
keeping the notation from above.

\begin{defn}\label{D:P boxtimes Q}
Let $\T$ be a transporter system with orbit category $\O\T$. For $P,Q
\in \O\T$, consider the object $P \boxtimes Q$ of $\O\T_{\amalg}$ given
by
\[
P \boxtimes Q = \coprod_{(L,\la) \in \K_{P,Q}^{\max}} L.
\]
That is, $P \boxtimes Q \colon \K_{P,Q}^{\max} \to \obj(\O\T)$ is the
function $(L,\la) \mapsto L$.  Let $\pi_1 \colon P \boxtimes Q \to P$ and
$\pi_2 \colon P \boxtimes Q \to Q$ be the morphisms in $\O\T_{\amalg}$ defined
by $\pi_1=\sum_{(L,\la)} [\iota_L^P]$ and $\pi_2=\sum_{(L,\la)} [\la]$.
\end{defn}

\begin{prop}\label{P:products in OT}
Let $\T$ be a transporter system with orbit category $\O\T$.  Then $P
\boxtimes Q$ is the product in $\O\T_{\amalg}$ of $P,Q \in \obj(\O\T)$.
\end{prop}

\begin{proof}
It follows from \eqref{Eq:morphisms of Ccoprod} that it suffices to show that
\[
\O\T_{\amalg}(R,\P \boxtimes Q) \xto{\ (\pi_1{}_*,\pi_2{}_*) \ } \Mor_{\O\T}(R,P) \times \Mor_{\O\T}(R,Q)
\]
is a bijection for any $R \in \obj(\O\T)$.  Write $\pi=(\pi_1{}_*,\pi_2{}_*)$.

{\em Surjectivity of $\pi$:} Consider $[\vp] \in \Mor_{\O\T}(R,P)$ and $[\psi]
\in \Mor_{\O\T}(R,Q)$.  Set $A=\rho(\vp)(R)$.  Then $A \leq P$ and there exists
an isomorphism $\bar{\vp} \in \Mor_{\T}(R,A)$ such that $\vp=\iota_A^P \circ
\bar{\vp}$.

Set $\al = \psi \circ (\bar{\vp})^{-1} \in \Mor_{\T}(A,Q)$.  Then $(A,\al) \in
K_{P,Q}$.  Choose $(B,\be) \in K_{P,Q}^{\max}$ such that $(A,\al) \preceq
(B,\be)$.  There exists a unique $(L,\la) \in \K_{P,Q}^{\max}$ and some  $x \in
P$ and $y \in Q$ such that 
\[
(L,\la)=(y,x) \cdot (B,\be) = (xBx^{-1}, \widehat{y} \circ \be \circ
(\widehat{x}|_B^L)^{-1}).
\]
Set $\mu = (\widehat{x}|_B^L) \circ \iota_A^B \circ \bar{\vp} \in \Mor_{\T}(R,L)$.  
It defines a morphism $[\mu] \colon R \to P \boxtimes Q$ in $\O\T_{\amalg}$ via
the inclusion $(L,\la) \subseteq P \boxtimes Q$.  We claim that
$\pi([\mu])=([\vp],[\psi])$, completing the proof of the surjectivity of
$\pi$.  By definition of $\pi_1 \colon P \boxtimes Q \to P$ and $\pi_2 \colon P
\boxtimes Q \to Q$,

\begin{eqnarray*}
\pi_1{}_*([\mu]) &=& 
[\iota_L^P] \circ [\mu] = 
[\iota_L^P \circ (\widehat{x}|_{B}^{L}) \circ \iota_A^{B} \circ \bar{\vp} ] = 
[(\widehat{x}|_A^P) \circ \bar{\vp}]=
[(\widehat{x}|_P^P) \circ \vp]= [\vp] 
\\
\pi_2{}_*([\mu]) &=& 
[\la] \circ [\mu] = 
[\widehat{y}^{-1} \circ \la \circ \mu] =
[\widehat{y}^{-1} \circ \la \circ (\widehat{x}|_{B}^{L}) \circ \iota_A^{B} \circ \bar{\vp} ] = 
\\
&=&
[\be \circ \iota_A^{B} \circ \bar{\vp} ] = 
[\al \circ \bar{\vp} ] = [\psi].
\end{eqnarray*}

\noindent 
{\em Injectivity of $\pi$:} Suppose that $h,h' \in \O\T_{\amalg}(R,P \boxtimes
Q)$ are such that $\pi(h)=\pi(h')$.  From \eqref{Eq:morphisms of Ccoprod}
there are $(L,\la), (L',\la') \in \K_{P,Q}^{\max}$ and $\vp \in \Mor_{\T}(R,L)$
and $\vp' \in \Mor_{\T}(R,L')$ such that $h=[\vp]$ and $h'=[\vp']$ via the
inclusions $L,L' \subseteq P \boxtimes Q$.  The hypothesis $\pi(h)=\pi(h')$
then becomes $[\iota_L^P \circ \vp]=[\iota_{L'}^P \circ \vp']$ and $[\la \circ
\vp]=[\la' \circ \vp']$.  Thus,
\begin{eqnarray}\label{Eq:pboxtimesq1}
&& \iota_{L'}^P \circ \vp' = \widehat{x} \circ \iota_L^P \circ \vp \qquad \text{ for some $x \in P$}\\
\nonumber
&& \la' \circ \vp' = \widehat{y} \circ \la \circ \vp \qquad \text{for some $y \in Q$}.
\end{eqnarray}
Set $A=\rho(\vp)(R)$ and $A'=\rho(\vp')(R)$.  
There are factorizations $\vp=\iota_{A}^L \circ \bar{\vp}$ and
$\vp'=\iota_{A'}^{L'} \circ \bar{\vp'}$ for isomorphisms $\bar{\vp} \in
\Mor_{\T}(R,A)$ and $\bar{\vp'} \in \Mor_{\T}(R,A')$ in $\T$.  We get from
\eqref{Eq:pboxtimesq1} that $\iota_{A'}^P \circ \bar{\vp'} = \widehat{x}|_{A}^P
\circ \bar{\vp}$.  From this we deduce that $A'=xAx^{-1}$ and that
$\bar{\vp'}=\widehat{x}|_{A}^{A'} \circ \bar{\vp}$.  The second equation in
\eqref{Eq:pboxtimesq1} gives
\[
\la' \circ \iota_{A'}^{L'} = \widehat{y} \circ \la \circ \iota_{A}^L \circ
\widehat{x^{-1}}|_{A'}^{A}.
\]
We deduce that $(A',\la'|_{A'}) = (y,x) \cdot (A,\la|_A)$.  Clearly
$(A',\la'|_{A'}) \preceq (L',\la')$ and $(A,\la|_A) \preceq (L,\la)$ so Lemma
\ref{L:max elements in T} implies that $(L',\la')=(y,x) \cdot (L,\la)$.  Since
$(L,\la)$ and $(L',\la')$ are elements of $\K_{P,Q}^{\max}$ and are in the same
orbit of $Q\times P$ it follows that $(L,\la)=(L',\la')$.  In particular $x \in
N_P(L)$, and it follows from \eqref{Eq:pboxtimesq1} that $\vp'=\widehat{x}\circ
\vp$ and that $\la =\widehat{y} \circ \la \circ \widehat{x}^{-1}$ (since $\vp$
is an epimorphism in $\T$).  By Axiom (II) of Definition
\ref{D:transportersystem}, there exists an extension of $\la$ to a morphism
$\tilde{\la} \colon \langle L,x\rangle \to Q$ in $\T$.  Notice that $\langle
L,x\rangle \subseteq P$ so the maximality of $(L,\la)$ implies that $x \in L$.
Since $\vp'=\widehat{x} \circ \vp$ we deduce  $[\vp']=[\vp]$ namely $h=h'$ as
needed.  
\end{proof}

\begin{defn}\label{D:pullback of PRQ}
Let $P \xto{f} R \xleftarrow{g} Q$ be morphisms in $\O\T$.  Let $U(f,g)$ be the
subobject of $P \boxtimes Q$ obtained by restriction of $P \boxtimes Q \colon
\K_{P,Q}^{\max} \xto{(L,\la) \mapsto L} \obj(\O\T)$ to the set $I$ of
those $(L,\la) \in \K_{P,Q}^{\max}$ such that $f \circ [\iota_L^P] = g \circ
[\la]$.
\end{defn}

\begin{prop}\label{P:pullbacks in OT}
Let $\T$ be a transporter system with orbit category $\O\T$, and let $P
\xto{f} R \xleftarrow{g} Q$ be morphisms in $\O\T$. Then $(U(f,g),
\pi_1|_{U(f,g)}, \pi_{2}|_{U(f,g)})$ is the pullback of $P$ and $Q$ along $f$
and $g$ in $\O\T_{\amalg}$.  Moreover, the pullback of $P \xto{\iota_P^R} R
\xleftarrow{\iota_Q^R} Q$ is 
\[
\coprod_{x \in (Q \backslash R /P)_\T} Q^x \cap P
\]
where $x$ runs through representatives of the double cosets such that $Q^x\cap
P=x^{-1}Qx \cap P$ is an object of $\T$.
\end{prop}

\begin{proof}
It follows from \eqref{Eq:morphisms of Ccoprod} that in order to check the
universal property of $U=U(f,g)$ it suffices to test objects $T \in \O\T$.
Suppose that we are given morphisms $T \xto{[\vp]} P$ and $T \xto{[\psi]} Q$
which satisfy $f \circ[\vp]=g \circ[\psi]$.  We obtain $T \xto{([\vp],[\psi])}
P \boxtimes Q$ which factors $T \xto{\bar{h}} L \subseteq P\boxtimes Q$ for
some $(L,\la) \in \K_{P,Q}^{\max}$.  Then
\begin{eqnarray*}
&& f \circ \pi_1|_{L} \circ \bar{h} = f \circ \pi_1 \circ ([\vp],[\psi]) = f \circ [\vp] \\
&& g \circ \pi_2|_{L} \circ \bar{h} = g \circ \pi_2 \circ ([\vp],[\psi]) = g \circ [\psi].
\end{eqnarray*}
Since $\bar{h}$ is an epimorphism in $\O\T$ and since $f \circ [\vp]=g \circ
[\psi]$ by assumption, it follows that $f \circ \pi_1|_{L}=g \circ \pi_2|_{L}$
which is the statement $f \circ [\iota_{L}^P]= g \circ [\la]$.  This precisely
means that $(L,\la) \in I$ where $I$ is as in Definition~\ref{D:pullback
of PRQ}, hence $h=([\vp],[\psi])$ factors through $U$ and clearly $\pi_1 \circ
h=[\vp]$ and $\pi_2 \circ h=[\psi]$.  Since the inclusion $U \subseteq P
\boxtimes Q$ is a monomorphism in $\O\T_{\amalg}$, there can be only one
morphism $h \colon T \to U$ such that $\pi_1 \circ h=[\vp]$ and $\pi_2 \circ
h=[\psi]$.  This shows that $U=U(f,g)$ is the pullback.

Now assume we are given $P \xto{\iota} R \xleftarrow{\iota} Q$.  The indexing
set of the object $U(f,g)$ consists of $(L,\la) \in \K_{P,Q}^{\max}$ such that
$[\iota_L^R]=[\iota_Q^R \circ \la]$, namely $\iota_Q^R \circ \la =
\widehat{x}|_L^R$ for some $x \in N_R(L,Q)$, which is furthermore unique.
Since $\iota_Q^R$ is a monomorphism, this implies that $\la=\widehat{x}|_L^Q$.
Since $(L,\la)$ is maximal, $L=Q^x \cap P$.  We obtain a map
$U(\iota_P^R,\iota_Q^R) \to (Q\backslash R/P)_\T$ which sends $(L,\lambda)$ to
$PxQ$ with $x \in N_R(L,Q)$ described above.  This map is injective because if
$QxP=Qx'P$ are the images of $(L,\la)$ and $(L',\la')$ then $x'=qxp$ for some
$p \in P$ and $q \in Q$ and it follows that $L'=p^{-1}Lp$ and that
$\la=\hat{x}|_L^Q$ and $\la'=\hat{x'}|_{L'}^Q$ and therefore $\la = \hat{q}
\circ \la' \circ \hat{p}|_L^{L'}$, so $(L,\la)$ and $(L',\la')$ are in the same
orbit of $Q \times P$, hence they must be equal.  It is surjective since for
any $PxQ \in(Q\backslash R/P)_\T$ we obtain a summand in
$U(\iota_P^R,\iota_Q^R)$ which is equivalent in $K_{P,Q}$ to $(L,\lambda)$ with
$L=Q^x \cap P$ and $\lambda=\widehat{x}|_L^Q$.
\end{proof}

\subsection{The $\La$-functors}

Let $\Ga$ be a finite group and $M$ a (right) $\Ga$-module.  
Let $p$ be a fixed prime and let $\O_p(\Ga)$ be the full subcategory of the category of $\Ga$-sets whose objects are the transitive $\Ga$-sets whose isotropy groups are $p$-groups.  
Let $F_M \colon \O_p(\Ga)^\op \to \Ab$ be the functor which assigns $M$ to the free orbit $\Ga/1$ and $0$ to all orbits with non-trivial isotropy.  
Define (\cite[Definition 5.3]{JackowskiMcClureOliver1995})
\[
\La^*(\Ga, M) \defeq \underset{\O_p(\Ga)^{\op}}{\varprojlim{}^*} F_M \qquad \Big( =H^*(\O_p(\Ga)^\op;F_M) \Big).
\]
These functors have the following important properties.

\begin{lem}\label{L:Lambda properties}
Suppose that $M$ is a $\ZZ_{(p)}[\Ga]$-module.
\begin{itemize}
\item[(a)] If $C_{\Ga}(M)$ contains an element of order $p$ then
$\La^*(\Ga;M)=0$.
\item[(b)] If $\Ga/C_\Ga(M)$ has order prime to $p$ then $\La^*(\Ga;M)=0$ for
all $* \geq 1$.
\end{itemize}
\end{lem}
\begin{proof}
Point (a) is \cite[Proposition 6.1(ii)]{JackowskiMcClureOliver1995}.  Point (b)
follows from \cite[Proposition 6.1(ii)]{JackowskiMcClureOliver1995} when $p$
divides $|C_\Ga(M)|$ and from \cite[Proposition 6.1(i) and
(iii)]{JackowskiMcClureOliver1995} when $p$ does not divide $|C_\Ga(M)|$.
\end{proof}

Observe that $\rho \colon \T \to \F$ reflects isomorphisms.  Hence the
isomorphism classes of objects of $\T$ and of $\O\T$ are $\F$-conjugacy
classes.

A functor $\Phi \colon \O\T^\op \to \Ab$ is called \emph{atomic} if there
exists $Q \in \obj(\T)$ such that $\Phi$ vanishes outside the $\F$-conjugacy
class of $Q$.  The fundamental property of $\Lambda$-functors is that they
calculate the higher limits of atomic functors:

\begin{lem}[{\cite[Lemma 4.3]{OliverVentura2007}}]
\label{L:La atomic functors}
Let $\T$ be a transporter system associated with a fusion system $\F$ over
$S$.  Let $\Phi \colon \O\T^\op \to \modl{\ZZ_{(p)}}$ be an atomic functor
concentrated on the $\F$-conjugacy class of $Q$.  Then there is a natural
isomorphism
\[
H^*(\O\T^{\op};\Phi) \cong \La^*(\Aut_{\O\T}(Q);\Phi(Q)).
\]
\end{lem}

We remark that the result holds, in fact, for any functor $\Phi$ into the
category of abelian groups (indeed, the proof given by Oliver and Ventura only
uses \cite[Proposition A.2]{OliverVentura2007}).

Notice that $\rho \colon \T \to \F$ induces a functor $\bar{\rho} \colon \O\T
\to \O(\F)$.  We will write $\O\T^c$ for the full subcategory of $\T$ spanned
by $P \in \T$ which are $\F$-centric.

\begin{cor}\label{C:restrict to centrics in T}
Let $\T$ be a transporter category for $\F$.  Let $\bar{\Phi} \colon \O(\F)^\op
\to \modl{\ZZ_{(p)}}$ be a functor and set $\Phi=\bar{\Phi} \circ \bar{\rho}$.
Then $\Phi$ is a functor $\O\T^\op \to \modl{\ZZ_{(p)}}$ and let $\Psi$ be the
restriction of $\Phi$ to $\O\T^c$.  Then the restriction induces an isomorphism
\[
H^*(\O\T^\op,\Phi) \xto{ \ \cong  \ } H^*((\O\T^c)^\op;\Psi).
\]
\end{cor}

\begin{proof}
Let $\Phi' \colon \O\T^\op \to \modl{\ZZ_{(p)}}$ be the functor obtained from
$\Phi$ by setting $\Phi'(Q)=0$ for all $Q\in \obj(\T \setminus \T^c)$ and
$\Phi'(Q)=\Phi(Q)$ otherwise.  This is a well defined functor since the
$\F$-centric subgroups are closed to overgroups. Since there is no
morphism in $\O\T$ from a centric object to a noncentric one, and since $\Psi =
\Phi'|_{(\O \T^c)^{\op}}$ there is an isomorphism of cochain complexes
$C^*(\O\T^{\op}, \Phi') \cong C^*((\O\T^c)^{\op}, \Psi)$ (cf. the description
of the bar resolution in \cite[Section~III.5.1]{AschbacherKessarOliver2011}),
and hence an isomorphism
\[
H^*(\O\T^\op,\Phi') \cong H^*((\O\T^c)^\op,\Psi).
\]
It remains to show that $H^*(\O\T^\op,\Phi) \cong H^*(\O\T^\op,\Phi')$.

Suppose that $Q \in \obj(\T \setminus \T^c)$ has minimal order.  Set
$M=\Phi(Q)$ and let $F_M \colon \O\T^\op \to \modl{\ZZ_{(p)}}$ be the induced
atomic functor.  The minimality of $Q$ implies that there is an injective
natural transformation $F_M \to \Phi$.  By possibly replacing it with an
$\F$-conjugate, we may assume that $Q$ is fully centralized in $\F$.  Since $Q$
is not $\F$-centric, choose some $x \in C_S(Q) \setminus Q$.  Its image in
$\Ga=\Aut_{\O\T}(Q)$ is a non-trivial element (since $x \notin Q$) of order
$p$-power.  It acts trivially on $\Phi(Q)$ because its image in $\Out_\F(Q)$ is
trivial (since the image of $C_S(Q)$ in $\Aut_\F(Q)$ is trivial) and because
$\Phi=\bar{\Phi} \circ \bar{\rho}$.  Lemma \ref{L:Lambda properties}(a) implies
that $\La^*(\Aut_{\O\T}(Q),\Phi(Q))=0$.  It follows from Lemma \ref{L:La atomic
functors} and the long exact sequence in derived limits associated with the
short exact sequence $0 \to F_M \to \Phi \to \Phi/F_M \to 0$ that
$H^*(\O\T^\op,\Phi) \cong H^*(\O\T^\op,\Phi/F_M)$.  But $\Phi/F_M$ is obtained
from $\Phi$ by annihilating the groups $\Phi(Q')$ for all $Q'$ in the
$\F$-conjugacy class of $Q$.  We may now apply the same process to $\Phi/F_M$
and continue inductively (on the number of $\F$-conjugacy classes in $\T
\setminus \T^c$) to show that $H^*(\O\T^\op,\Phi) \cong H^*(\O\T^\op,\Phi')$ as
needed.
\end{proof}

\begin{proof}[{Proof of Theorem \ref{T:sharpness for punctured groups}}]
Let $\F$ be a saturated fusion system over $S$ which affords a punctured group
$\T$.  That is, $\T$ is a transporter system associated to $\F$ with object set
$\De$ containing all the non-trivial subgroups of $S$.

Let $\H^j \colon \O(\F)^\op \to \modl{\ZZ_{(p)}}$ be the functor
\[
\H^j \colon P \mapsto H^j(P;\FF_p)
\]
and let $M^j \colon \O\T^{\op} \to \modl{\ZZ_{(p)}}$ be the composite
$\O\T^{\op} \xto{\bar{\rho}^{\op}} \O(\F)^{\op} \xto{\H^j} \modl{\ZZ_{(p)}}$.
Our goal is to show that for every $j \geq 0$,
\[
H^i(\O(\F^c)^\op;\H^j)=0 \qquad \text{ for all $i \geq 1$}.
\]

Choose a fully normalized $P \in \obj(\T)$.  Since $N_S(P)$ is a Sylow
$p$-subgroup of $\Aut_\T(P)$, see \cite[Proposition 3.4(a)]{OliverVentura2007},
it follows that $C_S(P)$ is a Sylow $p$-subgroup of the kernel of $\Aut_\T(P)
\to \Aut_\F(P)$ and hence $C_S(P)/Z(P)$ is a Sylow $p$-subgroup of the kernel
of $\Aut_{\O\T}(P) \to \Out_\F(P)$.  Thus, if $P$ is $\F$-centric, then this
kernel has order prime to $p$, and so \cite[Lemma 1.3]{BrotoLeviOliver2003b} implies
the first isomorphism in
\[
H^*(\O(\F^c)^\op;\H^j) \cong H^*((\O\T^c)^\op;M^j) \cong H^*(\O\T^\op;M^j),
\]
while Corollary \ref{C:restrict to centrics in T} gives the second.  It
remains to show that $H^*(\O\T^\op;M^j)=0$ for all $j \geq 0$ and all $* \geq
1$.

Assume first that $j \geq 1$.  We will show that $M^j$ is a proto-Mackey
functor for $\O\T$ in this case.  The transfer homomorphisms give rise to a
(covariant) functor $\H^j_* \colon \O(\F) \to \modl{\ZZ_{(p)}}$ where $P
\mapsto H^j(P;\FF_p)$ and to any $\vp \in \F(P,Q)$ we assign $\tr(\vp) \colon
H^j(P;\FF_p) \to H^j(Q;\FF_p)$.  The composition $M^j_* := \H^j_*
\circ \bar{\rho}$ is a covariant functor $\O\T \to \modl{\ZZ_{(p)}}$.

Now, $\O\T$ satisfies \PBtc~ by Propositions \ref{P:products in OT} and
\ref{P:pullbacks in OT}.  Clearly, $M^j$ and $M^j_*$ have the same values on
objects; this is the first condition in Definition \ref{D:proto Mackey}.  The
transfer homomorphisms have the property that if $\vp \colon P \to Q$ is an
isomorphism then $\tr_{P}^{Q}(\vp) = H^j(\vp^{-1};\FF_p)$.  This is the second
condition in Definition \ref{D:proto Mackey}.  The factorisation of morphisms
in $\T$ as isomorphisms followed by inclusions imply that any pullback diagram
$P' \xto{f} R \xleftarrow{g} Q'$ in $\O\T$ is isomorphic to one of the form $P
\xto{[\iota_{P}^R]} R \xleftarrow{[\iota_Q^R]} Q$. If
$U=U([\iota_P^R],[\iota_Q^R])$ is the pullback (Definition \ref{D:pullback of
PRQ}), then by Proposition \ref{P:pullbacks in OT}, $U=\coprod_{x \in X}
Q^x \cap P$ where $x$ runs through a set $X = (Q \backslash R/P)_\T$ of
representatives of those double cosets $QxP$ with $x \in R$ and $Q^x \cap P \in \T$, 
namely  $Q^x \cap P \neq 1$ (because $\obj(\T)$ is the set of all non-trivial
subgroups of $S$).  Since $j\geq 1$ we have that $H^j(1;\FF_p)=0$ so
$\bigoplus_{x \in X} H^j(Q^x \cap P;\FF_p) = \bigoplus_{x \in Q \backslash R/P}
H^j(Q^x \cap P;\FF_p)$, where here $Q \backslash R/P$ is a full set of double
coset representatives.  Mackey's formula \cite[Proposition~9.5(iii)]{Brown1982}
then gives the commutativity of the diagram
\[
\xymatrix{
\bigoplus_{x \in X} H^j(Q^x \cap P;\FF_p)
\ar[rr]^(0.60){\sum_x \tr_{Q^x \cap P}^P \circ c_x}
& &
H^j(Q;\FF_p)
\\
H^j(P; \FF_p)
\ar[u]^{(\res_{Q^x \cap P}^P)_{x \in X}}
\ar[rr]_{\tr_P^R}
& &
H^j(R; \FF_p)
\ar[u]_{\res_Q^R}
}
\]
This shows that the third condition in Definition \ref{D:proto Mackey} also
holds and that $M^j$ is a proto-Mackey functor.  Now, Condition (B1) in
Proposition \ref{P:JM acyclicity}  clearly holds for $\O\T$ and (B2) holds by
Lemma \ref{L:morphisms in OT mod p}.  It follows that $H^i(\O\T^{\op};M^j)=0$
for all $i \geq 1$ as needed.

It remains to deal with the case $j=0$.  In this case $\H^0$ is the constant
functor with value $\FF_p$.  Thus, $\Out_\F(P)$ acts trivially on $\FF_p$ for
any $P \in \F^c$.  It follows from Lemma \ref{L:Lambda properties}(b) that if
$P=S$ then $\La^i(\Out_\F(S),\FF_p)=0$ for all $i>0$, and if $P<S$ then
$\Out_\F(P)$ contains an element of order $p$ so $\La^*(\Out_\F(P),\FF_p)=0$.
Now \cite[Proposition 3.2]{BrotoLeviOliver2003} together with a filtration of
$\H^0$ by atomic functors show that $\H^0$ is acyclic.
\end{proof}

\section{Punctured groups for $\F_{\Sol}(q)$}\label{S:sol}

The Benson-Solomon systems were predicted to exist by Benson \cite{Benson1994},
and were later constructed by Levi and Oliver \cite{LeviOliver2002,
LeviOliver2005}. They form a family of exotic fusion systems at the prime $2$
whose isomorphism types are parametrized by the nonnegative integers. Later,
Aschbacher and Chermak gave a different construction as the fusion system of an
amalgamated free product of finite groups \cite{AschbacherChermak2010}.  The
main result of this section is the following theorem. 

\begin{thm}\label{T:solq}
A Benson-Solomon system $\F_{\Sol}(q)$ over the $2$-group $S$ has a punctured
group if and only if $q \equiv \pm 3 \pmod{8}$.  If $q \equiv \pm 3 \pmod{8}$,
there is a punctured group $\L$ for $\F_{\Sol}(q)$ which is unique up to rigid
isomorphism with the following two properties: 
\begin{itemize}
\item[(1)] $C_\L(Z(S)) = \Spin_7(3)$, and
\item[(2)] $\L|_\Delta$ is a linking locality, where $\Delta$ is the set of
$\F$-subcentric subgroups of $S$ of $2$-rank at least $2$. 
\end{itemize} 
\end{thm}

\subsection{Notation for $\Spin_7$ and $\Sol$}\label{SS:notation}
It will usually be most convenient to work with a Lie theoretic description of
$\Spin_7$. The notational conventions that we use in this section for algebraic
groups and finite groups of Lie type are summarized in Appendix~\ref{S:lie}. 

\subsubsection{The maximal torus and root system} 
Let $p$ be an odd prime, and set
\[
\ol{H} = \Spin_7(\ol{\FF}_p).
\]
Fix a maximal torus $\ol{T}$ of $\ol{H}$, let $X(\ol{T}) = \Hom(\ol{T},
\ol\FF_{p}^{\times}) \cong \ZZ^3$ be the character group (of algebraic
homomorphisms), and denote by $V = \RR \otimes_{\ZZ} X(\ol T)$ the ambient
Euclidean space which we regard as containing $X(\ol T)$. Let $\Sigma(\ol T)
\subseteq X(\ol T)$ be the set of $\ol T$-roots.  Denote a $\ol T$-root
subgroup for the root $\alpha$ by 
\[
\ol X_{\alpha} = \{x_{\alpha}(\lambda) \mid \lambda \in \ol \FF_p\}. 
\] 
As $\ol H$ is semisimple, it is generated by its root subgroups
\cite[Theorem~1.10.1(a)]{GLS3}.  We assume that the implicit parametrization
$x_{\alpha}(\lambda)$ of the root subgroups is one like that given by
Chevalley, so that the Chevalley relations hold with respect to certain signs
$c_{\alpha,\beta} \in \{\pm 1\}$ associated to each pair of roots
\cite[Theorem~1.12.1]{GLS3}. 

We often identify $\Sigma(\ol T)$ with the abstract root system
\[
\Sigma = \{\pm e_i \pm e_j, \pm e_i \mid 1 \leq i,j \leq 3\} \subseteq \RR^3
\]
of type $B_3$, having base $\Pi = \{\alpha_1,\alpha_2,\alpha_3\}$ with
\[
\alpha_1 = e_1-e_2, \quad \alpha_2 = e_2 - e_3, \quad \alpha_3 = e_3, 
\]
where the $e_i$ are standard vectors. Write $\Sigma^{\vee} = \{\alpha^{\vee}
\mid \alpha \in \Sigma\}$ for the dual root system, where $\alpha^\vee =
2\alpha/(\alpha,\alpha)$. 

Instead of working with respect to the $\alpha_i$, it is sometimes convenient
to work instead with a different set of roots $\{\beta_i\} \subseteq \Sigma$:
\[
\beta_1 = \alpha_1, \quad \beta_2 = \alpha_{1}+2\alpha_2+2\alpha_3 = e_1+e_2, \quad \beta_3 = \alpha_3.
\]
This is an orthogonal basis of $V$ with respect to the standard inner product
$(\,,\,)$ on $\RR^3$. An important feature of this basis is that for each $i$
and $j$, 
\begin{eqnarray}
\label{E:bibjZlin}
\Sigma \cap \{k\beta_i + l\beta_j \mid k,l \in \ZZ\} = \{\pm \beta_i, \pm
\beta_j\}, \end{eqnarray} a feature not enjoyed, for example, by an orthogonal
basis consisting of short roots. In particular, the $\beta_j$-root string
through $\beta_i$ consists of $\beta_i$ only. This implies via
Lemma~\ref{L:signs}(4), that the corresponding signs involving the $\beta_i$
that appear in the Chevalley relations for $\ol H$ are
\begin{eqnarray}
\label{E:bibj}
c_{\beta_i, \beta_j} = 1 \text{ if } i \neq j, \text{ and } c_{\beta_i,\beta_i} = -1. 
\end{eqnarray}

\subsubsection{The torus and the lattice of coroots}
We next set up notation and state various relations for elements of $\ol T$.
Let 
\[
h_{\alpha}(\lambda) \in \ol T \quad \text{and} \quad n_{\alpha}(\lambda) \in N_{\ol H}(\ol T)
\]
be as given in Appendix~\ref{S:lie} as words in the generators
$x_\alpha(\lambda)$. By Lemma~\ref{L:liebasic} and since $\ol H$ is of
universal type, there is an isomorphism $\ZZ\Sigma^\vee \otimes \ol\FF_p^\times
\to \ol T$ which on simple tensors sends $\alpha^\vee \otimes \lambda$ to
$h_{\alpha}(\lambda)$, and the homomorphisms $h_{\alpha_i} \colon \ol
\FF_p^\times \to \ol T$ are injective. In particular, as $\Pi^{\vee} =
\{\alpha_1^\vee, \alpha_2^\vee, \alpha_3^\vee\}$ is a basis for
$\ZZ\Sigma^{\vee}$, we have $\ol T = h_{\alpha_1}(\ol \FF_p^\times) \times
h_{\alpha_2}(\ol \FF_p^\times) \times h_{\alpha_3}(\ol \FF_p^\times)$.  Define
elements $z$ and $z_1 \in \ol T$ by  
\[
z_1 = h_{\alpha_1}(-1) \quad \text{ and } \quad z = h_{\alpha_3}(-1). 
\]
Thus, $z$ and $z_1$ are involutions.  Similar properties hold with respect to
the $\beta_i$'s.  Recall that $\beta_i = \alpha_i$ for $i = 1, 3$. Since
$\beta_2^\vee = \alpha_1^\vee + 2\alpha_2^\vee + \alpha_3^\vee$,
Lemma~\ref{L:liebasic}(3) yields 
\[
h_{\beta_2}(-1) = h_{\alpha_1}(-1)h_{\alpha_2}( (-1)^2 )h_{\alpha_3}(-1) = z_1z. 
\]
In particular, 
\begin{eqnarray}
\label{E:proddistinctinvs}
h_{\beta_1}(-1)h_{\beta_2}(-1)h_{\beta_3}(-1) = z_1z_1zz = 1. 
\end{eqnarray}
However, as the $\ZZ$-span of the $\beta_i^\vee$'s is of index $2$ in
$\ZZ\Sigma^\vee$ and every element of $\ol \FF_p^\times$ is a square, we still
have
\begin{eqnarray}
\label{E:hbetagenT}
\ol T = h_{\beta_1}(\ol \FF_p^{\times})h_{\beta_2}(\ol \FF_p^{\times})h_{\beta_3}(\ol \FF_p^{\times}).
\end{eqnarray}
So the $h_{\beta_i}(\ol \FF_p)^{\times}$ generate $\ol T$, but the product is
no longer direct.

\subsubsection{The normalizer of the torus and Weyl group}\label{SS:normtorus}
The subgroup
\[
\widehat W := \gen{n_{\alpha_1}(1),  n_{\alpha_2}(1), n_{\alpha_3}(1)} \leq N_{\ol H}(\ol T)
\]
projects onto the Weyl group 
\[
W = \gen{w_{\alpha_1}, w_{\alpha_2}, w_{\alpha_3}} \cong C_2 \wr S_3 \cong C_2 \times S_4
\]
of type $B_3$ in which the $w_{\alpha_i}$ are fundamental reflections.  Also,
$\widehat{W} \cap \ol T$ is the $2$-torsion subgroup $\{t \in \ol T \mid t^2 =
1\}$ of $\ol T$, see \cite[Remark~1.12.11]{GLS3}. A subgroup similar to
$\widehat W$ was denoted ``$W$'' in \cite[Lemma~4.3]{AschbacherChermak2010}. It
is sometimes called the Tits subgroup. 

Let 
\[
\gamma = c_{\alpha_1, \alpha_2+\alpha_3} \in \{\pm 1\},
\]
and fix a fourth root $i \in \FF_p^\times$ of $1$.  (This notation will
hopefully not cause confusion with the use of $i$ as an index.) Define elements
$w_0, \tau \in N_{\ol H}(\ol T)$ by
\[
w_0 = n_{\beta_1}(-\gamma)n_{\beta_2}(1)n_{\beta_3}(1) \quad \text{and} \quad \tau = n_{\alpha_2+\alpha_3}(1)h_{\beta_1}(-i)h_{\beta_2}(i)h_{\beta_3}(i). 
\]
It will be shown in Lemma~\ref{L:chevrels} that $w_0$ and $\tau$ are commuting
involutions and that $w_0$ inverts $\ol T$. 

\subsubsection{Three commuting $SL_2$'s}\label{SS:SL2}
Let 
\[
\ol L_{i} = \gen{\ol X_{\beta_i}, \ol X_{-\beta_i}},
\]
for $i = 1,2,3$. Thus, $\ol L_i \cong SL_2(\ol \FF_p)$ for each $i$ by the
Chevalley relations, again using that $\ol H$ is of universal type when $i =
3$. A further consequence of \eqref{E:bibjZlin} is that the Chevalley
commutator formula \cite[1.12.1(b)]{GLS3} yields
\[
[\ol L_i, \ol L_j] = 1 \,\,\, \text{ for all } \,\,\, i \neq j.
\]
For each $i$, $\ol L_i$ has unique involution $h_{\beta_i}(-1)$ which generates
the center of $\ol L_i$. By \eqref{E:proddistinctinvs}, the center of the
commuting product $\ol L_1 \ol L_2 \ol L_3$ is $\gen{z,z_1}$, of order $4$. By
\eqref{E:hbetagenT}, $\ol T \leq \ol L_1 \ol L_2 \ol L_3$. 

\subsubsection{The Steinberg endomorphism and $\Spin_7(q)$}\label{SS:Spin_7(q)}
We next set up notation for the Steinberg endomorphism we use to descend from
$\ol H$ to the finite versions. Let $q = p^a$ be a power of $p$.  Let $\epsilon
\in \{\pm 1\}$ be such that $q \equiv \epsilon \pmod{4}$, and let $k$ be the
$2$-adic valuation of $q-\epsilon$.  

The standard Frobenius endomorphism $\zeta$ of $\ol H$ is determined by its
action $x_{\alpha}(\lambda)^{\zeta} = x_{\alpha}(\lambda^p)$ on the root
groups, and so from the definition of the $n_{\alpha}$ and $h_{\alpha}$ in
\eqref{E:nalphahalpha}, also $n_{\alpha}(\lambda)^{\zeta} =
n_{\alpha}(\lambda^p)$ and $h_\alpha(\lambda)^\zeta = h_{\alpha}(\lambda^p)$.
Write $c_{w_0}$ conjugation map induced by $w_0$, as usual, and define 
\[
\sigma =
\begin{cases}
\zeta^a & \text{if } \epsilon = 1\\
\zeta^a c_{w_0} &\text{if } \epsilon = -1. 
\end{cases}
\]
Then $\sigma$ is a Steinberg endomorphism of $\ol H$ in the sense of
\cite[Definition~1.15.1]{GLS3}, and we set 
\[
H := C_{\ol H}(\sigma) = \Spin_7(q).
\]
Given that $w_0$ inverts $\ol T$, the action of $\sigma$ on $\ol T$ is given
for each $t \in \ol T$ by
\[
t^\sigma = t^{\epsilon q},
\]
and hence 
\[
C_{\ol T}(\sigma) = \{t \in \ol T \mid t^{\epsilon q} = t\} \cong (C_{q-\epsilon})^3. 
\]
Likewise,
\[
C_{\ol T}(\sigma c_{w_0}) \cong (C_{q + \epsilon})^3.
\]

Finally, let $\mu = \mu_q \in \ol \FF_p^{\times}$ be a fixed element of
$2$-power order satisfying $\mu^{\epsilon q} = -\mu$ and powering to the fourth
root $i$, and set
\[
c = h_{\beta_1}(\mu)h_{\beta_2}(\mu)h_{\beta_3}(\mu) \in \ol T.
\]

\subsubsection{A Sylow $2$-subgroup}\label{SSS:SylowSpin7}
We next set up notation for Sylow $2$-subgroups of $\ol H$ and $H$ along with
various important subgroups of them. Let
\[
\ol S = \ol T_{2^{\infty}}\widehat{W}_{\ol S},
\]
where $\ol T_{2^{\infty}}$ denotes the $2$-power torsion in $\ol T$ and where
$\widehat{W}_{\ol S} = \gen{n_{\alpha_1}(1), n_{\alpha_2+\alpha_3}(1),
n_{\alpha_3}(1)}$.  Set
\[
S = C_{\ol S}(\sigma).
\]
Define subgroups 
\[
Z < U < E < A \leq S
\]
by 
\[
Z =  \gen{z}, \quad U = \gen{z,z_1}, \quad  E = \{t \in T \mid t^2 = 1\}, \quad  \text{and } A = E\gen{w_0}.
\]
Then $Z = Z(S)$, $U$ is the unique four subgroup normal in $S$, $E$ is
elementary abelian of order $8$, and $A$ is elementary abelian of order $16$.  It will be shown
in Lemma~\ref{L:chevrels} that $w_0 \in S$, and hence $A \leq S$. 

We also write 
\[
T_S = T \cap S;
\]
thus, $T_S = O_2(T) \cong (C_{2^k})^3$ is the $2^k$-torsion in $\ol T$, a Sylow
$2$-subgroup of $T$.

\subsection{Conjugacy classes of elementary abelian subgroups of $\overline{H}$ and $H$}
\label{SS:spinelab}
We state and prove here several lemmas on conjugacy classes of elementary
abelian subgroups of $\ol H$ and $H$, and on the structure of various $2$-local
subgroups.  Much of the material here is written down elsewhere, for example in
\cite{LeviOliver2002} and \cite{AschbacherChermak2010}. Our setup is a little
different because of the emphasis on the Lie theoretic approach, so we aim to
give more detail in order to make the treatment here as self-contained as
possible. 

The first lemma is elementary and records several initial facts about the
elements we have defined in the previous section. Its proof is mainly an
exercise in applying the various Chevalley relations defining $\ol H$. 
\begin{lem}\label{L:chevrels}
Adopting the notation from \S\S\ref{SS:notation}, we have
\begin{itemize}
\item[(1)] $Z(\ol H) = Z = \gen{z}$; 
\item[(2)] the elements $w_0$ and $\tau$ are involutions in $N_{S}(\ol T)-\ol T$,
and $c \in T_S$ has order $2^k$, powering into $E-U$; 
\item[(3)] $w_0$ inverts $\ol T$; and
\item[(4)] $[w_0, \tau] = [c,\tau] = 1$. 
\end{itemize}
\end{lem}
\begin{proof}
\textbf{(1)}: It is well known that $Z(\ol H)$ has order $2$.  We show here for
the convenience of the reader that the involution generating $Z(\ol H)$ is $z =
h_{\alpha_3}(-1)$.  We already observed in \S\S\ref{SS:normtorus} that $z$ is
an involution. For each root $\alpha \in \Sigma$, the inner product of $\alpha$
with $\alpha_3$ is an integer, and so $\gen{\alpha,\alpha_3} =
2(\alpha,\alpha_3) \in 2\ZZ$. By Lemma~\ref{L:liebasic}(1), $h_{\alpha_3}(-1)$
lies in the kernel of $\alpha$.  Thus, the centralizer in $\ol H$ of
$h_{\alpha_3}(-1)$ contains all root groups by
Proposition~\ref{P:subtorus-cent}, and hence $C_{\ol
H}(h_{\alpha_3}(-1)) = \ol H$.

\medskip
\noindent
\textbf{(2)}: We show that $w_{0}$ is an involution. 
Using equations \eqref{E:nn} and \eqref{E:bibj}, we see that 
\begin{eqnarray}
\label{E:nbetainbetaj}
\text{$[n_{\beta_i}(\pm 1), n_{\beta_j}(\pm 1)] = 1$ for each $i, j \in \{1,2,3\}$.} 
\end{eqnarray}
So $w_0^2 = 1$ by \eqref{E:n2} and \eqref{E:proddistinctinvs}. 

We next prove that $\tau$ is an involution. Recall 
\[
\tau = n_{\alpha_2+\alpha_3}(1)h_{\beta_1}(-i)h_{\beta_2}(i)h_{\beta_3}(i).
\]
First, note that $n_{\alpha_2+\alpha_3}(1)^2 = z$. To see this, use
\eqref{E:n2} to get $n_{\alpha_2+\alpha_3}(1)^2 = h_{\alpha_2+\alpha_3}(-1)$.
Then use $(\alpha_2+\alpha_3)^\vee = 2\alpha_2 + 2\alpha_3 = 2\alpha_2^\vee +
\alpha_3^\vee$ and Lemma~\ref{L:liebasic}(3) to get 
\[
n_{\alpha_2+\alpha_3}(1)^2 = h_{\alpha_2+\alpha_3}(-1) = h_{\alpha_2}(-1)^2h_{\alpha_3}(-1) = h_{\alpha_3}(-1) = z
\]
as desired.  Next, the fundamental reflection $w_{\alpha_2+\alpha_3}$
interchanges $\beta_1$ and $\beta_2$ and fixes $\beta_3$, so
$n_{\alpha_2+\alpha_3}(1)$ inverts $h_{\beta_1}(-i)h_{\beta_{2}}(i)$ by
conjugation and centralizes $h_{\beta_3}(i)$ by
\eqref{E:hn}. Hence,
\begin{align*}
\tau^2 &= n_{\alpha_2+\alpha_3}(1)^2(h_{\beta_1}(-i)h_{\beta_2}(i)h_{\beta_3}(i))^{n_{\alpha_2+\alpha_3}(1)}(h_{\beta_1}(-i)h_{\beta_2}(i)h_{\beta_3}(i))\\
       &= n_{\alpha_2+\alpha_3}(1)^2h_{\beta_3}(i)^2 = zz = 1.
\end{align*}

We show $c$ is of order $2^{k}$ and powers into $E-U$. Recall that $k$ is the
$2$-adic valuation of $q-\epsilon$, and that $C_{\ol T}(\sigma) =
(C_{q-\epsilon})^3$.  The latter has Sylow $2$-subgroup of exponent $2^{k}$.
But $c \in C_{\ol T}(\sigma)$ since 
\[
c^\sigma = h_{\beta_1}(\mu^{\epsilon q})h_{\beta_2}(\mu^{\epsilon
q})h_{\beta_3}(\mu^{\epsilon q}) =
h_{\beta_1}(-\mu)h_{\beta_2}(-\mu)h_{\beta_3}(-\mu)
\overset{\eqref{E:proddistinctinvs}}{=}
h_{\beta_1}(\mu)h_{\beta_2}(\mu)h_{\beta_3}(\mu) = c.
\]
So $c$ has order at most $2^k$. On the other hand,
\[
c^{2^{k-1}} = h_{\beta_1}(i)h_{\beta_2}(i)h_{\beta_3}(i). 
\]
As in \S\S\ref{SS:SL2}, we have $h_{\beta_2}(i) =
h_{\alpha_1}(i)h_{\alpha_2}(i^2)h_{\alpha_3}(i)$, and so 
\[
c^{2^{k-1}} = h_{\alpha_1}(-1)h_{\alpha_2}(-1)h_{\alpha_3}(-1). 
\]
Since $\ol H$ is of universal type and $U = \gen{h_{\alpha_1}(-1),
h_{\alpha_3}(-1)}$, it follows from Lemma~\ref{L:liebasic}(2) that
$c^{2^{k-1}} \in E-U$, and hence $c$ has order $2^k$ as claimed.
In particular, this shows $c \in T_S$. 

It remains to show that $w_0, \tau \in S$ in order to complete the proof of
(2). For each $\alpha \in \Sigma$, we have $[n_{\alpha}(\pm 1), \zeta] = 1$ by
\eqref{E:nalphahalpha}, while $[n_{\beta_i}(\pm 1), w_0] = 1$ for $i = 1,2,3$
by \eqref{E:nn} and \eqref{E:bibj}. Also, $h_{\beta_1}(\pm i)h_{\beta_2}(\pm
i)h_{\beta_3}(\pm i) \in E \leq H$ by \eqref{E:proddistinctinvs}. These points
combine to give $w_0 \in H$, $\tau^{\zeta} = \tau$, and $\tau \in \ol S$.  As
$[w_0,\tau] = 1$ by (4) below, we see $\tau \in H$, so indeed $\tau \in H \cap
\ol S = S$. Finally,
\[
n_{\beta_1}(1)n_{\beta_2}(\gamma)n_{\beta_3}(1) = n_{\beta_1}(1)n_{\beta_1}(1)^{n_{\alpha_2+\alpha_3}(1)}n_{\beta_3}(1) \in \ol S
\]
and this element represents the same coset modulo $E$ as $w_0$ does by
\eqref{E:n2} and \eqref{E:nbetainbetaj}. Since $E \leq S$, it follows that $w_0
\in S$.

\medskip
\noindent
\textbf{(3)}:
Since $\{\beta_1,\beta_2,\beta_3\}$ is an orthogonal basis of $V$, the image
$w_{\beta_1}w_{\beta_2}w_{\beta_3}$ in $W$ of $w_0$ acts as minus the identity
on $V$. In particular, it acts as minus the identity on the lattice of coroots
$\ZZ\Sigma^{\vee} \subseteq V$. This implies via Lemma~\ref{L:liebasic}(4) that
$w_0$ inverts $\ol T$, and so (3) holds. 

\medskip
\noindent
\textbf{(4)}:
Showing $[w_0,\tau] = 1$ requires some information about the signs appearing in
our fixed Chevalley presentation. First,
\[
\gen{\beta_1,\alpha_2+\alpha_3} = \frac{2(\alpha_1,\alpha_2+\alpha_3)}{(\alpha_2+\alpha_3,\alpha_2+\alpha_3)} = -2.
\]
So by Lemma~\ref{L:signs}(3), 
\[
c_{\beta_1,\alpha_2+\alpha_3}c_{\beta_2,\alpha_2+\alpha_3} = (-1)^{\gen{\beta_1,\alpha_2+\alpha_3}} = (-1)^{-2} = 1,
\]
and hence $c_{\beta_2, \alpha_2+\alpha_3} = \gamma \overset{\mathrm{def}}{=}
c_{\beta_1, \alpha_2+\alpha_3}$. The root string of $\alpha_2+\alpha_3 = e_2$
through $\beta_3 = e_3$ is $e_3-e_2, e_3, e_3+e_2$ so 
\[
c_{\beta_3,\alpha_2+\alpha_3} = (-1)^1 = -1
\]
by Lemma~\ref{L:signs}(4). So $n_{\alpha_2+\alpha_3}(1)$ inverts each of
$n_{\beta_1}(-\gamma)n_{\beta_2}(1)$ and $n_{\beta_3}(1)$, by \eqref{E:nn}.
Using $[w_0, n_{\alpha_2+\alpha_3}(1)] \in E$, $h_{\beta_1}(\pm
i)h_{\beta_2}(\pm i)h_{\beta_3}(\pm i) \in E$, and \eqref{E:nbetainbetaj}, we
thus have
\begin{align*}
[w_0,\tau] &= [w_0, h_{\beta_1}(-i)h_{\beta_2}(i)h_{\beta_3}(i)][w_0,n_{\alpha_{2}+\alpha_3}(1)]^{h_{\beta_1}(-i)h_{\beta_2}(i)h_{\beta_3}(i)}\\
           &= [w_0, h_{\beta_1}(-i)h_{\beta_2}(i)h_{\beta_3}(i)][w_0,n_{\alpha_{2}+\alpha_3}(1)]\\
           &= [w_0,n_{\alpha_{2}+\alpha_3}(1)]\\
           &= [n_{\beta_1}(-\gamma)n_{\beta_2}(1),n_{\alpha_{2}+\alpha_3}(1)]^{n_{\beta_3}(1)}[n_{\beta_3}(1),n_{\alpha_2+\alpha_3}(1)]\\
           &= (n_{\beta_1}(-\gamma)^2n_{\beta_2}(1)^2)^{n_{\beta_3}(1)}n_{\beta_3}(1)^2\\
           &= n_{\beta_1}(\gamma)^2n_{\beta_2}(-1)^2n_{\beta_3}(1)^2\\
           &= z_1z_1zz\\
           &= 1. 
\end{align*}
Finally, since $[c,n_{\alpha_2+\alpha_3}(1)] = 1$ by \eqref{E:hn}, we have $[c,\tau] = 1$. 
\end{proof}

For any group $X$ and nonnegative integer $r$, write $\EE_r(X)$ for the
elementary abelian subgroups of $X$ of order $2^r$ and $\EE_r(X,Y)$ for the
subset of $\EE_r(X)$ consisting of those members containing the subgroup $Y$.

We next record information about the conjugacy classes and normalizers of four
subgroups containing $Z$.
\begin{lem}
\label{L:omnibus23}
Let $\ol B = N_{\ol H}(U)$ and $B = N_H(U)$.  Write $\ol B^{\circ}$ for the
connected component of $\ol B$. 
\begin{itemize}
\item[(1)]
$\EE_2(\ol{H},Z) = U^{\ol{H}}$, and 
\[
\ol{B} = (\ol{L}_1\ol{L}_2\ol{L}_3)\gen{\tau} \quad \text{ and } \quad \ol B^{\circ} = C_{\ol H}(U) = \ol{L}_1\ol{L}_2\ol{L}_3,
\]
where $\tau$ interchanges $\ol{L}_1$ and $\ol{L}_2$ by conjugation and
centralizes $\ol{L}_3$. Moreover $Z(\ol B^\circ) = U$. 
\item[(2)]
$\EE_2(H,Z) = U^H$, and 
\[
B = (L_1L_2L_3)\gen{c,\tau} \text{ and } C_H(U) = (L_1L_2L_3)\gen{c}, 
\]
where $L_i = C_{\ol{L}_i}(\sigma)$, and where $c \in
N_{\ol{T}}(L_1L_2L_3)$ acts as a diagonal automorphism on each $L_i$.
\end{itemize}
\end{lem}
\begin{proof}
Viewing $\ol H$ classically, an involution in $\ol{H}/Z$ has involutory
preimage in $\ol{H}$ if and only if it has $-1$-eigenspace of dimension $4$ on
the natural orthogonal module (see, for example,
\cite[Lemma~4.2]{AschbacherChermak2010} or
\cite[Lemma~A.4(b)]{LeviOliver2002}). It follows that all noncentral
involutions are $\ol{H}$-conjugate into $U$, and hence that all four subgroups
containing $Z$ are conjugate. Viewing $\ol H$ Lie theoretically gives another
way to see this: let $V$ be a four subgroup of $\ol{H}$ containing $Z$, and let
$v \in V-Z$.  By e.g. \cite[6.4.5(ii)]{Springer2009}, $v$ lies in a maximal
torus, and all maximal tori are conjugate. So we may conjugate in $\ol{H}$ and
take $v \in E$.  Using Lemma~\ref{L:omnibus45}(1) below for example,
$N_{\ol{H}}(\ol{T})/C_{N_{\ol{H}}(\ol{T})}(E) \cong S_4$ acts faithfully on $E$
and centralizes $Z$, so as a subgroup of $GL(E)$ it is the full stabilizer of
the chain $1 < Z < E$. This implies $N_{\ol H}(\ol T)$ acts transitively on the
nonidentity elements of the quotient $E/Z$, so $v$ is $N_{\ol{H}}(T)$-conjugate
into $U$.

We next use Proposition~\ref{P:subtorus-cent} to compute $\ol B$.  Recall that 
\[
U = \gen{z,z_1} = \gen{h_{\alpha_3}(-1), h_{\alpha_1}(-1)}
\]
and that $z = h_{\alpha_3}(-1)$ is central in $\ol H$ by
Lemma~\ref{L:chevrels}(1). So $C_{\ol H}(U) = C_{\ol H}(h_{\alpha_1}(-1))$.  By
Proposition~\ref{P:subtorus-cent} and inspection of $\Sigma$,
\begin{align*}
C_{\ol{H}}(U)^\circ &= \gen{\ol{T}, \ol{X}_\alpha \mid \text{ $\gen{\alpha,\alpha_1}$ is even}}\\ 
                    &= \gen{\ol{T}, \ol{X}_{\pm \alpha} \mid \alpha \in \{\beta_1, \beta_2, \beta_3\}}. 
\end{align*}
Further, $\ol{T} \leq \ol{L}_1\ol{L}_2\ol{L}_3$ by \eqref{E:hbetagenT}, so 
\begin{eqnarray}
\label{E:conncomp}
C_{\ol H}(U)^\circ = \gen{\ol X_{\beta_i}, \ol X_{-\beta_i} \mid i \in \{1,2,3\}} = \ol L_1\ol L_2\ol L_3
\end{eqnarray}
as claimed.  

We next prove that $C_{\ol{H}}(U)$ is connected. Since $C_{\ol H}(U) = C_{\ol
H}(z_1)$, this follows directly from a theorem of Steinberg to the effect that
the centralizer of a semisimple element in a simply connected reductive group
is connected, but it is possible to give a more direct argument in this special
case. By Proposition~\ref{P:subtorus-cent},
\[
C_{\ol{H}}(U) = C_{\ol{H}}(U)^{\circ}C_{N_{\ol{H}}(\ol{T})}(U),
\]
and we claim that $C_{N_{\ol{H}}(\ol{T})}(U) \leq C_{\ol{H}}(U)^\circ$.  By
\eqref{E:conncomp}, $N_{C_{\ol{H}}(U)^{\circ}}(\ol{T})/\ol{T}$ is elementary
abelian of order $8$. On the other hand, $C_{N_{\ol{H}}(\ol{T})}(U)/\ol{T}$
stabilizes the flag $1 < Z < U < E$, and so induces a group of transvections on
$E$ of order $4$ with center $Z$ and axis $U$.  The element $w_0$ of
$N_{\ol{H}}(\ol{T})$ inverts $\ol{T}$ and is trivial on $E$ by
Lemma~\ref{L:chevrels}(3). It follows that $|N_{C_{\ol{H}}(U)}(\ol{T})/\ol{T}|
= |N_{C_{\ol{H}}(U)^{\circ}}(\ol{T})/\ol{T}|$, and so $N_{C_{\ol H}(U)}(\ol T)
= N_{C_{\ol H}(U)^\circ}(\ol T)$.  Thus, 
\[
C_{N_{\ol{H}}(\ol{T})}(U) = N_{C_{\ol H}(U)}(\ol T) = N_{C_{\ol H}(U)^{\circ}}(\ol T) \leq C_{\ol{H}}(U)^\circ,
\]
completing the proof of the claim. By \eqref{E:conncomp}
\begin{eqnarray}
\label{E:CHU}
C_{\ol H}(U) = \ol L_1 \ol L_2 \ol L_3.
\end{eqnarray}

For each $\lambda \in \ol\FF_p$, we have
\begin{align*}
x_{\beta_3}(\lambda)^\tau &= x_{\beta_3}(\lambda)^{n_{\alpha_2+\alpha_3}(1)h_{\beta_1}(-i)h_{\beta_2}(i)h_{\beta_3}(i)}\\
                          &= x_{\beta_3}(-\lambda)^{h_{\beta_3}(i)}\\
                          &= x_{\beta_3}(i^{\gen{\beta_3,\beta_3}}(-\lambda))\\
                          &= x_{\beta_3}(i^2(-\lambda))\\
                          &= x_{\beta_3}(\lambda)
\end{align*}
Similarly, $x_{-\beta_3}(\lambda)^\tau = x_{-\beta_3}(i^{-2}(-\lambda)) =
x_{-\beta_3}(\lambda)$. So as $\ol L_3 = \gen{x_{\pm \beta_3}(\lambda)}$, we
have $[\ol L_3, \tau] = 1$.  Finally, since $w_{\alpha_2+\alpha_3}$
interchanges $\beta_1$ and $\beta_2$, and since $\ol T$ normalizes all root
groups, $\tau$ interchanges $\ol L_1$ and $\ol L_2$. In particular, $\tau$
interchanges the central involutions $h_{\beta_1}(-1) = z_1$ and
$h_{\beta_2}(-1) = zz_1$ of $\ol L_1$ and $\ol L_2$. This shows $\tau$ acts
nontrivially on $U$, and hence
\[
\ol B = (\ol{L}_1\ol{L}_2\ol{L}_3)\gen{\tau},
\]
completing the proof of (1). 

By (1), $C_{\ol H}(U)$ is connected, so \cite[Theorem~2.1.5]{GLS3} applies to
give $\EE_2(H,Z) = U^H$.  Let $L_i = C_{\ol{L}_i}(\sigma)$ for $i = 1,2,3$, and
set $B^\circ = L_1L_2L_3 \leq H$.  Since $w \in N_H(U)-C_H(U)$, we have
$C_{H}(U) = C_{\ol{B}^{\circ}}(\sigma)$. Let $\tilde{B}$ denote the direct
product of the $\ol{L}_i$, and let $\tilde{\sigma}$ be the Steinberg
endomorphism lifting $\sigma|_{\ol{B}^\circ}$ along the isogeny $\tilde{B} \to
\ol{B}^\circ$ given by quotienting by $\gen{(-1,-1,-1)}$  (see, e.g.
\cite[Lemma~2.1.2(d,e)]{GLS3}). Then $C_{\tilde{B}}(\tilde{\sigma}) = L_1
\times L_2 \times L_3$. So by \cite[Theorem~2.1.8]{GLS3} applied with the pair
$\tilde{B}$, $\gen{(-1,-1,-1)}$ in the role of $\ol K$, $\ol Z$, we see that
$B^\circ$ is of index $2$ in $C_H(U)$ with $C_H(U) = B^\circ(C_H(U) \cap \ol T)
= B^\circ T$. The element $c = h_{\beta_1}(\mu)h_{\beta_2}(\mu)h_{\beta_3}(\mu)
\in T$ lifts to an element $\tilde c \in \tilde B$ with $[\tilde c, \tilde
\sigma] = (-1,-1,-1)$ by definition of $\mu$, and so $c \in C_H(U)-B^\circ$ by
\cite[Theorem~2.1.8]{GLS3}. Finally as each $L_i$ is generated by root groups
on which $c$ acts nontrivially, $c$ acts as a diagonal automorphism on each
$L_i$. 
\end{proof}

Next we consider the $H$-conjugacy classes of elementary abelian subgroups of
order $8$ which contain $Z$. 

\begin{lem}\label{L:omnibus45}
The following hold. 
\begin{itemize}
\item[(1)] $N_{\ol{H}}(E) = N_{\ol{H}}(\ol{T})$ and $C_{\ol{H}}(E) = \ol{T}\gen{w_0}$. 
\item[(2)] $N_{H}(E) = N_H(T)$ and $C_{H}(E) = T\gen{w_0}$. 
\item[(3)] $N_{\ol H}(\ol T)/\ol T \cong C_2 \times S_4 \cong N_H(T)/T$. 
\end{itemize}
\end{lem}
\begin{proof}
Given that $w_0$ inverts $\ol T$ (Lemma~\ref{L:chevrels}(3)) part (1) is proved
in Proposition~\ref{P:faithful}.  

By (1),
\[
N_H(E) = N_{\ol H}(E) \cap H = N_{\ol H}(\ol T) \cap H = N_H(\ol T),
\]
while $N_H(\ol T) \leq N_H(H \cap \ol T) = N_H(T)$. These combine to show the
inclusion $N_{H}(E) \leq N_H(T)$. But $N_H(T) \leq N_H(E)$ since $E =
\Omega_1(O_2(T))$ is characteristic in $T$. Next, by
(1), 
\[
C_H(E) = C_{\ol H}(E) \cap H = \ol T\gen{w_0} \cap H = (\ol T \cap H)\gen{w_0}
\]
with the last equality as $w_0 \in H$ by Lemma~\ref{L:chevrels}(2). This shows
$C_H(E) = T\gen{w_0}$. 

For part (3) in the case of $\ol H$, see Section~\ref{SS:normtorus}.  We show
part (3) for $H$. First, by (1) and (2), 
\begin{eqnarray}
\label{E:12cons}
N_{H}(\ol T) = C_{N_{\ol H}(\ol T)}(\sigma) = C_{N_{\ol H}(E)}(\sigma) = N_{H}(E) = N_H(T). 
\end{eqnarray}
In the special case $\epsilon = 1$, $\sigma$ centralizes $\widehat{W}$, which
covers $W = N_{\ol H}(\ol T)/\ol T$.  Using \eqref{E:12cons}, this shows
$N_{H}(T) = N_{H}(\ol T)$ projects onto $W$ with kernel $\ol T \cap C_{N_{\ol
H}(\ol T)}(\sigma) = T$.  So $N_H(T)/T \cong W$ in this case. 

In any case, $\ol Tw_0$ generates the center of $N_{\ol H}(\ol T)/\ol T$, so
$g^{\sigma}g^{-1} \in \ol T$ for each $g \in N_{\ol H}(\ol T)$. Since $\ol T$
is connected, for each such $g$ there is $t \in \ol T$ with $t^{-\sigma}t =
g^{\sigma}g^{-1}$ by the Lang-Steinberg Theorem, and hence $tg \in C_{N_{\ol
H}(\ol T)}(\sigma)$. This shows each coset $\ol Tg$ contains an element
centralized by $\sigma$, and so arguing as in the previous paragraph, we have
$N_H(T)/T \cong W$. 
\end{proof}

\begin{lem}\label{L:EE'}
Let $d = cw_0$ and $E' = U\gen{d} \leq S$.  Then $\EE_3(H,Z)$ is the disjoint
union of $E^H$ and $E'^H$.  Moreover, there is a $\sigma$-invariant maximal
torus $T'$ of $H$ with $E' = \{t \in T' \mid t^2 = 1\}$ such that the following
hold. 
\begin{itemize}
\item[(1)] $O_{2'}(C_H(E)) = O_{2'}(T) \cong (C_{(q-\epsilon)/2^k})^3$, and
$N_H(T)/T \cong C_2 \times S_4$ acts faithfully on the $r$-torsion subgroup of
$T$ for each odd prime $r$ dividing $q-\epsilon$; 
\item[(2)] $O_{2'}(C_H(E')) = O_{2'}(T') \cong (C_{(q+\epsilon)/2})^3$, and
$N_H(T')/T' \cong C_2 \times S_4$ acts faithfully on the $r$-torsion subgroup
of $T'$ for each odd prime $r$ dividing $q+\epsilon$; and 
\item[(3)] $C_{H}(E') = T'\gen{w_0'}$ for some involution $w_0'$ inverting $T'$.
\end{itemize}
\end{lem}
\begin{proof}
By Lemma~\ref{L:chevrels}, $w_0$ is an involution inverting $\ol T$ and hence
inverting $c$. So $d$ is an involution, and indeed, $E'$ is elementary abelian
of order $8$. 

Part of this Lemma is proved by Aschbacher and Chermak
\cite[Lemma~7.8]{AschbacherChermak2010}. We give an essentially complete proof
for the convenience of the reader. Let $\ol{X} \in \{\ol{B}^{\circ}, \ol{H}\}$,
and write $X = C_{\ol{X}}(\sigma)$. The centralizer $C_{\ol{X}}(E) =
\ol{T}\gen{w_0}$ is not connected, but has the two connected components $\ol T$
and $\ol Tw_0$. Thus, there are two $C_{\ol{X}}(\sigma)$-conjugacy
classes of subgroups of $X$ conjugate to $E$ in $\ol{H}$ \cite[2.1.5]{GLS3}. A
representative of the other $X$-class can be obtained as follows.  Since $\ol
X$ is connected, we may fix by the Lang-Steinberg Theorem $g \in \ol{X}$ such
that $w_0 = g^{\sigma}g^{-1}$.  Then $g^\sigma = w_0g$. In the semidirect
product $\ol{X}\gen{\sigma}$, we have $\sigma^g = \sigma w_0$.  Now as
$\ol{T}\gen{w_0}$ is invariant under $\sigma w_0$, it follows that
$(\ol{T}\gen{w_0})^g$ is $\sigma$-invariant. Indeed by choice of $g$, we have
$t^{g\sigma} = t^{\sigma w_0 g}$ for each $t \in \ol{T}$, i.e., the conjugation
isomorphism $\ol{T}\gen{w_0} \xto{c_g} \ol{T}^g\gen{w_0^g}$ intertwines the
actions of $\sigma w_0$ on $\ol{T}\gen{w_0}$ and $\sigma$ on
$\ol{T}^g\gen{w_0^g}$. Then $E$ and $E^g$ are representatives for the
$X$-classes of subgroups of $X$ conjugate in $\ol{X}$ to $E$, and
\begin{eqnarray} 
\label{E:tori} 
X \cap \ol{T}^{g} = C_{\ol{T}^g}(\sigma) \cong C_{\ol{T}}(\sigma w_0) = \{t \in
T \mid t^{-\epsilon q} = t\} \cong (C_{q+\epsilon})^3.  
\end{eqnarray} 

The above argument shows we may take $g \in \ol{B}^{\circ}$ even when $\ol{X} =
\ol{H}$. By Lemma~\ref{L:omnibus23}, $\ol{B}^\circ$ is a commuting product
$\ol{L}_1\ol{L}_2\ol{L}_3$ with $\ol{L}_i \cong SL_2(\ol{\FF}_p)$ and
$Z(\ol{B}^{\circ}) = U$. Also, $\ol{B}^\circ \cong \ol{J}/\gen{j}$ where
$\ol{J}$ is a direct product of the $\ol{L}_i$'s and $j$ the product of the
unique involutions of the direct factors (Section~\ref{SS:SL2}). Thus, each
involution in $\ol{B}^\circ-U$ is of the form $f_1f_2f_3$ for elements $f_i \in
\ol{L}_i$ of order $4$. But $\ol{L}_i$ is transitive on its elements of order
$4$. Hence, all elementary abelian subgroups of $\ol{B}^{\circ}$ of order $8$
containing $U$ are $\ol{B}^{\circ}$-conjugate. Now $E$ is contained in the
normal subgroup $L_1L_2L_3$ of $C_H(U)$, while $E'$ is not since $d$ lies in
the coset $L_1L_2L_3c$. It follows that $E^g$ is $C_{H}(U)$-conjugate to $E'$.
Hence, $E$ and $E'$ are representatives for the $X$-conjugacy classes of
elementary abelian subgroups of $X$ of order $8$ containing $Z$. 

Fix $b \in C_H(U)$ with $E^{gb} = E'$. Set $\ol T' = \ol T^{gb}$, $T' = C_{\ol
T^{gb}}(\sigma)$, and $w_{0}' = w_0^{gb}$. By \eqref{E:tori}, $O_{2'}(T')$ is
as described in (a)(ii), and $w_0'$ inverts $T'$. Now $N_H(T)/T \cong C_{2}
\times S_4$ by Lemma~\ref{L:omnibus45}(3).  Since $\ol Tw_0$ generates the
center of $N_{\ol{H}}(\ol{T})/\ol T$, it follows by choice of $g$ and
\cite[3.3.6]{Carter1985} that $N_H(T^g)/T^g \cong N_H(T)/T$, and hence
$N_H(T')/T' \cong N_H(T)/T$ because $b \in H$. 

Fix an odd prime $r$ dividing $q-\epsilon$ (resp. $q+\epsilon$), and let $T_r$
(resp. $T_r'$) be the $r$-torsion subgroup of $\ol T$ (resp. $\ol T'$).
Then $T_r \leq T$ (resp. $T_r' \leq T'$).  Since $N_{\ol H}(\ol T)/\ol T$
(resp. $N_{\ol H}(\ol T')/\ol T'$) acts faithfully on $T_r$ (resp.
$T_r'$) by Proposition~\ref{P:faithful}, it follows that the same is true for
$N_H(T)/T$ (resp.  $N_H(T')/T'$). This completes the proof of (1) and (2), and
part (3) then follows. 
\end{proof}

\subsection{Conjugacy classes of elementary abelian subgroups in a
Benson-Solomon system}\label{SS:solelab}

In this subsection we look at the conjugacy classes and automizers of
elementary abelian subgroups of the Benson-Solomon systems. We adopt the
notation from the first part of this section, so $S$ is a Sylow $2$-subgroup of
$H = \Spin_7(q)$, $Z = Z(S)$ is of order $2$, $U$ is the unique normal
four subgroup of $S$, and $E$ is the $2$-torsion in the fixed maximal torus
$T$ of $H$, and $A = E\gen{w_0}$. 

\begin{lem}\label{L:solomnibus} 
Let $\F = \F_{\Sol}(q)$ be a Benson-Solomon fusion system over $S$.  Then 
\begin{itemize}
\item[(1)] $\EE_1(S) = Z^\F$, and $N_\F(Z) = C_{\F}(Z) \cong \F_S(H)$.
\item[(2)] $\EE_2(S)=U^\F$.
\item[(3)] For $T_S = T \cap S$, $\Out_\F(T_S) = \Aut_\F(T_S) = C_2 \times
GL_3(2)$, and $\Out_\F(T_S\gen{w_0}) \cong GL_3(2)$ acts naturally on
$T_S/\Phi(T_S)$ and on $E$. 
\end{itemize}
\end{lem}
\begin{proof}
Part (1) follows from the construction of $\F_{\Sol}(q)$.  By part (1)
and \cite[Lemma~II.3.1]{AschbacherKessarOliver2011}, every element of
$\EE_2(S)$ is $\F$-conjugate to a subgroup containing $Z$ and thus by
Lemma~\ref{L:omnibus23}(2) to $U$. 

The structure of $\Out_\F(T_S)$ in (3) follows from the construction of
the Benson-Solomon systems; for example, see \cite[Proposition~5.4(b),
Lemma~7.13(e)]{AschbacherChermak2010} or
\cite[Proposition~1.5]{LeviOliver2005}.  The structure of
$\Out_\F(T_S\gen{w_0})$ follows from that of $\Out_\F(T_S)$; for the details we
refer the reader to \cite[Lemma~2.38(c)]{HenkeLynd2022}. Since the actions in
(3) are induced by the restriction map $\Aut_\F(T_S\gen{w_0}) \to
\Aut_{\F}(T_S)$, the remainder of (3) is clear.
\end{proof}

We saw in Lemma~\ref{L:EE'} that $H$ has two conjugacy classes of elementary
abelian subgroups of order $8$ containing $Z$. As far as we can tell,
Aschbacher and Chermak do not discuss the possible $\F$-conjugacy of $E$ and
$E'$ explicitly, but such information can be deduced from their
description of the conjugacy classes of elementary abelian subgroups of order
$16$. Since we need to show later in Lemma~\ref{L:EE'F} that $E$ and
$E'$ are in fact not $\F$-conjugate, we provide an account of that description.

On p.935 of \cite{AschbacherChermak2010}, $T_S$ is denoted $R_0$. As on
pg.935-936, write $R_1 = N_{\ol{T}}(T_S\gen{w_0}) \cong (C_{2^{k+1}})^3$.
Thus, $T_S$ has index $8$ in $R_1$, and $R_1/T_S$ is elementary abelian of
order $8$. Fix a set
\[
\{x_e \mid e \in E\}
\]
of coset representatives for $T_S$ in $R_1$, with notation chosen so that $x_1
= 1$ and $x_e^{2^k} = e \in E$ for each $e \in E-\{1\}$, and set
\[
A_e = A^{x_e}.
\]
Since $w_0$ inverts $\ol{T}$, we have $A_e = E\gen{t_ew_0}$ where $t_e
:= x_e^{-2} = [x_e,w_0] \in T_S$ also powers to $e$. 

Denote by $\A$ the set of $T_S$-conjugacy classes of elementary abelian
subgroups of $T_S\gen{w_0}$ of order $16$. Then since $E \leq T_S$ and
$[T_S,w_0] = \Phi(T_S)$, there are $\Aut_{\F}(T_S\gen{w_0})$-equivariant
bijections
\begin{align*}
\A &\longrightarrow T_S/\Phi(T_S) \longrightarrow E\\
A_e^{T_S} &\longmapsto t_e\Phi(T_S) \longmapsto e = t_e^{2^{k-1}}.
\end{align*}
Since $\Inn(T_S\gen{w_0})$ acts trivially on these sets,  by
Lemma~\ref{L:solomnibus}(3), $\Aut_\F(T_S\gen{w_0})$ has two orbits on
$\mathscr{E}_4(T_S\gen{w_0})$ with representatives $A = A_1$ and $A_e$ with $e
\neq 1$.

\begin{lem}
\label{L:AC712}
$\EE_4(S)$ is the disjoint union of $A_1^\F$ and $A_e^\F$, where $e$ is any
nonidentity element of $E$. All $A_e$ with $e \neq 1$ are
$\Aut_\F(T_S\gen{w_0})$-conjugate, and $\Aut_\F(A_e) = C_{\Aut(A_e)}(e)$ for
each $e \in E$. 
\end{lem}
\begin{proof}
This is Lemma~7.12(c) of \cite{AschbacherChermak2010}, except for the statement
on $\Aut_\F(T_S\gen{w_0})$-conjugacy, which contained in the proof of
7.12(c) and observed above. An equivalent description of the
$\F$-conjugacy classes of elementary abelian subgroups of $S$ of rank $4$ was
given in \cite[Proposition~A.8, Lemma~3.1]{LeviOliver2002}. See also
\cite[p.3018]{LeviOliver2005}.
\end{proof}

\begin{lem}\label{L:EE'F}
$\EE_{3}(S)$ is the disjoint union of $E^\F$ and $E'^{\F}$, and we have
$\Aut_\F(E) = \Aut(E)$ and $\Aut_\F(E')=\Aut(E')$. 
\end{lem}
\begin{proof}
In Lemma~\ref{L:EE'} we defined $E' = U\gen{d}$ where $d = cw_0$.  It was
shown there that $E$ and $E'$ are representatives for the two $H$-conjugacy
classes of elementary abelian subgroups of order $8$ containing $Z$.  Recall
that $c$ was defined as $h_{\beta_1}(\mu)h_{\beta_2}(\mu)h_{\beta_3}(\mu)$ at
the end of Section~\ref{SS:Spin_7(q)}, and in Lemma~\ref{L:chevrels}(2) it was
shown that $c \in T_S$ and $c^{2^{k-1}} \in E-U$. In particular, $c \in
T_S-\Phi(T_S)$. 

Take $e = c^{2^{k-1}}$ and consider $A_e = A^{x_e} = E\gen{t_ew_0}$. Since both
$t_e = [x_e,w_0]$ and $c$ have $2^{k-1}$st power $e$, there is $s \in \Phi(T_S)
= \mho^1(T_S)$ with $c = t_es$. Choose $t \in T_S$ with $t^{-2} = s$. Then
$A_{e}^t = E^t\gen{(t_ew_0)^t} = E\gen{(t_ew_0)^t}$, and $(t_ew_0)^t = t_e
w_0^t = t_e [t,w_0]w_0 = t_et^{-2}w_0 = t_esw_0 = cw_0$. This shows $A_e$ is
$T_S$-conjugate to the elementary abelian subgroup 
\begin{equation}
\label{E:E'leA'}
A' := E\gen{c w_0} = EE'
\end{equation}
of order $16$. Alternatively, we could have chosen the coset representative
$x_e$ at the outset to satisfy $[x_e,w_0] = x_e^{-2} = c$ and in doing so
arrange for $A_e = A'$, and thus for $A_e$ to contain $E'$. 

Assume to get a contradiction that $E$ and $E'$ are $\F_{\Sol}(q)$-conjugate.
Since $E$ is normal in $S$, it is fully $\F$-normalized, hence fully
$\F$-centralized by saturation. So there is a morphism $\phi\colon C_S(E') \to
C_S(E)$ in $\F$ with $E'^{\phi} = E$ by
\cite[I.2.6]{AschbacherKessarOliver2011}.  By \eqref{E:E'leA'}, $\phi$ is
defined on $A'$ and $A'^{\phi} \leq C_S(E) = T_S\gen{w_0}$.  By
Lemma~\ref{L:AC712}, we may choose $\alpha \in \Aut_{\F}(T_S\gen{w_0})$ with
$A'^{\phi\alpha} = A'$. Then as $\phi\alpha \in \Aut_{\F}(A') =
C_{\Aut_\F(A')}(e)$ by the same lemma, we have $e^{\phi\alpha} = e$.  On the
other hand, since $E$ is characteristic in $T_S\gen{w_0}$, we have
$U^{\phi\alpha} \leq E'^{\phi\alpha} = E^\alpha = E$. Thus, $E^{\phi\alpha} =
(U\gen{e})^{\phi\alpha} = U^{\phi\alpha}\gen{e^{\phi\alpha}} \leq E$, a
contradiction. Now we appeal to \cite[Lemma~7.8]{AschbacherChermak2010} for
the structure of the $\F$-automorphism groups. 
\end{proof}

\begin{remark}\label{R:LO3.1}
Lemma~\ref{L:EE'F} says that there are two conjugacy classes of
elementary abelian subgroups of order eight in a Benson-Solomon system, and is
therefore incompatible with the part of Lemma~3.1 of \cite{LeviOliver2002}
which states that there is a single conjugacy class. We thank the referee for
alerting us to this, and we refer the reader to \cite{OliverLOcorr}.
\end{remark}

\subsection{Proof of Theorem~\ref{T:solq}}

We now turn to the proof of Theorem~\ref{T:solq}. As an initial observation,
note that if $\L$ is a punctured group for $\F_{\Sol}(q')$ for some odd prime
power $q'$, then $C_{\L}(Z)$ is a group whose $2$-fusion system is isomorphic
to that of $\Spin_7(q')$. It is in this context that we use the following
lemma.

By a field automorphism of $\Spin_7(q)$ we mean an automorphism acting on the
root groups via $x_\alpha(\lambda) \mapsto x_\alpha(\lambda^\psi)$ where $\psi$
is an automorphism of $\mathbb{F}_q$.

\begin{lem}\label{L:spin7reductions}
Let $G$ be a finite group whose $2$-fusion system is isomorphic to that of
$\Spin_7(q')$ for some odd $q'$. Then $G/O_{2'}(G) \cong \Spin_7(q)\gen{\phi}$
for some odd $q$ with $v_2(q^2-1)=v_2(q'^2-1)$, and where $\phi$ induces a
field automorphism of odd order.
\end{lem}

\begin{proof}
It was shown by Levi and Oliver in the course of proving $\F_{\Sol}(q)$ is
exotic that $O^{2'}(G/O_{2'}(G))$ is isomorphic to $\Spin_7(q)$ for some odd
$q$ \cite[Proposition~3.4]{LeviOliver2002}. If $S'$ and $S$ are the
corresponding Sylow $2$-subgroups, then $S'$ and $S$ are isomorphic by
definition of an isomorphism of a fusion system. If $k$ and $k'$ are one less
than the valuations of $q^2-1$ and $q'^2-1$, then the orders of $S$ and $S'$
are $2^{4+3k}$ and $2^{4+3k'}$, so $k = k'$. The description of $G/O_{2'}(G)$
follows, since $\Out(\Spin_7(q)) \cong C_a \times C_2$, where $q = p^a$ and
$C_a$ is generated by the class of a field automorphism.
\end{proof}

The extension of $\Spin_7(q)$ by a group of field automorphisms of odd order
has the same $2$-fusion system as $\Spin_7(q)$, but we will not need this. 

\begin{lem}\label{L:num}
Let $q$ be an odd prime power with the property that $GL_3(2)$ has a faithful
$3$-dimensional representation over $\FF_r$ for each prime divisor $r$ of
$q^2-1$. Then each such $r$ is a square modulo $7$, and $q = 3^{1+6a}$ for some
$a \geq 0$. In particular, $q \equiv 3 \pmod{8}$. 
\end{lem}
\begin{proof}
Set $G = GL_3(2)$ for short. We first show that $GL_3(2)$ has a faithful
$3$-dimensional representation over $\FF_r$ if and only if $r$ is a square
modulo $7$.  If $r = 2$, $3$, or $7$, then as $|SL_3(3)|$ is not divisible by
$7$ and $G \cong PSL_2(7) \cong \Omega_3(7)$, the statement holds. So we may
assume that $p$ does not divide $|G|$, so that $\mathbb{F}_rGL_3(2)$ is
semisimple. Let $V$ be a faithful $3$-dimensional module with character $\phi$,
necessarily irreducible. From the character table for $GL_3(2)$, we see that
$\phi$ takes values in $\mathbb{F}_r((1+\sqrt{-7})/2)$.  By
\cite[I.19.3]{Feit1982}, a modular representation is writable over its field of
character values, so this extension is a splitting field for $V$. Thus, $V$ is
writeable over $\FF_r$ if and only if $-7$ is a square modulo $r$, which by
quadratic reciprocity is the case if and only if $r$ is a square modulo $7$. 

Now fix an odd prime power $q$ with the property that $q^2-1$ is divisible only
by primes which are squares modulo $7$.  Since $q(q-1)(q+1)$ is divisible by
$3$ and $3$ is not a square, we have $q = 3^l$ for some $l$. Now $q-1$ and
$q+1$ are squares, so $q \equiv 1$ or $3 \pmod{7}$. Assume the former.  Then $6$
divides $l$, so $q = 3^l \equiv \pm 1 \pmod{5}$.  But then $q^2-1$ is divisible
by the nonsquare $5$, a contradiction. So $q \equiv 3 \pmod{7}$, $l = 1+6a$ for
some $a \geq 0$, and hence $q \equiv 3 \pmod{8}$.
\end{proof}

We write $G'=[G,G]$ for the commutator subgroup of a group $G$.

\begin{lem}\label{L:GSplit}
Let $G$ be a finite group and $W\unlhd M\unlhd G$ such that $G/M\cong GL_3(2)$,
$|W|=2$ and $M/W$ is cyclic of odd order. Then $E(G)=G'$ is isomorphic to
$GL_3(2)$ or $L_2(7)$. In the latter case, $G$ has quaternion Sylow
$2$-subgroups.
\end{lem}

\begin{proof}
If we pick a generator $xW$ of $M/W$, then $M=\gen{x}W$ is abelian as $W\leq
Z(M)$. Since $M/W$ is of odd order and $|W|=2$, it follows indeed that $M$ is
cyclic. In particular, $\Aut(M)$ is abelian. Since $C_G(F^*(G))\leq F^*(G)$ and
$G/M$ is non-abelian simple, it follows that $F^*(G)\neq M=F(G)$ and so
$E(G)\neq 1$. As $GL_3(2)$ is the only non-abelian composition factor in a
composition series for $G$, it follows that $K:=E(G)$ is quasisimple with
$K/Z(K)\cong GL_3(2)\cong L_2(7)$. In particular, $G=KM$ where $M=F(G)$ is
abelian and commutes with $K=E(G)=K'$. Hence, $G'=K=E(G)$.

\smallskip

Since the Schur multiplier of $GL_3(2)\cong L_2(7)$ has order $2$, we have
$K\cong GL_3(2)$ or $K\cong SL_2(7)$. In the latter case, $Z(K)=W$ and $K$
contains a Sylow $2$-subgroup of $G$. As $SL_2(7)$ has quaternion Sylow
$2$-subgroups, the assertion follows.
\end{proof}

If $\widehat{\Theta}$ is a partial normal $p^\prime$-subgroup of a locality
$(\L,\Delta,S)$ at the prime $p$, then the restriction of the natural
projection $\L\rightarrow \L/\widehat{\Theta}$ to $S$ is a monomorphism. Thus, we
may identify $S$ with its image in $\L/\widehat{\Theta}$. This is used to formulate
the following proposition.

\begin{prop}\label{P:StructureNLZ}
Suppose $\F=\F_{\Sol}(q')$ is a Benson--Solomon fusion system and
$(\L,\Delta,S)$ is a punctured group over $\F$. Set $Z:=Z(S)$ and $G=C_\L(Z)$.
\begin{itemize}
\item [(1)] There exists a partial normal $2^\prime$-subgroup $\widehat{\Theta}$ of
$\L$ such that, identifying $S$ with its image in $\L/\widehat{\Theta}$ in the
natural way, $(\L/\widehat{\Theta},\Delta,S)$ is a punctured group over $\F$ and
$N_{\L/\widehat{\Theta}}(Z)\cong G/O_{2^\prime}(G)$ is $2^\prime$-reduced.  \item
[(2)] $O^{2^\prime}(G/O_{2^\prime}(G))\cong \Spin_7(q)$, where $q=3^{1+6a}$ for
some $a\geq 0$, and every prime divisor of $q^2-1$ is a square
modulo $7$. In particular, $q \equiv 3\pmod{8}$ and $q' \equiv \pm 3\pmod{8}$.
\end{itemize}
\end{prop}

\begin{proof}
Set $\ov{G}:=G/O_{2^\prime}(G)$ and $H:=O^{2^\prime}(\ov{G})$. As $\L$ is a
locality on $\F=\F_{\Sol}(q')$, it follows from Lemma~\ref{L:NFPNLP}(b) and
Lemma~\ref{L:solomnibus}(1) that $\F_{\ov{S}}(\ov{G})\cong\F_S(G) =
\F_S(N_\L(Z))=N_\F(Z)$ is isomorphic to the $2$-fusion system of $\Spin_7(q')$.
Hence, by Lemma~\ref{L:spin7reductions} there is an odd prime power $q$ with
$(q^2-1)_2 = (q'^2-1)_2$ such that $H$ can be identified with $\Spin_7(q)$, and
such that $\bar{G} = H\gen{\phi}$ for $\phi$ a field automorphism of odd order.
Where convenient we adopt below the notation from
Sections~\ref{SS:notation}-\ref{SS:solelab}. 

\smallskip

\textbf{(1)} We will show the existence of a suitable signalizer functor on
elements of order $2$ (as introduced in Definition~\ref{D:SignalizerFunctor}).
For that we set
\[
\theta(a) = O_{2'}(C_{\L}(a))\mbox{ for each involution }a \in S.
\]
By Lemma~\ref{L:LocalitiesProp}(b), $\theta$ is conjugacy invariant. Let $a,b
\in S$ be two distinct commuting involutions.  By Lemma~\ref{L:solomnibus}(2)
and conjugacy invariance, to verify the balance condition in
Definition~\ref{D:SignalizerFunctor}, we can assume $b = z$ and $a = u \in
U-Z$. Set $X = O_{2'}(C_\L(u)) \cap G$ and note that $X$ is an odd order normal
subgroup of $C_\L(U)=C_G(U)$. By a Frattini argument $\ol{C_{G}(U)} =
C_{\bar{G}}(\bar{U})$, so $\bar{X}$ is normal in the latter group. We use now
that $\bar{G}$ is an extension of $H=\Spin_7(q)$ by a cyclic group generated by
a field automorphism $\phi$ of odd order. Since each component $L_i \cong
SL_2(q)$ of $C_H(\bar{U})$ is generated by a root group and its opposite
(Section~\ref{SS:SL2} and Lemma~\ref{L:omnibus23}(2)), it follows that $\phi$
acts nontrivially as a field automorphism on each such $L_i$, and hence
$\bar{X} \leq O_{2'}(C_{\bar{G}}(\bar{U})) \leq O_{2'}(L_1L_2L_3) = 1$.
Equivalently, $X = O_{2'}(C_\L(u)) \cap G \leq O_{2'}(G)$. This shows that the
balance condition holds. For each $P \in \Delta$, set
\[
\Theta(P) = \left(\bigcap_{x \in \I_2(P)} \theta(x)\right) \cap C_\L(P).
\]
Then by Theorem~\ref{T:sigfun}, $\Theta$ defines a signalizer functor on
objects. By Theorem~\ref{P:GetLocalityObjectiveCharp}, $\widehat{\Theta} =
\bigcup_{P \in \Delta} \Theta(P)$ is a partial normal $p^\prime$-subgroup of
$\L$, and $(\L/\widehat{\Theta},\Delta,S)$ is again a punctured group for
$\F=\F_{\Sol}(q')$ with $N_{\L/\widehat{\Theta}}(Z)\cong G/\Theta(Z)$. Writing
$Z=\gen{z}$, note that $\I_2(Z)=\{z\}$ and $\Theta(Z)=\theta(z)\cap
C_\L(Z)=O_{2^\prime}(C_\L(z))=O_{2^\prime}(G)$. This proves (1).

\smallskip

\textbf{(2)} For the proof of (2), part (1) allows us to assume
$O_{2^\prime}(G)=1$. Then $G=\ov{G}$, $H=O^{2^\prime}(G)\cong\Spin_7(q)$, and
$G/H$ is cyclic of odd order.  Recall that $H$ has Sylow $2$-subgroup $S$,
$\epsilon \in \{\pm 1\}$ is such that $q \equiv \epsilon \pmod{4}$, and $E_1 :=
E$ and $E_{-1} := E'$ are the representatives for $\F$-conjugacy classes of
elementary abelian subgroups of order $8$ in $S$ (Lemmas~\ref{L:EE'} and
\ref{L:EE'F}).  For $\delta = \pm 1$, let $T_{\delta}$ be the maximal torus
containing $E_{\delta}$ of Lemma~\ref{L:EE'}. For each positive integer $r$
dividing $q-\delta\epsilon$, write $T_{\delta,r}$ for the $r$-torsion in
$T_{\delta}$.  Moreover, set $T_{\delta,S} = T_\delta \cap S$.  Thus, $T_{1,S}
= T_S = T_{1,2^k}$ is homocyclic of order $2^{3k}$, and $T_{-1,S} = E_{-1}$.

Now fix $\delta$ and let $N = N_\L(T_{\delta,S})$. By
Lemmas~\ref{L:omnibus45}(2) and \ref{L:EE'}(3),
\begin{eqnarray}\label{E:CHEdelta}
C_H(E_\delta) = T_\delta\gen{w},
\end{eqnarray}
where $w$ is an involution inverting $T_\delta$. In particular, since
\[
O_{2'}(T_\delta) = [O_{2'}(T_\delta),\gen{w}]
\]
and $O^{2'}(G) = H$, we have
\[
C_H(E_{\delta}) = O^{2'}(C_H(E_\delta)) = O^{2'}(C_G(E_\delta)).\]
Also, $C_\L(E_\delta) = C_G(E_\delta)$ as $E_\delta$ contains $Z$. It follows
that $C_H(E_\delta) = O^{2'}(C_{\L}(E_\delta))$ is normal in
$N_{\L}(E_{\delta})$, so
\begin{eqnarray}
\label{E:oddnormal}
\text{$C_H(E_\delta)$ and $O_{2'}(C_H(E_\delta))$ are normal in $N_\L(E_\delta)$.}
\end{eqnarray}

Next we show
\begin{eqnarray}
\label{E:N=NL(E)}
N = N_\L(E_\delta).
\end{eqnarray}
We may assume $T_{\delta,S} > E_\delta$, and so $\delta = 1$, $T_{\delta,S} =
T_S$, and $E_{\delta} = E$.  Certainly $N_\L(T_S) \leq N_\L(E)$. For the
other inclusion, note $N_\L(E)$ acts on $C_H(E)$ by \eqref{E:oddnormal} so it
acts on $T_{S}$ since $T_{S}$ is the unique abelian $2$-subgroup of maximum
order in $C_H(E)$.  Thus, $N_\L(E) \leq N_\L(T_S)$, completing the proof of
\eqref{E:N=NL(E)}.

Using \eqref{E:CHEdelta} one observes easily that
$T_\delta=T_{\delta,S}O_{2'}(T_\delta)=T_{\delta,S}O_{2'}(C_H(E_\delta))$.
Hence, it follows from \eqref{E:oddnormal} and \eqref{E:N=NL(E)} that
\begin{eqnarray}\label{E:TdeltaNormN}
 T_\delta\unlhd N.
\end{eqnarray}
Notice that $C_N(E_\delta)/C_H(E_\delta)\leq C_G(E_\delta)/C_H(E_\delta)\cong
C_G(E_\delta)H/H\leq G/H$ and recall that $G/H$ is cyclic of odd order. Hence,
by \eqref{E:CHEdelta}, \eqref{E:oddnormal}, \eqref{E:N=NL(E)} and
\eqref{E:TdeltaNormN}, we are given a normal series
\[
T_\delta\leq C_H(E_\delta)\leq C_N(E_\delta)=C_\L(E_\delta)\leq N_\L(E_\delta)=N
\]
where $C_H(E_\delta)/T_\delta\cong C_2$, $C_N(E_\delta)/C_H(E_\delta)$ is
cyclic of odd order, and (by Lemma~\ref{L:EE'F})
\[
N/C_N(E_\delta)=N_\L(E_\delta)/C_\L(E_\delta)\cong \Aut_\F(E_\delta)=\Aut(E_\delta)\cong GL_3(2).
\]
Set $\widetilde{N}:=N/T_\delta$. By Lemma~\ref{L:EE'}, $C_2\times S_4\cong
\widetilde{N_H(T_\delta)}\leq \widetilde{N}$. In particular, $\widetilde{N}$
does not have quaternion Sylow $2$-subgroups. Applying Lemma~\ref{L:GSplit}
with $(\widetilde{N},\widetilde{C_N(E_\delta)},\widetilde{C_H(E_\delta)})$ in
place of $(G,M,W)$, we see now that 
\[
\widetilde{N}'\cong GL_3(2).
\] 
Let $r$ be a prime divisor of $q-\delta\epsilon$ and note that $\widetilde{N}$
acts on $T_{\delta,r}$. By Lemma~\ref{L:EE'}, $\widetilde{N_H(T_\delta)}\cong
C_2\times S_4$ acts faithfully on $T_{\delta,r}$ and thus $A_4\cong
\widetilde{N_H(T_\delta)}'  \leq \widetilde{N}'$ acts non-trivially on
$T_{\delta,r}$. As $C_{\widetilde{N}'}(T_{\delta,r})\unlhd \widetilde{N}'$ and
$\widetilde{N}'\cong GL_3(2)$ is simple, it follows that
$C_{\widetilde{N}'}(T_{\delta,r})=1$ and $\widetilde{N}'\cong GL_3(2)$ acts
faithfully on $T_{\delta,r}\cong C_r^3$. Since this holds for each $\delta =
\pm 1$ and prime $r$, Lemma~\ref{L:num} implies that $q=3^{1+6a}$ for some
$a\geq 0$ and $q \equiv 3 \pmod{8}$. As $q'\equiv\pm q\pmod{8}$, this shows
(2).
\end{proof}

Note that the conclusion of the following lemma does not hold if
$H=\Spin_7(q)$ for some $q\neq 3$.

\begin{lem}
\label{L:spin-link}
Let $H = \Spin_7(3)$ and $Z = Z(H)$.  If $P \geq Z$ is a $2$-subgroup of $H$ of
$2$-rank at least $2$, then $N_H(P)$ and $C_H(P)$ are of characteristic
$2$.
\end{lem}
\begin{proof}
Let $P \leq S$ with $Z \leq V \leq P$ and $V$ a four group.  By
Lemma~\ref{L:omnibus23}(2), we may conjugate in $H$ and take $V = U$, and
$C_H(U) = L_1L_2L_3\gen{c}$, where $c$ induces a diagonal automorphism on each
$L_i \cong SL_2(3)$. Thus, $O_2(C_H(U))$ is a commuting product of three
quaternion subgroups of order $8$ which contains its centralizer in $C_H(U)$,
and hence $C_H(U)$ is of characteristic $2$. 

Recall that $N_H(P)$ is of characteristic $2$ if and only if $C_H(P)$ is of
characteristic $2$ and that the normalizer of any $2$-subgroup in a group of
characteristic $2$ is of characteristic $2$ (see, e.g.
\cite[Lemma~2.2]{Henke2019}).  It follows that $N_{C_H(U)}(P)$ is of
characteristic $2$, so $C_H(P) = C_{C_H(U)}(P)$ is of characteristic $2$, so
$N_H(P)$ is of characteristic $2$. 
\end{proof}

\begin{lem}\label{L:2rank2Subcentric}
Let $\F=\F_{\Sol}(3)$ be a fusion system over $S$. Then every subgroup of $S$
of $2$-rank at least $2$ is $\F$-subcentric.
\end{lem}

\begin{proof}
Set $Z=Z(S)$. By Lemma~\ref{L:solomnibus}(1), we have $\H:=N_\F(Z)=\F_S(H)$
where $H=\Spin_7(q)$ and $S$ can be identified with the Sylow $2$-subgroup of
$H$ defined in Section~\ref{SSS:SylowSpin7}. Define $U$ as before so that
$\EE_2(S)=U^\F$ by Lemma~\ref{L:solomnibus}(2).  As $\F^s$ is by
\cite[Proposition~3.3]{Henke2019} closed under passing to $\F$-conjugates and
overgroups, it is  enough to prove that $U$ is $\F$-subcentric. Indeed, as
$Z\leq U$, we have $C_\F(U)=C_\H(U)=\F_{C_S(U)}(C_H(U))$. Hence, $\C_\F(U)$ is
constrained by Lemma~\ref{L:spin-link} and so $U\in\F^s$ by
\cite[Lemma~3.1]{Henke2019}.
\end{proof}

We may now prove the main theorem of this section.

\begin{proof}[Proof of Theorem~\ref{T:solq}]
$(\implies)\colon$ If $(\L,\Delta,S)$ is a punctured group over $\F :=
\F_{\Sol}(q)$  for some odd prime power $q$, then it follows from
Proposition~\ref{P:StructureNLZ}(2) (applied with $q$ in place of $q'$) that
$q\equiv\pm 3\pmod{8}$.

\smallskip

\noindent $(\impliedby):$ 
Now let $\F = \F_{\Sol}(3)$ and $\H = C_\F(Z) = \F_S(H)$ with $H = \Spin_7(3)$.
Set 
\[
\Delta = \{P \in \F^s \mid P \text{ is of $2$-rank at least $2$}\},
\]
and $\Delta_Z = \{P \in \Delta \mid P \geq Z\}$.  Then $\Delta$ is closed under
$\F$-conjugacy and passing to overgroups by \cite{Henke2019}. So it is also
closed under $\H$-conjugacy. We show now

\begin{equation}\label{E:FcrHcr}
\mbox{Every element of $\H^{cr}\cup\F^{cr}$ is of $2$-rank at least $2$.}
\end{equation}
Indeed, assume there exists $Q\in\H^{cr}\cup\F^{cr}$ of $2$-rank $1$. Then
$Q\neq S$ and so $C_S(Q)\leq Q$ implies $\Inn(Q)< \Aut_S(Q)\leq\Aut_\H(Q)\leq
\Aut_\F(Q)$. Suppose first that $Q$ is cyclic, or generalized quaternion of
order at least $16$. Then $\Aut(Q)$ is a $2$-group and so $\Out_\H(Q)$ and
$\Out_\F(Q)$ are non-trivial $2$-groups, which contradicts the assumption that
$Q$ is radical in $\H$ or $\F$.  Assume now that $Q$ is quaternion of order
$8$. As $U$ is a normal subgroup of $S$, we have $[Q,U] \leq Z = Z(S) \leq Q$,
so $U \leq N_S(Q)$. But $N_S(Q)$ is a $2$-group containing $Q$
self-centralizing with index at most $2$, and so $N_S(Q)$ is quaternion
or semidihedral of order $16$. But neither of these groups has a normal four
subgroup, a contradiction. This shows \eqref{E:FcrHcr}.

\smallskip

Each element of $\F^{cr} \cup \H^{cr}$ contains $Z$. It follows moreover
from \eqref{E:FcrHcr} and Lemma~\ref{L:2rank2Subcentric} that
$\F^{cr}\cup\H^{cr}\subseteq \Delta$. Also $\F^s \subseteq \H^s$ by
\cite[Lemma~3.16]{Henke2019}. Thus, we have shown
\begin{eqnarray}
\label{E:containscr}
\F^{cr} \cup \H^{cr} \subseteq \Delta_Z \subseteq \Delta \subseteq \F^s \subseteq \H^s. 
\end{eqnarray}
The hypotheses of \cite[Theorem~A]{Henke2019} are thus satisfied, so we may fix
a linking locality $\L$ on $\F$ with object set $\Delta$, and this $\L$ is
unique up to rigid isomorphism. 

We shall verify the conditions (1)-(5) of \cite[Hypothesis~5.3]{Chermak2013}
with $Z$ in the role of ``$T$'' and $H$ in the role of ``$M$''.  Conditions
(1), (2) hold by construction. Condition (4) holds since $Z$ is normal in $H$
and $\F_S(N_\L(Z)) \cong \H$ by \cite{LeviOliver2002}. To see condition (3),
first note that $Z$ is fully normalized in $\F$ because it is central in $S$.
Let $Z', Z''$ be distinct $\F$-conjugates of $Z$. Then $\gen{Z',Z''}$ contains
a four group $V$.  By Lemma~\ref{L:solomnibus}(2), $V$ is $\F$-conjugate
to $U$, and $O_{2}(N_\F(U)) \in \F^c$ is a commuting product of three
quaternion groups of order $8$. Thus, $V \in \Delta$, and hence $\gen{Z',Z''}
\in \Delta$.  So Condition (3) holds.  It remains to verify Condition (5),
namely that $N_\L(Z)$ and $\L_{\Delta_Z}(H)$ are rigidly isomorphic.  By
\eqref{E:containscr} and Lemma~\ref{L:spin-link}, $\L_{\Delta_Z}(H)$ is a
linking locality over $\H$ with $\Delta_Z$ as its set of objects.  

On the other hand, by \cite[Lemma~2.19]{Chermak2013}, $N_\L(Z)$ is a locality
on $\H$ with object set $\Delta_Z$, in which $N_{N_\L(Z)}(P) = C_{N_\L(P)}(Z)$
for each $P \in \Delta_Z$. As $\L$ a linking locality, $N_\L(P)$ is of
characteristic $2$, and hence the $2$-local subgroup $N_{N_\L(Z)}(P)$ of
$N_\L(P)$ is also of characteristic $2$. So again this together with
\eqref{E:containscr} gives that $N_\L(Z)$ is a linking locality over $\H$ with
object set $\Delta_Z$.  Thus, $\L_{\Delta_Z}(H)$ and $N_\L(Z)$ are
linking localities over the same fusion system and with the same object set,
thus rigidly isomorphic by \cite[Theorem~A]{Henke2019}. This completes the
proof of (5).

So by \cite[Theorem~5.14]{Chermak2013}, there is a locality $\L^+$ over $\F$
with object set 
\[
\Delta^+ := \{P \leq S \mid Z^\phi \leq P \text{ for some } \phi \in \Hom_\F(Z,S)\},
\]
such that $\L^+|_\Delta = \L$ and $N_{\L^+}(Z)=H$, and $\L^+$ is unique
up to rigid isomorphism with this property. Since each non-trivial
subgroup of $S$ contains an involution, and all involutions in $S$ are
$\F$-conjugate (by Lemma~\ref{L:solomnibus}(1)), $\Delta^+$ is the collection
of all nontrivial subgroups of $S$. Thus, $\L^+$ is a punctured group for $\F$. 
\end{proof}

\begin{remark}
Theorem~\ref{T:solq} leaves open the question whether there is a punctured
group $(\L,\Delta,S)$ over $\F_{\Sol}(3)$ such that, setting $Z:=Z(S)$, the
centralizer $C_\L(Z)$ is not isomorphic to $\Spin_7(3)$.  Indeed, we show in
Proposition~\ref{P:StructureNLZ} that always $O^{2'}(C_\L(Z)/O_{2'}(C_\L(Z)))
\cong \Spin_7(q)$, where $q = 3^{1+6a}$ for some $a \geq 0$ with the property
that $q^2-1$ is divisible only by primes which are squares modulo $7$. Although
there are at least several such nonnegative integers $a$ with this property
(the first few are $0, 1, 2, 3, 5, 7, 8, 13, 15, \dots$), we are unable to
determine whether a punctured group for $\F_{\Sol}(q)$ exists when $a > 0$.
\end{remark}

\section{Punctured groups for exotic fusion systems at odd primes}\label{S:pgexotic}

In this section, we survey some of the known examples of exotic fusion systems
at odd primes in the literature, and determine which ones have associated
punctured groups.

Let $\F$ be a saturated fusion system over the $p$-group $S$. A subgroup $Q$ of
$S$ is said to be $\F$-\emph{subcentric} if $Q$ is $\F$-conjugate to a subgroup
$P$ for which $O_p(N_\F(P))$ is $\F$-centric. Equivalently, by
\cite[Lemma~3.1]{Henke2019}, $Q$ is $\F$-subcentric if, for any fully
$\F$-normalized $\F$-conjugate $P$ of $Q$, the normalizer $N_\F(P)$ is
constrained.  Write $\F^s$ for the set of subcentric subgroups of $\F$. Thus,
$\F^s$ contains the set of nonidentity subgroups of $S$ if and only if $\F$ is
of characteristic $p$-type (and $\F^s$ is the set of \emph{all} subgroups of
$S$ if and only if $\F$ is constrained).

A finite group $G$ is said to be of \emph{characteristic $p$} if
$C_G(O_p(G))\leq O_p(G)$. A \emph{subcentric linking system} is a transporter
system $\L^s$ associated to $\F$ such that $\Obj(\L^s) = \F^s$ and
$\Aut_{\L^s}(P)$ is of characteristic $p$ for every $P\in\Obj(\L^s)$.  By a
theorem of Broto, Castellana, Grodal, Levi and Oliver \cite{BCGLO2005}, the
constrained fusion systems are precisely the fusion systems of finite
groups of characteristic $p$. The finite groups of characteristic $p$, which
realize the normalizers of fully normalized subcentric subgroups, can be
``glued together'' to build a subcentric linking systems associated with
$\F$.  More precisely, building on the unique existence of centric linking
systems, the first author \cite[Theorem~A]{Henke2019} has used Chermak
descent to show that each saturated fusion system has a unique associated
subcentric linking system.

For each of the exotic systems $\F$ considered in this section, it will
turn out that either $\F$ is of characteristic $p$-type, or $S$ has a fully
$\F$-normalized subgroup $X$ of order $p$ such that $N_\F(X)$ is exotic. In the
latter case, there is the following elementary observation.

\begin{lem}
\label{L:normexotic}
Let $\F$ be a saturated fusion system over $S$. Assume there is some nontrivial
fully $\F$-normalized subgroup $X$ such that $N_\F(X)$ is exotic. Then a
punctured group for $\F$ does not exist.
\end{lem}
\begin{proof}
If there were a transporter system $\L$ associated with $\F$ having object set
containing $X$, then $\Aut_\L(X)$ would be a finite group whose fusion system
is $N_\F(X)$. 
\end{proof}

We restrict attention here to the following families of exotic systems
at odd primes, considered in order: the Ruiz-Viruel systems
\cite{RuizViruel2004}, the Oliver systems \cite{Oliver2014}, 
the Clelland-Parker systems \cite{ClellandParker2010}, and the Parker-Stroth
systems \cite{ParkerStroth2015}. The results are summarized in the following
theorem. 

\begin{thm}
\label{T:pgexotic}
Let $\F$ be a saturated fusion system over a finite $p$-group $S$.
\begin{itemize}
\item[(a)] If $\F$ is a Ruiz-Viruel system at the prime $7$, then $\F$ is
of characteristic $7$-type, so has a punctured group.
\item[(b)] If $\F$ is an exotic Oliver system, then $\F$ has a punctured
group if and only if $\F$ occurs in cases (a)(i), (a)(iv), or (b) of
\cite[Theorem~2.8]{Oliver2014}.
\item[(c)] If $\F$ is an exotic Clelland-Parker system, then $\F$ has a
punctured group if and only if each essential subgroup is abelian. Moreover, if
so then $\F$ is of characteristic $p$-type.
\item[(d)] If $\F$ is a Parker-Stroth system, then $\F$ is of characteristic
$p$-type, so has a punctured group.
\end{itemize}
\end{thm}
\begin{proof}
This follows upon combining Theorem~\ref{T:subcentric} or
Lemma~\ref{L:normexotic} with Lemma~\ref{L:rvext}, Proposition~\ref{P:Oliver},
Propositions~\ref{P:CPR} and \ref{P:CPQ1}, and Proposition~\ref{P:PS},
respectively.
\end{proof}

When showing that a fusion system is of characteristic $p$-type, we will often
use the following elementary lemma.

\begin{lem}\label{ShowNormalizerConstrained}
Let $X$ be a fully $\F$-normalized subgroup of $S$ such that $C_S(X)$ is
abelian. Then $N_\F(X)$ is constrained.
\end{lem}
\begin{proof}
Using Alperin's Fusion Theorem
\cite[Theorem~I.3.6]{AschbacherKessarOliver2011}, one sees that $C_S(X)$ is
normal in $C_\F(X)$. In particular, $C_\F(X)$ is constrained. Therefore, by
\cite[Lemma~2.13]{Henke2019}, $N_\F(X)$ is constrained.
\end{proof}

\subsection{The Ruiz-Viruel systems}

Three exotic fusion systems at the prime $7$ were discovered by Ruiz and
Viruel, two of which are simple. The other contains one of the simple ones with
index $2$.

\begin{lem}\label{L:rvext}
Let $\F$ be a saturated fusion system over an extraspecial $p$-group $S$ of
order $p^3$ and exponent $p$. Then $N_\F(Z(S)) = N_\F(S)$.  In particular, $\F$
is of characteristic $p$-type. 
\end{lem}

\begin{proof}
Clearly $N_\F(S)\subseteq N_\F(Z(S))$. Note that $N_\F(Z(S))$ is a saturated
fusion system over $S$ as well. So by \cite[Lemma~3.2]{RuizViruel2004}, if a
subgroup of $S$ is centric and radical in $N_\F(Z(S))$, then it is either
elementary abelian of order $p^2$ or equal to $S$. Moreover, by
\cite[Lemma~4.1]{RuizViruel2004}, an elementary abelian subgroup $V$ of order
$p^2$ is radical in $N_\F(Z(S))$ if and only if $\Aut_\F(V)$ contains
$\SL_2(p)$. However, if $\Aut_\F(V)$ contains $\SL_2(p)$, then it does not
normalize $Z(S)$. This implies that $S$ is the only subgroup of $S$ which is
centric and radical in $N_\F(Z(S))$. Hence, by Alperin's Fusion Theorem
\cite[Theorem~I.3.6]{AschbacherKessarOliver2011}, we have $N_\F(Z(S))\subseteq
N_\F(S)$ and thus $N_\F(Z(S))=N_\F(S)$. In particular, $N_\F(Z(S))$ is
constrained. If $X$ is a non-trivial subgroup of $\F$ with $X\neq Z(S)$, then
$C_S(X)$ is abelian. So it follows from Lemma~\ref{ShowNormalizerConstrained}
that $\F$ is of characteristic $p$-type.
\end{proof}

In Section~\ref{S:RV}, it is shown that for the three exotic Ruiz-Viruel
systems, the subcentric linking system is the unique associated punctured group
whose full subcategory on the centric subgroups is the centric linking system.

\subsection{Oliver's systems}\label{SS:oliver}
A classification of the simple fusion systems $\F$ on $p$-groups with a unique
abelian subgroup $A$ of index $p$ is given in \cite{Oliver2014, COS2017,
OliverRuiz2020}. Here we consider only those exotic fusion systems in which $A$
is not essential in $\F$, namely those fusion systems appearing in the
statement of \cite[Theorem~2.8]{Oliver2014}.  

Whenever $\F$ is a saturated fusion system on a $p$-group $S$ with a
unique abelian subgroup $A$ of index $p$, we adopt Notation~2.2 of
\cite{Oliver2014}. For example,
\[
Z = Z(S),\quad Z_2 = Z_2(S),\quad  S' = [S,S],\quad Z_0 = Z \cap S',\quad
\text{ and } \quad A_0 = ZS',
\]
and also
\[
\H = \{Z\gen{x} \mid x \in S-A\}\quad  \text{ and } \quad \B = \{Z_2\gen{x} \mid x \in S-A\}. 
\]

\begin{lem}
\label{L:indp-basic}
Let $\F$ be a saturated fusion system on a finite $p$-group $S$ having a unique
abelian subgroup $A$ of index $p$. 
\begin{itemize}
\item[(a)] If $P \leq S$ is $\F$-essential, then $P \in \{A\} \cup \H \cup \B$,
$|N_S(P)/P| = p$, and each $\alpha \in N_{\Aut_\F(P)}(\Aut_S(P))$ extends to an
automorphism of $S$. 
\end{itemize}
Assume now in addition that $A$ is not essential in $\F$.
\begin{itemize}
\item[(b)] If $O_p(\F) = 1$, then $\F^e \cap \H \neq \varnothing$, $Z_0 = Z$ is
of order $p$, $S' = A_0$ is of index $p^2$ in $S$, and $S$ has maximal class.
\item[(c)] If $P \in \H \cup \B$ is $\F$-essential, then $P \cong C_p^2$ or
$p^{1+2}_+$ according to whether $P \in \H$ or $P \in \B$, and
$O^{p'}(\Out_\F(P)) \cong SL_2(p)$ acts naturally on $P/[P,P]$. 
\item[(d)] If $P \in \F^e \cap \H$, then each $\alpha \in N_{\Aut_\F(P)}(Z)$
extends to an automorphism of $S$. 
\item[(e)] A subgroup $P \leq S$ is essential in $N_\F(Z)$ if and only if $P
\in \F^e \cap \B$. 
\item[(f)] There is $x \in S-A$ such that $A_0\gen{x}$ is $\Aut_\F(S)$-invariant. 
\end{itemize}
\end{lem}
\begin{proof}
Parts (a), (b), and (f) are shown in \cite[Lemma~2.3,2.4]{Oliver2014}, and (c)
follows from \cite[Lemma~2.7]{Oliver2014}.  Suppose as in (d) that $P \in \F^e
\cap \H$. By (c), $\Aut_\F(P)$ is a subgroup of $GL_2(p)$ containing $SL_2(p)$,
and the stabilizer of $Z$ in this action normalizes $O^{p'}(C_{\Aut_\F(P)}(Z))
= \Aut_S(P)$. So (d) follows from (a).

It remains to prove (e).  If $P \in \F^e \cap \B$, then as $Z = [P,P]$ is
$\Aut_\F(P)$-invariant in this case, $\Out_{N_\F(Z)}(P) = \Out_\F(P)$ has a
strongly $p$-embedded subgroup, and so $P$ is essential in $N_\F(Z)$.
Conversely, suppose $P$ is $N_\F(Z)$-essential. By (a) applied to $N_\F(Z)$, $P
\in \{A\} \cup \H \cup \B$ and $\Out_S(P)$ is of order $p$, so by assumption
$N_{\Out_{N_\F(Z)}(P)}(\Out_S(P))$ is strongly $p$-embedded in
$\Out_{N_\F(Z)}(P)$ by \cite[Proposition~A.7]{AschbacherKessarOliver2011}.  Now
each member of $N_{\Aut_{\F}(P)}(\Aut_S(P))$ extends to $S$ by (a), so $Z$ is
$N_{\Aut_{\F}(P)}(\Aut_S(P))$-invariant. Thus, $N_{\Out_\F(P)}(\Out_S(P)) =
N_{\Out_{N_\F(Z)}(P)}(\Out_S(P))$ is a proper subgroup of $\Out_\F(P)$, and
hence strongly $p$-embedded by
\cite[Proposition~A.7]{AschbacherKessarOliver2011} again. So $P$ is essential
in $\F$. By assumption $P \neq A$, and $P \notin \H$ by (d). So $P \in \B$. 
\end{proof}

\textbf{For the remainder of this subsection, we let $\F$ be a saturated fusion
system on a $p$-group $S$ with a unique abelian subgroup $A$ of index $p$.
Further, we assume that $O_p(\F) = 1$ and $A$ is not essential in $\F$.}

\medskip
We next set up some additional notation. Fix an element $x \in S-A$ such that
$A_0\gen{x}$ is $\Aut_\F(S)$-invariant, as in Lemma~\ref{L:indp-basic}(f).
Since $O_p(\F) = 1$, $S$ is of maximal class by Lemma~\ref{L:indp-basic}(b). In
particular $Z = Z_0$ is of order $p$, $A/A_0$ is of order $p$, and $S' = A_0$,
so we can adopt \cite[Notation~2.5]{Oliver2014}.  As in
\cite[Notation~2.5]{Oliver2014}, let $a\in A\backslash A_0$, and define $\H_i$ and $\B_i$ to be the
$S$-conjugacy classes of the subgroups $Z\gen{xa^i}$ and $Z_2\gen{xa^i}$ for $i
= 0,1,\dots,p-1$, and set 
\[
\H_* = \H_1 \cup \cdots \cup \H_{p-1} \quad \text{ and } \quad \B_* = \B_1 \cup \cdots \cup \B_{p-1},
\]
so that $\H = \H_0 \cup \H_*$ and $\B = \B_0 \cup \B_*$. 

Set
\[
\Delta = (\ZZ/p\ZZ)^\times \times (\ZZ/p\ZZ)^\times \quad \text{ and } \quad
\Delta_i = \{(r,r^i) \mid r \in (\ZZ/p\ZZ)^\times\}.
\]
Define $\mu\colon \Aut_\F(S) \to \Delta$ and $\widehat{\mu}\colon \Out_\F(S)
\to \Delta$ by $\widehat{\mu}([\alpha]) = \mu(\alpha) = (r,s)$, where
\[
(xA_0)^\alpha = x^rA_0 \quad  \text{ and } \quad z^\alpha = z^s.
\]
The following lemma looks at the image of homomorphisms analogous to $\mu$ and
$\widehat{\mu}$ which are defined instead with respect to $N_\F(Z)/Z$ and
$C_\F(Z)/Z$. 

\begin{lem}
\label{L:indp-aut}
Assume $|S/Z| = p^m$ with $m \geq 4$. Let $\E \in \{N_\F(Z), C_\F(Z)\}$, and
let $\mu_\E$ be the restriction of $\mu$ to $\Aut_\E(S)$. 
Let $\mu_{\E/Z}\colon \Aut_{\E/Z}(S/Z) \to \Delta$ be the map analogous to
$\mu$ but defined instead with respect to $S/Z$. Then
\[
\im(\mu_{\E/Z}) = \{(r,sr^{-1}) \mid (r,s) \in \im(\mu_\E)\}
\]
In particular, if $\im(\mu_\E) = \Delta$, then $\im(\mu_{\E/Z}) = \Delta$. And
if $\im(\mu_\E) = \Delta_i$ for some $i$, then $\im(\mu_{\E/Z}) =
\Delta_{i-1}$, where the indices are taken modulo $p-1$. 
\end{lem}
\begin{proof}
This essentially follows from \cite[Lemma~1.11(b)]{COS2017}. By assumption,
$\E/Z$ is a fusion system over a $p$-group $S/Z$ of order at least $p^4$.  So
$A/Z$ is the unique abelian subgroup of $S/Z$ of index $p$ by
\cite[Lemma~1.9]{Oliver2014}. Since $S$ is of maximal class, so is the quotient
$S/Z$.  In particular, $Z(S/Z)$ is of order $p$, so we can define $\mu_{\E/Z}$
as suggested with $xZ$ in the role of $x$ and $gZ$ in the role of $z$, where $g
\in Z_2-Z$ is a fixed element. 

Let $\alpha \in \Aut_\E(S)$ with $\mu(\alpha) = (r,s)$, let $\bar{\alpha}$ be
the induced automorphism of $S/Z$, and let $t \in (\ZZ/p\ZZ)^\times$ be such
that $a^\alpha A_0 = a^t A_0$ (which exists since $A$ and $A_0$ are
$\Aut_\F(S)$-invariant and $|A/A_0| = |Z_0| = p$).  By
\cite[Lemma~1.11(b)]{COS2017}, $\alpha$ acts on the $i$-th upper central
quotient $Z_{i}(S)/Z_{i-1}(S)$ by raising a generator to the power $tr^{m-i}$
for $i = 1,\dots,m-1$.  Thus, $s = tr^{m-1}$ and $(gZ)^{\bar{\alpha}} =
g^\alpha Z = g^{tr^{m-2}}Z$.  Hence, $\mu_{\E/Z}(\bar{\alpha}) = (r, sr^{-1})$.
Conversely if $\mu_{\E/Z}(\bar{\alpha}) = (r, \bar{s})$, then $\mu_{\E}(\alpha) =
(r,\bar{s}r)$.  
\end{proof}

In the following proposition, we refer to Oliver's systems according to the
itemized list (a)(i-iv), (b) given in \cite[Theorem~2.8]{Oliver2014}.

\begin{prop}\label{P:Oliver}
Assume $\F$ is one of the exotic systems appearing in Theorem~2.8 of
\cite{Oliver2014}. Write $|S/Z| = p^m$ with $m \geq 3$. 
\begin{itemize}
\item[(a)] $\F$ is of characteristic $p$-type whenever $\F^e \subseteq \H$. In
particular, this holds if $\F$ occurs in case (a)(i), (a)(iv), or (b).
\item[(b)] If $\F$ is in case (a)(ii) and $m \geq 4$, then $N_\F(Z)$ is exotic.
Moreover, $\F$ is of component type, and $C_{\F}(Z)/Z$ is simple, exotic, and
occurs in (a)(iv) in this case. If $\F$ is in case (a)(ii) with $m = 3$ (and
hence $p = 5$), then $N_\F(Z)/Z$ is the fusion system of $5^2GL_2(5)$, and $\F$
is of characteristic $5$-type.  
\item[(c)] If $\F$ is in case (a)(iii), then $N_\F(Z)$ is exotic.  Moreover,
$\F$ is of component type where $C_\F(Z)/Z$ is simple, exotic, and of type
(a)(i). 
\end{itemize}
\end{prop}

\begin{proof}
Each of Oliver's systems is simple on $S$ with a unique abelian subgroup $A$ of
index $p$ which is not essential, so it satisfies our standing assumptions and
the hypotheses of Lemmas~\ref{L:indp-basic} and \ref{L:indp-aut}, and we can
continue the notation from above. In particular, $Z_0 = Z$ is of order $p$, $S'
= A_0$, and $S$ is of maximal class. Set \[\E:=C_\F(Z),\;\ov{S}=S/Z\mbox{ and
}\ov{\E}=\E/Z.\] If $G$ is a group realizing $N_\F(Z)$, then $C_G(Z)$ realizes
$\E=C_{N_\F(Z)}(Z)$ and so $C_G(Z)/Z$ realizes $\ov{\E}$. Hence
\begin{equation}\label{E:ovEexoticNFZexotic}
\mbox{if $\ov{\E}$ is exotic, then $N_\F(Z)$ is exotic.}
\end{equation}
The $\E$-automorphism group of $S$ is the centralizer of $Z$ in $\Aut_\F(S)$.
Thus, if $\im(\mu) = \Delta$, then by definition of the map $\mu$, we have
\[
\im(\mu_\E) = \{(r,1) \mid r \in (\ZZ/p\ZZ)^\times\} = \Delta_0,
\]
which implies $\im(\mu_{\bar{\E}}) = \Delta_{-1}$ by Lemma~\ref{L:indp-aut}. So
\begin{equation}\label{E:ImMuDelta}
 \mbox{if $\im(\mu) = \Delta$, then $\im(\mu_\E)=\Delta_0$ and $\im(\mu_{\bar{\E}}) = \Delta_{-1}$.}
\end{equation}
For each fully $\F$-normalized subgroup $X \leq S$ of order $p$ and not equal
to $Z$, $C_S(X)$ is abelian: if $X \leq A$ this follows since $C_S(X) = A$ ($X$
is not central), while if $X \nleq A$, this follows since $C_A(X) = Z$ by
Lemma~\ref{L:indp-basic}(b).  Thus $N_\F(X)$ is constrained in this case by
Lemma~\ref{ShowNormalizerConstrained}.  Hence
\begin{equation}\label{E:ShowCharpType}
\mbox{$\F$ is of characteristic
$p$-type if and only if $N_\F(Z)$ is constrained.}
\end{equation}
By Lemma~\ref{L:indp-basic}(e), if $\F^e \subseteq \H$, then $N_\F(Z)$ has no
essential subgroups. By the Alperin-Goldschmidt fusion theorem
\cite[I.3.5]{AschbacherKessarOliver2011}, each morphism in $N_\F(Z)$ extends in
this case to $S$, and hence $S$ is normal in $N_\F(Z)$. So
\begin{equation}\label{E:FeinH}
\mbox{if $\F^e \subseteq \H$, then
$N_\F(Z)$ is constrained.}
\end{equation}
In particular, the first part of (a) holds.

\smallskip
\noindent
\textbf{Case: $\F$ occurs in (a)(i), (a)(iv), or (b) of
\cite[Theorem~2.8]{Oliver2014}.} 
We have $\F^e \subseteq \H$ precisely in these cases. So $\F$ is of
characteristic $p$-type by \eqref{E:FeinH}. This completes the proof of (a).

\smallskip
\noindent
\textbf{Case: $\F$ occurs in (a)(ii).}
Here, $m \equiv -1 \pmod{p-1}$, $\im(\mu) = \Delta$, and $\F^e
= \B_0 \cup \H_*$.  By assumption $\F$ is exotic, so as $\F$ is the fusion
system of ${ }^3D_4(q)$ when $p = 3$, we have $p \geq 5$.

By Lemma~\ref{L:indp-basic}(e), the set of $N_\F(Z)$-essential subgroups is
$\F^e \cap \B=\B_0$. A straightforward argument shows now that the elements of
$\B_0$ are also essential in $\E=C_\F(Z)$, and their images in $\ov{S}$ are
essential in $\ov{\E}$. Thus, $\bar{\B}_0\subseteq \bar{\E}^e$, where
$\bar{\B}_0 = \{\bar{P} \mid P \in \B_0\}$. 

\smallskip
\noindent
\textbf{Subcase:} $m \geq 4$. By \eqref{E:ovEexoticNFZexotic}, it is sufficient
to show that $\ov{\E}$ is simple, exotic and occurs in case (a)(iv) of Oliver's
classification. Since $\bar{S}$ has order $p^m$, we know that $\bar{A}$ is the
unique abelian subgroup of $\bar{S}$ of index $p$ by
\cite[Lemma~1.9]{Oliver2014}. As $A$ is not $\F$-essential, it follows from the
Alperin--Goldschmidt fusion theorem that every element of $\Aut_\F(A)$ extends
to an $\F$-automorphism of $S$. From this one sees that $\ov{A}$ is not radical
and thus not essential in $\ov{\E}$.

We will prove first that $\bar{\E}$ is reduced.  Since $O_{p}(\bar{\E})$ is
contained in every $\bar{\E}$-essential subgroup, we have $O_{p}(\bar{\E}) \leq
\bigcap \bar{\B}_0 = Z(\bar{S})$. By Lemma~\ref{L:indp-basic}(c), $Z_2$ is not
$\Aut_{\E}(P)$-invariant for any $P \in \B_0$, and hence $\bar{Z}_2 =
Z(\bar{S})$ is not $\Aut_{\bar{\E}}(\bar{P})$-invariant for any $\bar{P} \in
\bar{\B}_0$. So $O_{p}(\bar{\E}) = 1$.

We next show that $O^{p}(\bar{\E}) = \bar{\E}$.  By
\cite[Proposition~1.3(c,d)]{Oliver2014}, the focal subgroup of $\bar{\E}$ is
generated by $[\bar{P}, \Aut_{\bar{\E}}(\bar{P})]$ for $\bar{P} \in \bar{\B}_0
\cup \{\bar{S}\}$, and $O^p(\bar{\E}) = \bar{\E}$ if and only if
$\foc(\bar{\E}) = \bar{S}$.  Since $\bar{P}$ is a natural module for
$O^{p'}(\Aut_{\bar{\E}}(\bar{P})) \cong SL_2(p)$ for each $\bar{P} \in
\bar{\B}_0$ (Lemma~\ref{L:indp-basic}(c)), the focal subgroup of $\bar{\E}$
contains $\gen{\bar{\B}_0} = \bar{A}_0\gen{\bar{x}}$.  Thus, $\foc(\bar{\E}) =
\bar{S}$ if $\bar{a} \in [\bar{S},\Aut_{\bar{\E}}(\bar{S})]$. By
\eqref{E:ImMuDelta}, $\im(\mu_{\bar{\E}})= \Delta_{-1}$.  Further, if
$\bar{\alpha}$ is an $\bar{\E}$-automorphism of $\bar{S}$ with
$\mu_{\bar{\E}}(\bar{\alpha}) = (r,r^{-1})$, then for the class $t \in
(\ZZ/p\ZZ)^{\times}$ with $(\bar{a}\bar{A}_0)^{\bar{\alpha}} =
\bar{a}^t\bar{A}_{0}$, we have $r^{-1} = tr^{m-2}$ by
\cite[Lemma~2.6(a)]{Oliver2014}, and hence $t = r^{-(m-1)}$.  As $m+1\equiv 0
\pmod{p-1}$ and $p \geq 5$, we have $-(m-1) \not\equiv 0 \pmod{p-1}$.  So
$\Aut_{\bar{\E}}(\bar{S})$ acts nontrivially on $\bar{A}/\bar{A}_0$. Hence
$\foc(\bar{\E}) = \bar{S}$ and $\ov{\E}=O^p(\ov{\E})$.

We next show that $O^{p'}(\bar{\E}) = \bar{\E}$ using Lemma~1.4 of
\cite{Oliver2014}. Set $\bar{P} = Z(\bar{S})\gen{\bar{x}} =
\bar{Z}_2\gen{\bar{x}} \in \bar{\B}_0$, and let $\bar{\alpha}$ be an
$\bar{\E}$-automorphism of $\bar{S}$. Recall that $x$ was chosen such that
$A_0\gen{x}$ is $\Aut_\F(S)$-invariant. Moreover, $\B$ is
$\Aut_\F(S)$-invariant and $\B_0$ consists of the elements of $\B$ that lie in
$A_0\gen{x}$. Hence, $\B_0$ is $\Aut_\F(S)$-invariant and so $\bar{\alpha}$
preserves the $\bar{S}$-class $\bar{\B}_{0}$ under conjugation. Thus, upon
adjusting $\bar{\alpha}$ by an inner automorphism of $\bar{S}$ (which doesn't
change the image of $\bar{\alpha}$ under $\mu_{\bar{\E}}$), we can assume that
$\bar{\alpha}$ normalizes $\bar{P}$. The restriction of $\bar{\alpha}$ to
$\bar{P}$ acts via an element of $SL_2(p)$ on $\bar{P}$ since
$\im(\mu_{\bar{\E}}) = \Delta_{-1}$, and so this restriction is contained in
$O^{p'}(\Aut_{\bar{\E}}(\bar{P}))$. Thus, $O^{p'}(\bar{\E}) = \bar{\E}$ by
Lemma~1.4 of \cite{Oliver2014}. 

Thus, $\bar{\E}$ is reduced. Step 1 of the proof of
\cite[Theorem~2.8]{Oliver2014} then shows that $\bar{\E}$ is the unique reduced
fusion system with the given data, and then Step 2 shows that $\bar{\E}$ is
simple. So $\bar{\E}$ is exotic and occurs in case (a)(iv) of Oliver's
classification, since $m-1 \equiv -2 \not\equiv 0,-1 \pmod{p-1}$.

\smallskip
\noindent
\textbf{Subcase: $m = 3$.}
Since $m \equiv -1 \pmod{p-1}$, we have $p = 5$.  So $\bar{S}$ is extraspecial
of order $5^3$ and exponent $5$. We saw above that $N_\F(Z)^e = \B_0$, which is
of size $1$ in this case. That is $Z_2\gen{x}$ is the unique essential subgroup
of $N_\F(Z)$, which is therefore invariant under
$\Aut_\F(S)=\Aut_{N_\F(Z)}(S)$. By the Alperin-Goldschmidt fusion theorem, this
subgroup is normal $N_\F(Z)$ and so $Z(\ov{S})\gen{\ov{x}}$ is normal in
$N_\F(Z)/Z$. This implies that $N_\F(Z)$ and $N_\F(Z)/Z$ are constrained. In
particular, $\F$ is of characteristic $5$-type by \eqref{E:ShowCharpType}. As
$\im(\mu_{N_\F(Z)})=\im(\mu)=\Delta$, it follows from Lemma~\ref{L:indp-aut}
that $\im(\mu_{N_\F(Z)/Z})=\Delta$. Using this and Lemma~\ref{L:indp-basic}(c),
one sees that $N_\F(Z)/Z$ is indeed isomorphic to the fusion system of
$5^2GL_2(5)$.  This completes the proof of (b).

\smallskip
\noindent
\textbf{Case: $\F$ occurs in (a)(iii).}
Then $m \equiv 0 \pmod{p-1}$, $\F^e = \H_0 \cup \B_*$, and $\im(\mu) = \Delta$.
Again, by \eqref{E:ovEexoticNFZexotic}, it is sufficient to show that $\ov{\E}$
is simple, exotic and occurs in case (a)(i) of Oliver's classification.
Similarly as in the previous case, by Lemma~\ref{L:indp-basic}(e), the set of
$N_\F(Z)$-essential subgroups is $\F^e \cap \B=\B_*$ and $\bar{\B}_*=\{\bar{P}
\mid P \in \B_*\}\subseteq \ov{\E}^e$.

Since $m \geq 3$, we have $p \geq 5$, and hence in fact $m \geq 4$. In
particular, $\bar{A}$ is the unique abelian subgroup of $\bar{S}$. Moreover,
$\bar{A}$ is not essential in $\bar{\E}$, and $O_p(\bar{\E}) = 1$ by a similar
argument as in the previous case. Also as the previous case, the focal subgroup
of $\bar{\E}$ contains $\gen{\bar{\B}_*}$, which this time is equal to
$\bar{S}$. So $O^p(\bar{\E}) = \bar{\E}$.

Notice that $m\equiv 0\pmod{p-1}$ implies $\Delta_0=\Delta_m$ and thus
$\im(\mu_\E)=\Delta_m$ by \eqref{E:ImMuDelta}. It follows therefore from
\cite[Lemma~2.6(b)]{Oliver2014} that $\B_i$ is $\Aut_\E(S)$-invariant for
$i=1,2,\dots,p-1$. Hence, arguing as in the previous step (but with some $\B_i$
instead of $\B_0$), one sees that $O^{p^\prime}(\bar{\E})=\bar{\E}$. Hence,
$\bar{\E}$ is reduced.

It follows now from Steps 1 and 2 of the proof of
\cite[Theorem~2.8]{Oliver2014} that $\ov{\E}$ is simple and uniquely
determined. As every essential subgroup of $\ov{\E}$ has order $p^2$ and
$|\bar{S}/Z(\bar{S})| = p^{m-1}$ where $m-1 \equiv -1 \pmod{p-1}$, it follows
moreover that $\bar{\E}$ occurs in case (a)(i) of Oliver's classification. In
particular, $\bar{\E}$ is exotic. This completes the proof of (c) and thus the
proof of the proposition.
\end{proof}

\subsection{The Clelland-Parker systems}\label{SS:CP}

We now describe the fusion systems constructed by Clelland and Parker in
\cite{ClellandParker2010}. Throughout we fix a power $q$ of the odd prime $p$,
a natural number $n\leq p-1$, and we set $k := \mathbb{F}_q$.  Let $A :=
A(n,k)$ be the $(n+1)$-dimensional space of homogeneous polynomials of degree
$n$ in two variables with coefficients in $k$. The group $D := k^\times \times
GL_2(k)$ acts on $A$ via $f(x,y)\cdot
(\lambda,[\begin{smallmatrix}a&b\\c&d\end{smallmatrix}]) = \lambda f(ax+b,
cy+d)$. The subgroup $SL_2(k)$ of $D$ acts irreducibly on $A$.  Write $G$ for
the semidirect product $DA$. Let $U$ be a Sylow $p$-subgroup of $D$ and let $S
:=S(n,k):=UA$ be the semidirect product of $A$ by $U$.

The center $Z := Z(S)$ is a one-dimensional $k$-subspace of $A$ and by
\cite[Lemma~4.2(iii)]{ClellandParker2010}, we have 
\begin{eqnarray}
\label{E:CSU}
\text{$C_A(X) = Z(S)$ for each subgroup $X$ not contained in $A$.}
\end{eqnarray} 
The second center $Z_2(S)$ is a $2$-dimension $k$-subspace of $A$.  Let $R =
ZU$ and $Q = Z_2(S)U$. Then $R \cong q^2$ and $Q$ is special of shape
$q^{1+2}$.  Let $H_R$ be the stabilizer in $GL_3(k)$ of a one dimensional
subspace, and identify its unipotent radical with $R$.  Let $H_Q$ be the
stabilizer in $GSp_4(k)$ of a one dimensional subspace and identify the
corresponding unipotent radical with $Q$. It is shown in
\cite{ClellandParker2010} that $N_G(R)$ is isomorphic to a Borel subgroup
$GL_3(k)$, and that $N_G(Q)$ is isomorphic to a Borel subgroup of $GSp_4(k)$.
This allows to form the free amalgamated products 
\[
F(1, n, k, R) := G *_{N_G(R)} H_R
\]
and 
\[
F(1,n,k,Q) := G *_{N_G(Q)} H_Q.
\]
Set 
\[
\F(1,n,k,R) := \F_S(F(1,n,k,R)) 
\]
and
\[
\F(1,n,k,Q) := \F_S(F(1,n,k,Q)).
\]
More generally, for each $X\in\{R,Q\}$ and each divisor $r$ of $q-1$, subgroup
$F(r,n,k,X)$ of $F(1,n,k,X)$ of index $r$, which contains $O^{p^\prime}(G)$ and
$O^{p^\prime}(H_X)$. They set then \[\F(r,n,k,X)=\F_S(F(r,n,k,X)).\] As they
show, distinct fusion systems are only obtained for distinct divisors $r$ of
$(n+2,q-1)$ when $X=R$, and for distinct divisors $r$ of $(n,q-1)$ when $X=Q$.
By \cite[Theorem~4.9]{ClellandParker2010}, for all $n\geq 1$ and each divisor
$r$ of $(n+2,q-1)$, $\F(r,n,k,R)$ is saturated. Similarly, $\F(r,n,k,Q)$ is
saturated for each $n\geq 2$ and each divisor $r$ of $(n,q-1)$. It is
determined in Theorem~5.1, Theorem~5.2 and Lemma~5.3 of
\cite{ClellandParker2010} which of these fusion systems are exotic. It turns
out that $\F(r,n,k,R)$ is exotic if and only if either $n>2$ or $n=2$  and
$q\not\in\{3,5\}$. Furthermore, $\F(r,n,k,Q)$ is exotic if and only if $n\geq
3$, in which case $p\neq 3$ as $n\leq p-1$. 

\textbf{For the remainder of this subsection, except in
Lemma~\ref{L:ClParkerExotic}, we use the notation introduced above.}

For the problems we will consider here, we will sometimes be able to reduce to
the case $r=1$ using the following lemma.

\begin{lem}\label{L:ReduceTo1}
For any divisor $r$ of $q-1$, the fusion system $\F(r,n,k,R)$ is a normal
subsystem of $\F(1,n,k,R)$ of index prime to $p$, and the fusion system
$\F(r,n,k,Q)$ is a normal subsystem of $\F(1,n,k,Q)$ of index prime to $p$.
\end{lem}
\begin{proof}
For $X\in\{R,Q\}$, the fusion systems $\F(r,n,k,X)$ and $\F(1,n,k,X)$ are both
saturated by the results cited above, As $F(r,n,k,X)$ is a normal subgroup of
$F(1,n,k,X)$, it is easy to check that $\F(r,n,k,X)$ is
$\F(1,n,k,X)$-invariant. As both $\F(1,n,k,X)$ and $\F(r,n,k,X)$ are fusion
systems over $S$, the claim follows. 
\end{proof}

\begin{prop}\label{P:CPR}
$\F(r,n,k,R)$ is of characteristic $p$-type for all $1 \leq n \leq p-1$ and for
all divisors $r$ of $(n+2,q-1)$. 
\end{prop}

\begin{proof}
Fix $1 \leq n \leq p-1$ and a divisor $r$ of $(n+2,q-1)$. Set $\F =
\F(1,n,k,R)$. By Lemma~\ref{L:ReduceTo1}, $\F(r,n,k,R)$ is a normal subsystem
of $\F$ of index prime to $p$. So by \cite[Proposition~2(c)]{Henke2019},
it suffices to show that $\F$ is of characteristic $p$-type. By
\cite[Lemma~5.3(i,ii)]{ClellandParker2010}, $\F$ is of realizable and of
characteristic $p$-type when $n = 1$, so we may and do assume $n \geq 2$.

Using the notation above, set $\F_1 = \F_S(G)$, $S_2 = N_S(R)$, and $\F_2 =
\F_{S_2}(H_R)$.  The fusion system $\F$ is generated by $\F_1$ and $\F_2$ by
\cite[Theorem~3.1]{ClellandParker2010}, and so as $\F_1$ and $\F_2$ are both
constrained with $O_p(\F_1) = A$ and $O_p(\F_2) = R$, it follows that $\F$ is
in turn generated by $\Aut_{\F_1}(A)$, $\Aut_{\F_1}(S)$, $\Aut_{\F_2}(R)$, and
$\Aut_{\F_2}(S_2)$. However, the last automorphism group is redundant, since
$N_{H_R}(S_2) = N_G(R)$ induces fusion in $\F_1$. Hence 

\begin{eqnarray}
\label{E:Fgen}
\F = \gen{\Aut_{\F_1}(S), \Aut_{\F_1}(A), \Aut_{\F_2}(R)}.
\end{eqnarray}
Observe also that the following property is a direct consequence of
\eqref{E:CSU}:

\begin{eqnarray}\label{E:Cases} 
&&\mbox{If $X\leq S$ with $X\not\leq Z$, then either $X\leq A$ and $C_S(X)=A$,
or $|C_S(X)|\leq q^2$.} 
\end{eqnarray}
We can now show that $\F$ is of characteristic $p$-type. Let first $X\in\F^f$
such that $X\not\leq Z$. We show that $N_\F(X)$ is constrained. If $X$ is not
$\F$-conjugate into $A$ or into $R$, then every morphism in $N_\F(X)$ extends
by \eqref{E:Fgen} to an automorphism of $S$. So $N_\F(X)=N_{N_\F(S)}(X)$. As
$N_\F(S)$ is constrained, it follows thus from
\cite[Lemma~2.11]{Henke2019} that $N_\F(X)$ is constrained. So we may
assume that there exists an $\F$-conjugate $Y$ of $X$ with $Y\leq A$ or $Y\leq
R$. We will show that $C_S(X)$ is abelian so that $N_\F(X)$ is constrained by
Lemma~\ref{ShowNormalizerConstrained}. Note that $|C_S(X)|\geq|C_S(Y)|$ as $X$
is fully normalized and thus fully centralized in $\F$. Since $X$ is not
contained in $Z=Z(S)$, we have in particular $Y\not\leq Z$. If $Y\leq A$, then
$A\leq C_S(Y)$ and, since $X$ is fully centralized and $n\geq 2$, $|C_S(X)|\geq
|C_S(Y)|\geq|A|>q^2$. So by \eqref{E:Cases}, $C_S(X)=A$ is abelian. Similarly,
by \eqref{E:Cases}, if $X\leq A$ then $C_S(X)=A$ is abelian. Thus we may assume
$Y\leq R$ and $X\not\leq A$. Then $R\leq C_S(Y)$ and \eqref{E:Cases} implies
$q^2\geq |C_S(X)|\geq |C_S(Y)|\geq |R|=q^2$. So the inequalities are
equalities, $C_S(Y)=R$ and $|C_S(X)|=q^2$. By the extension axiom, there exists
$\phi\in\Hom_\F(C_S(Y),C_S(X))$. So it follows that $C_S(X)\in R^\F$ is
abelian. This completes the proof that $N_\F(X)$ is constrained for every
$X\in\F^f$ with $X\not\leq Z$.

Let now $1\neq X\leq Z$. It remains to show that $N_\F(X)$ is constrained. If
$N_\F(X)\subseteq N_\F(S)$, then again by \cite[Lemma~2.11]{Henke2019},
$N_\F(X)=N_{N_\F(S)}(X)$ is constrained since $N_\F(S)$ is constrained. We will
finish the proof by showing that indeed $N_\F(X)\subseteq N_\F(S)$. Assume by
contradiction that $N_\F(X)\not\subseteq N_\F(S)$. Then there exists an
essential subgroup $E$ of $N_\F(X)$. Observe that $Z < E$, since $E$ is
$N_\F(X)$-centric. As $\Aut_S(E)$ is not normal in $\Aut_{N_\F(Z)}(E)$, there
exists an element of $\Aut_{N_\F(Z)}(E)$ which does not extend to an
$\F$-automorphism of $S$. So by \eqref{E:Fgen}, $E$ is $\F$-conjugate into $A$
or into $R$. Assume first that there exists an $\F$-conjugate $\widehat{E}$ of $E$
such that $\widehat{E}\leq A$. Property \eqref{E:CSU} yields that $R\cap A=Z$. So
$E$ is conjugate to $\widehat{E}\leq A$ via an element of $\Aut_{\F_1}(S)$ by
\eqref{E:Fgen}. Thus, as $C_S(E)\leq E$, we have $A\leq C_S(\widehat{E})\leq
\widehat{E}$. Hence $A=\widehat{E}$ by \eqref{E:Cases}. As $A$ is
$\Aut_{\F_1}(S)$-invariant, it follows $E=A$. Looking at the structure of $G$,
we observe now that $N_G(X)=N_G(S)$ and so $\Aut_S(A)$ is normal in
$\Aut_{N_\F(X)}(A)=N_{\Aut_{\F}(A)}(X)=N_{\Aut_{\F_1}(A)}(X)$. Hence, $A$
cannot be essential in $N_\F(X)$ and we have derived a contradiction. Thus, $E$
is not $\F$-conjugate into $A$. Therefore, again by \eqref{E:Fgen}, $E$ is
conjugate into $R$ under an element of $\Aut_{\F_1}(S)$. Let
$\alpha\in\Aut_{\F_1}(S)$ such that  $E^\alpha\leq R$. As $C_S(E)\leq E$, we
have then $C_S(E^\alpha)\leq E^\alpha$. Since $R$ is abelian, it follows
$E^\alpha=R$. Thus, we have
\[
\Aut_{N_\F(X)}(E)^\alpha = N_{\Aut_\F(E)}(X)^\alpha = 
N_{\Aut_\F(R)}(X^\alpha).
\]
As $1\neq X^\alpha\leq Z^\alpha=Z\leq R$ and $\Aut_\F(R)=\Aut_{\F_2}(R)$ acts
$k$-linearly on $R$, $N_{\Aut_\F(R)}(X^\alpha)$ has a normal Sylow
$p$-subgroup. Thus, $\Aut_{N_\F(X)}(E)\cong N_{\Aut_\F(R)}(X^\alpha)$ has a
normal Sylow $p$-subgroup, contradicting the fact that $E$ is essential in
$N_\F(X)$. This final contradiction shows that $N_\F(X)\subseteq N_\F(S)$ is
constrained. This completes the proof of the assertion.
\end{proof}

Our next goal will be to show that $\F:=\F(r,n,k,Q)$ does not have a punctured
group for $n\geq 3$ (i.e. in the case that $\F$ is exotic). For that we prove
that, using the notation introduced at the beginning of this subsection,
$N_\F(Z)/Z$ is exotic. The structure of $N_\F(Z)/Z$ resembles the structure of
$\F(r,n-1,k,R)$ except that the elementary abelian normal subgroup of index $q$
is not essential. Indeed, it will turn out that the problem of showing that
$N_\F(Z)/Z$ is exotic reduces to the situation treated in the following lemma,
whose proof of part (c) depends on the classification of finite simple groups.

\begin{lem}\label{L:ClParkerExotic}
Fix a power $q$ of $p$ as before. Let $S$ be an arbitrary $p$-group such that
$S=U\ltimes A$ splits as a semidirect product of an elementary abelian subgroup
$A$ with an elementary abelian subgroup $U$. Assume $|U|=q$, and $|A|=q^{n}$
for some $3\leq n\leq p-1$. Set $P:=Z(S)U$, $T:=[S,S]U$, and let $\F$ be a
saturated fusion system over $S$. Assume the following conditions hold: 
\begin{itemize}
\item [(i)] $Z(S)$ has order $q$, $[S,S]\not\leq Z(S)$, and $Z(S)=C_A(u)$ for
every $1\neq u\in U$.
\item [(ii)] $O^{p^\prime}(\Aut_\F(P))\cong \SL_2(q)$ and $P$ is a natural
$\SL_2(q)$-module for $O^{p^\prime}(\Aut_\F(P))$, 
\item [(iii)] $\F$ is generated by $\Aut_\F(P)$ and $\Aut_\F(S)$,
\item [(iv)] $\Aut_\F(S)$ acts irreducibly on $A/[S,S]$, $|A/[S,S]|\geq q$, and
\item [(v)] there is a complement to $\Inn(S)$ in $\Aut_\F(S)$ which normalizes
$U$. 
\end{itemize}
Then the following hold:
\begin{itemize}
\item[(a)] The non-trivial strongly closed subgroups of $\F$ are precisely $S$
and $T$.
\item[(b)] Neither $S$ nor $T$ can be written as the direct product of two
non-trivial subgroups.
\item[(c)] $\F$ is exotic.
\end{itemize}
\end{lem}

\begin{proof}
Observe first that (iii) implies that $P$ is fully normalized. In particular,
$\Aut_S(P)\in\Syl_p(\Aut_\F(P))$. As $Z:=Z(S)$ has order $q$, it follows from
(ii) that $Z(S)=C_P(N_S(P))=[P,N_S(P)]\leq [S,S]$. In particular, $P\leq T$. We
note also that $C_S(P)=P$ as $C_A(U)=Z(S)$ by (i).

\smallskip

\textbf{(a)} We argue first that $T$ is strongly closed. Observe that $T$ is
normal in $S$, since $T$ contains $[S,S]$. As $[S,S]$ is characteristic in $S$,
it follows thus from (v) that $T$ is $\Aut_\F(S)$-invariant. Thus, as $P\leq
T$, (iii) implies that $T$ is strongly closed in $\F$. Let now $S_0$ be a
non-trivial proper subgroup of $S$ strongly closed in $\F$. Since $S_0$ is
normal in $S$, it follows $1\neq S_0\cap Z(S)\leq P$. By (ii), $\Aut_\F(P)$
acts irreducibly on $P$. So $P\leq S_0$. Hence, $[S,S]=[A,U]\leq [S,P]\leq
[S,S_0]\leq S_0$ and thus $T=[S,S]U=[S,S]P\leq S_0$. Suppose $T<S_0$. As $U\leq
S_0\leq S=AU$, we have $S_0=(S_0\cap A)U$ and thus $[S,S]<S_0\cap A<A$. So
$\Aut_\F(S)$ does not act irreducibly on $A/[S,S]$, contradicting (iv). This
shows (a).

\smallskip

\textbf{(b)} Let $S^*\in\{S,T\}$ and assume by contradiction that
$S^*=S_1\times S_2$ where $S_1$ and $S_2$ are non-trivial subgroups of $S^*$.
Notice that in either case $Z=Z(S^*)$ by (i). Moreover, again using (i), we
note that $[S_1,S_1]\times [S_2,S_2]=[S,S]\not\leq Z=Z(S)$. So there exists in
either case $i\in\{1,2\}$ with $[S_i,S_i]\not\leq Z$ and thus $S_i\cap
A\not\leq Z$. We assume without loss of generality that $S_1\cap A\not\leq Z$.
Setting $\ov{S^*}=S^*/Z$, we note that $\ov{S_1\cap A}$ is a non-trivial normal
subgroup of $\ov{S^*}$, and intersects thus non-trivially with $Z=Z(\ov{S^*})$.
Hence, $((S_1\cap A)Z)\cap Z_2(S^*)\not\leq Z$ and so $S_1\cap A \cap
Z_2(S^*)\not\leq Z$. Choosing $s\in (S_1\cap A\cap Z_2(S^*))\backslash Z$, we
have $s\in N_S(P)\backslash P$ as $A\cap P=Z$ and $[s,P]\leq [Z_2(S^*),P]\leq
Z\leq P$. Using (ii) and $C_S(P)\leq P$, it follows $Z=[P,s]$. So $Z=[P,s]\leq
[P,S_1]\leq S_1$ as $P\leq S^*$ and $S_1$ is normal in $S^*$. Since $S_2$ is a
non-trivial normal subgroup of $S^*$, we have $S_2\cap Z=S_2\cap Z(S^*)\neq 1$.
This contradicts $S_1\cap S_2=1$. Thus, we have shown that $S^*$ cannot be
written as a direct product of two nontrivial subgroups, i.e., property (b)
holds.

\smallskip

\textbf{(c)} Part (c) follows now using the classification of
finite simple groups. Most notably, we use knowledge of the automorphism
groups of finite simple groups, Oliver's work on fusion systems over
$p$-groups with an abelian subgroup of index $p$ \cite{Oliver2014}, and the
work of Flores--Foote \cite{FloresFoote2009} determining the simple groups
having a Sylow $p$-subgroup with a proper non-trivial strongly closed subgroup.
To argue in detail, assume that $\F$ is realizable. By (b), neither $S$ nor $T$
can be written as a direct product of two non-trivial subgroups. By (a), $S$
and $T$ are the only non-trivial strongly closed subgroups. The subgroup $T$ is
$\F$-centric since $T$ is strongly closed and $P\leq T$ is self-centralizing in
$S$. Clearly, $S$ is $\F$-centric. So as $\F$ is realizable, it follows from
\cite[Proposition~2.19]{DiazRuizViruel2007} that $\F=\F_S(G)$ for some almost
simple group $G$ with $S\in\Syl_p(G)$. Set
\[
G_1:=F^*(G)\mbox{ and }\F_1:=\F_{S\cap G_1}(G_1).
\]
If $S\leq G_1$, then note that $T$ is a proper subgroup of $S$ which is
strongly closed in $S$ with respect to $G$ and thus with respect to $G_1$.
Hence, it follows work of Flores--Foote \cite{FloresFoote2009} that $p=|T|=3$,
which contradicts our assumption. (We refer the reader to
\cite[Theorem~II.12.12]{AschbacherKessarOliver2011}, which summarizes for us
the relevant part of the work of Flores--Foote.) Hence, $S\neq S\cap G_1$. As
$S\cap G_1$ is strongly closed in $\F$, it follows thus from (a) that
\[
S\cap G_1=T,
\]
i.e. $\F_1$ is a fusion system over $T$. In particular, since $T<S$, the prime
$p$ divides $G/G_1$ and thus the outer automorphism group of the simple group
$G_1$. Since $p\geq 5$, it follows that $G_1$ is not alternating or a sporadic
simple group. Hence, by the classification of finite simple groups, $G_1$ is of
Lie type. We identify $G$ now in the natural way with a subgroup of
$\Aut(G_1)$; in particular, we identify $G_1$ with $\Inn(G_1)$. Write $D$ for
the subgroup of $\Aut(G_1)$ generated by $G_1$ and the diagonal automorphisms
of $G_1$, and let $E$ be the subgroup of $\Aut(G_1)$ generated by $D$
and the group of field automorphisms of $G_1$ with respect to some
fixed maximal torus and root structure. By \cite[Theorem~2.5.12]{GLS3}, $D$
and $E$ are normal in $\Aut(G_1)$, $|\Aut(G_1):E|$ is not
divisible by $p\geq 5$,  $E/D$ is cyclic, and $D/G_1$ is either cyclic
or of order $4$. In particular, $G_0:=O^{p^\prime}(G)\leq E \cap G$,
$G_0/D\cap G_0$ is cyclic, and $D\cap G_0/G_1$ is cyclic or of order $4$.

\smallskip

If $p$ divides the order of $G_0/D\cap G_0$, then this group has a unique
subgroup of index $p$ whose preimage is then a normal subgroup of $G$ which has
index $p$ in $G_0$. Otherwise, $G_0=O^{p^\prime}(G)\leq D$ and $p$ must divide
the order of $G_0/G_1\leq D/G_1$. Thus, in this case $G_0/G_1$ is cyclic and so
there is a unique subgroup of $G_0/G_1$ of index $p$.  So in either case we
find a normal subgroup $N$ of $G$ which has index $p$ in $G_0=O^{p^\prime}(G)$.
As $N\cap S$ is strongly closed in $\F=\F_S(G)$, it follows now from (a) that
$N\cap S=T$ and $p=|S/T|$. Using (iv) we see now that $p=|S/T|=|A/[S,S]|\geq q$
and hence $q=p$. In particular, $A$ has index $p$ in $S$, and similarly,
$A_0:=[S,S]$ is an abelian subgroup of $T$ of index $p$.

\smallskip

Assume now first $n \geq 4$ and thus $|A_0|\geq p^3$. Our goal is to apply
\cite[Lemma~1.6]{Oliver2014} with $(G_1,T)$ in place of $(G,S)$, so we verify
now the hypotheses of this lemma. It follows from the last condition in (i)
that any abelian subgroup of $T$ is either contained in $A_0=[S,S]=A\cap T$ or
has order at most $p^2$. Hence, $A_0$ is the unique abelian subgroup of $T$ of
index $p$. Moreover, by (i), $Z(T)=Z$ has order $p$ and
$|[T,T]|=|[A_0,U]|=|A_0/C_{A_0}(U)|=|A_0/Z|$. This implies that $[T,T]$ has
index $p$ in $A_0$ and thus index $p^2$ in $T$. As $|A_0|\geq p^3>|P|$, it
follows from (iii) that every $\F$-automorphism of $A_0$ lifts to an
$\F$-automorphism of $S$. In particular, $\Aut_T(A_0)$ is normal in
$\Aut_\F(A_0)$ and thus also in $\Aut_{\F_1}(A_0)$. Therefore, $A_0$ is not
essential in $\F_1$. Now \cite[Lemma~1.6]{Oliver2014} implies $p=3$,
contradicting our assumption.

\smallskip

We have thus $n = 3$. So $|A_0|=p^2$ and $T$ is extraspecial of order $p^3$ and
exponent $p$. In particular, $P$ is normal in $T$ and so (ii) together with
$C_S(P)\leq P$ implies $T=N_S(P)$. As $T$ is strongly $\F$-closed, every
$\F$-conjugate of $P$ is in $T$. If $Q\in P^\F$ is fully normalized, then $P$
is conjugate to $Q$ under some element of $\Hom_\F(N_S(P),S)$, and thus by
(iii) under some element of $\Aut_\F(S)$. So  $P$ itself is fully normalized
and a similar argument  yields $P^\F=P^{\Aut_\F(S)}$. In particular, for every
$P^*\in P^\F$, the fact that $T$ is strongly closed implies that $N_S(P^*)=T$
and so
\[
O^{p^\prime}(\Aut_\F(P^*))=\<\Aut_T(P^*)^{\Aut_\F(P)}\>\leq \Aut_{\F_1}(P^*).
\]
In particular, $\Aut_{\F_1}(P^*)$ is isomorphic to a subgroup of $GL_2(p)$
containing $SL_2(p)$ and has thus a strongly $p$-embedded subgroup. Hence, the
elements of $P^\F$ are essential and thus centric radical in $\F_1$. Since
$A_0$ is normal in $S$, it follows moreover that $A_0$ is not conjugate to $P$.
As $|A_0|=|P|$, property (iii) implies thus that every element of
$\Aut_\F(A_0)$ extends to an $\F$-automorphism of $S$. In particular, $A_0$ is
not radical in $\F_1$.  As $T$ is extraspecial of order $p^3$ and exponent $p$,
$T$ has exactly $p+1$ subgroups of order $p^2$. Moreover, since $T=N_S(P)$ has
index $p$ in $S$, the conjugacy class $P^S$ has $p$ elements. This shows that
there are exactly $p$ subgroups of $T$ of order $p^2$ which are centric and
radical in $\F_1$, namely the elements of $P^S=P^\F$. However, by the
classification of Ruiz and Viruel \cite[Tables~1.1, 1.2]{RuizViruel2004}, there
is no saturated fusion system over $T$ with exactly $p$ essential subgroups of
order $p^2$. This contradiction completes the proof of (c) and the lemma.
\end{proof}

Recall that $\F(r,n,k,Q)$ is realizable in the case $n = 2$ and thus has a
punctured group. So the case $n\geq 3$, which we consider in the following
proposition, is actually the only interesting remaining case.

\begin{prop}\label{P:CPQ1}
Let $3\leq n\leq p-1$ (and thus $p\geq 5$), let $r$ be a divisor of $(n,q-1)$,
and set $\F=\F(r,n,k,Q)$. Then $N_\F(Z)$ and $N_\F(Z)/Z$ are exotic. In
particular, $\F$ does not have a punctured group.
\end{prop}

\begin{proof}
By Lemma~\ref{L:normexotic}, $\F$ does not have a punctured group if $N_\F(Z)$
is exotic. Moreover, if $N_\F(Z)$ is realized by a finite group $H$,
then $N_\F(Z)$ is also realized by $N_H(Z)$, and $N_\F(Z)/Z$ is
realized by $N_H(Z)/Z$. So it is sufficient to show that $N_\F(Z)/Z$ is exotic.

Recall from above that $S=S(n,k)$, $A=A(n,k)$ and $Z:=Z(S)$. Set $\F_1 =
\F_S(G)$ and $\F_2 = \F_{S_2}(H_Q)$ with $S_2 = N_S(Q)$. Suppose first $r=1$.
Then one argues similarly as in the proof of Proposition~\ref{P:CPR} that $\F =
\gen{\Aut_{\F_1}(S), \Aut_{\F_1}(A), \Aut_{\F_2}(Q)}$. Namely, $\F$ is
generated by $\F_1$ and $\F_2$ by \cite[Theorem~3.1]{ClellandParker2010}, and
so as $\F_1$ and $\F_2$ are both constrained with $O_p(\F_1) = A$ and
$O_p(\F_2) = Q$, it follows that $\F$ is in turn generated by $\Aut_{\F_1}(A)$,
$\Aut_{\F_1}(S)$, $\Aut_{\F_2}(Q)$, and $\Aut_{\F_2}(S_2)$. However, the last
automorphism group is redundant, since $N_{H_R}(S_2) = N_G(Q)$ induces fusion
in $\F_1$. So indeed $\F = \gen{\Aut_{\F_1}(S), \Aut_{\F_1}(A),
\Aut_{\F_2}(Q)}$ if $r=1$. This implies $\Aut_{\F_1}(S)=\Aut_\F(S)$,
$\Aut_\F(A)=\Aut_{\F_1}(A)$ and (as $N_G(Q)=N_{H_Q}(S_2)$)
$\Aut_\F(Q)=\Aut_{\F_2}(Q)$. Moreover, the set of $\F$-essential subgroups
comprises $A$ and all $\Aut_\F(S)$-conjugates of $Q$. One easily checks
that, for any saturated fusion system $\G$, a normal subsystem of $\G$ of index
prime to $p$ has the same essential subgroups as $\G$ itself. Note moreover
that, for arbitrary $r$, $\F$ is a normal subsystem of $\F(1,n,k,Q)$ of index
prime to $p$ by Lemma~\ref{L:ReduceTo1}. Hence, in any case, the $\F$-essential
subgroups are $A$ and the $\Aut_{\F_1}(S)$-conjugates of $Q$.  Since there is
a complement to $S$ in $N_G(S)$ which normalizes $U$ and thus $Q$, the
$\Aut_{\F_1}(S)$-conjugates of $Q$ are precisely the $S$-conjugates of $Q$. So,
for arbitrary $r$, we have 

\begin{eqnarray}
\label{E:FgenQ}
\F = \gen{\Aut_\F(S), \Aut_\F(A), \Aut_\F(Q)}.
\end{eqnarray}
Moreover, $\Aut_\F(S)\leq \Aut_{\F_1}(S)$, $\Aut_\F(A)\leq \Aut_{\F_1}(A)$, and
$\Aut_\F(Q)\leq \Aut_{\F_2}(Q)$. Recall also that $O^{p^\prime}(H_Q)\leq
F(r,n,k,Q)$ and thus $\SL_2(q)\cong O^{p^\prime}(\Aut_{\F_2}(Q))\leq
\Aut_\F(Q)$. 

Note that $\Aut_\F(Q)$ normalizes $Z$ and lies thus in $N_\F(Z)$. We will show
next that $N_\F(Z)$ is generated by $\Aut_\F(S)$ and $\Aut_\F(Q)$. By the
Alperin--Goldschmidt fusion theorem, it suffices to show that every essential
subgroup of $N_\F(Z)$ is an  $\Aut_\F(S)$-conjugate of $Q$. So fix an essential
subgroup $E$ of $N_\F(Z)$ and assume that $E\not\in Q^{\Aut_\F(S)}$. As
$C_S(E)\leq E$, we have $Z<E$. If $E\leq A$ then $E=A$. However,
$\Aut_\F(A)\leq\Aut_{\F_1}(A)=\Aut_G(A)$ and one observes that $S$ is normal in
$N_G(Z)$. So $\Aut_{N_\F(Z)}(A)=N_{\Aut_\F(A)}(Z)$ has a normal Sylow
$p$-subgroup, which contradicts $E$ being essential. Assume now that $E \leq
Q$. Suppose first $Z<Z(E)$. The images of the maximal abelian subgroups of $Q$
are precisely the $1$-dimensional $k$-subspaces of $Q/Z$. As $\Aut_\F(Q)$ fixes
$Z$ and acts transitively on the one-dimensional $k$-subspaces of $Q/Z$, we see
that $Z(E)$ is conjugate into $Z_2(S)=A\cap Q$ under an element of
$\Aut_\F(Q)$. So replacing $E$ by a suitable $\Aut_\F(Q)$-conjugate, we may
assume $Z(E)\leq A\cap Q$. As $Z<Z(E)$, it follows then from \eqref{E:CSU} that
$E\leq A$. As $C_S(E)\leq E$ and $A\not\leq Q$, this is a contradiction. So we
have $Z=Z(E)$. As $[E,Q]\leq [Q,Q]\leq Z$, it follows $\Aut_Q(E)\leq
C:=C_{\Aut_{N_\F(Z)}(E)}(E/Z(E))\cap C_{\Aut_{N_\F(Z)}(E)}(Z(E))$.
However, $C$ is a normal $p$-subgroup of $\Aut_{N_\F(Z) }(E)$. Thus, as $E$ is
radical in $N_\F(Z)$, we have $\Aut_Q(E)\leq C\leq \Inn(E)$. As $C_S(E)\leq E$,
it follows $E=Q$ contradicting the choice of $E$. So we have shown that $E$
lies neither in $A$ nor in $Q$. Since the choice of $E$ was arbitrary, this
means that $E$ is not $\Aut_\F(S)$-conjugate into $A$ or $Q$. So by
\eqref{E:FgenQ}, every $\F$-automorphism of $E$ extends to an $\F$-automorphism
of $S$. This implies that $\Aut_S(E)$ is normal in $\Aut_\F(E)$ and thus in
$\Aut_{N_\F(Z)}(E)$.  Again, this contradicts $E$ being essential. So we have
shown that $N_\F(Z)$ is generated by $\Aut_\F(S)$ and $\Aut_\F(Q)$. 

Set $\ov{S}=S/Z$ and $\ov{\F}=N_\F(Z)/Z$. We will check that the hypotheses of
Lemma~\ref{L:ClParkerExotic} are fulfilled with $\ov{\F}$, $\ov{S}$, $\ov{A}$,
$\ov{U}$ and $\ov{Q}$ in place of $\F$, $S$, $A$, $U$ and $P$. Part (c) of this
Lemma will then imply that $N_\F(Z)/Z$ is exotic as required. Notice that
$|\ov{U}|=|U|=q$, $|A|=q^{n+1}$ and $|\ov{A}|=q^n$. As $Q=Z_2(S)U$, we have
$\ov{Q}=Z(\ov{S})\ov{U}$. By \cite[Lemma~4.2(i)\&(iii)]{ClellandParker2010},
hypothesis (i) of Lemma~\ref{L:ClParkerExotic} holds. Recall that
$O^{p^\prime}(\Aut_{H_Q}(Q))\cong \SL_2(q)$ lies in $N_\F(Z)$. In particular,
hypothesis (ii) in Lemma~\ref{L:ClParkerExotic} holds with $\ov{\F}$ and
$\ov{Q}$ in place of $\F$ and $P$. Since we have shown above that $N_\F(Z)$ is
generated by $\Aut_\F(S)$ and $\Aut_\F(Q)$, it follows that $\ov{\F}$ fulfills
hypothesis (iii) of Lemma~\ref{L:ClParkerExotic}. Observe that there exists a
complement $K$ of $S$ in $N_G(S)$ which normalizes $U$. Then $\Aut_\F(S)\leq
\Aut_G(S)=\Inn(S)\Aut_K(S)$. Thus $\Aut_\F(S)=\Inn(S)(\Aut_K(S)\cap
\Aut_\F(S))$ and $\Aut_K(S)\cap \Aut_\F(S)$ is a complement to $\Inn(S)$ in
$\Aut_\F(S)$ which normalizes $U$. This implies that hypothesis (v) of
Lemma~\ref{L:ClParkerExotic} holds for $\ov{\F}$. 

It remains to show hypothesis (iv) of Lemma~\ref{L:ClParkerExotic} for
$\ov{\F}$. Notice that $[\ov{S},\ov{S}]=[\ov{A},\ov{U}]$ is a proper
$\mathbb{F}_q$-subspace of $\ov{A}$ and has thus index at least $q$ in
$\ov{A}$. So it remains to show the first condition in (iv). Equivalently, we
need to show that $\Aut_\F(S)=\Aut_{N_\F(Z)}(S)$ acts irreducibly on $A/[S,S]$.
For the proof, we use the representations Clelland and Parker  give for $G$ and
$H_Q$, and the way they construct the free amalgamated product; see pp.~293 and
pp.~296 in \cite{ClellandParker2010}.  Let $\xi$ be a generator of $k^\times$.
We have 
\[
g:=
 \begin{pmatrix}
  1 & 0 & 0 & 0\\
  0 & \xi^{-1} & 0 & 0\\
  0 & 0 & \xi & 0\\
  0 & 0 & 0 & 1
\end{pmatrix}
\in O^{p^\prime}(H_Q)\leq N_{F(r,n,k,Q)}(Z).
\]
In the free amalgamated product $F(1,n,k,Q)$, the element $g\in H_Q$ is
identified with 
\[
\left(1,\begin{pmatrix} 1 & 0\\ 0 & \xi\end{pmatrix},0_{A(n,k)}\right)\in N_G(Q),
\]
and this element can be seen to act by scalar multiplication with $\xi^n$ on
$y^n\in A=A(n,k)$ and thus on $A/[S,S]$. As $n\leq p-1$ and $\xi$
has order $q-1$, the action of $g$ on $A(n,k)/[S,S]$ is thus irreducible.
Hence, the action of $\Aut_\F(S)$ on $A/[S,S]$ is irreducible. This shows that
the hypothesis of Lemma~\ref{L:ClParkerExotic} is fulfilled with $\ov{\F}$ in
place of $\F$, and thus $\ov{\F}=N_\F(Z)/Z$ is exotic as required.
\end{proof}

\subsection{The Parker-Stroth systems}

Let $p \geq 5$ be a prime and $m = p-4$.  Let $A = A(m,\mathbf{F}_p)$ and $D$
be as in \S\S\ref{SS:CP}. The Parker-Stroth systems are fusion systems over the
Sylow subgroup $S$ of a semidirect product $P=Q\rtimes D$, where $Q$ is
extraspecial of order $p^{1 + (p-3)}$ and of exponent $p$, and where $Q/Z(Q)
\cong A$ as an $\mathbf{F}_pD$-module.  Then $Z := Z(S)=Z(Q)$ is of order $p$,
while $Z_2(S) \leq Q$ is elementary abelian of order $p^2$.

It turns out that $C_D(Q)$ has order $p-1$ (cf.
\cite[Lemma~2.3(i)]{ParkerStroth2015} where our $D$ is called $L$) and so $S$
can be identified with its image in $P_1:=P/C_D(Q)$. Parker and Stroth find
then a subgroup $W$ of $S$ such that $W$ is elementary abelian of order $p^2$
and
\begin{eqnarray}
\label{E:W}
W\not\leq Q.
\end{eqnarray}
We refer to \cite[p.317]{ParkerStroth2015} for more details on the embedding of
$W$ in $S$, where our $W$ is denoted $W_0$. Choose a finite group $K$ with
$K\cong p^2:SL_2(p)$, and let $C$ be the normalizer in $K$ of a Sylow
$p$-subgroup of $K$ (cf. \cite[p.315]{ParkerStroth2015}). It turns out that
$N_{P_1}(W)$ can be identified with $C$ in such a way that $W$ is identified
with $O_p(K)$ (cf. \cite[p.319]{ParkerStroth2015}). The exotic Parker-Stroth
system $\F$ at the prime $p$ is then the fusion system over $S$ of the free
amalgamated product $P_1*_CK$, where we identify $S\in\Syl_p(P)$ with its image
in $P_1$ as before. Identifying further $N_{P_1}(W)$ with $C$ (and thus
$N_S(W)$ with a Sylow $p$-subgroup of $K$), it is shown in
\cite[Lemma~3.1]{ParkerStroth2015} that $\F$ is generated by
$\F_S(P_1)=\F_S(P)$ and $\F_{N_S(W)}(K)$. Here
$\Aut_K(N_S(W))=\Aut_C(N_S(W))=\Aut_{N_P(W)}(N_S(W))$ because of the
identification in the free amalgamated product. Note that $\F_S(P)$ is
generated by $\Aut_P(Q)$ and $\Aut_P(S)$, and that $\F_{N_S(W)}(K)$ is
generated by $\Aut_K(W)\cong SL_2(p)$ and $\Aut_K(N_S(W))$. Hence we obtain
that 
\begin{equation}\label{E:ParkerStrothGen}
\F=\<\Aut_\F(Q),\Aut_\F(S),\Aut_\F(W)\> 
\end{equation}
where $\Aut_\F(Q)=\Aut_P(Q)$, $\Aut_\F(S)=\Aut_P(S)$, and $\Aut_\F(W) \cong
SL_2(p)$. In particular, as $Q\unlhd P$, we have that
\begin{eqnarray}\label{E:QAutFSinvariant}
Q\mbox{ is $\Aut_\F(S)$-invariant.}
\end{eqnarray}

\begin{prop}\label{P:PS}
Each Parker-Stroth system is of characteristic $p$-type, and so has a punctured
group in the form of its subcentric linking system.
\end{prop}

\begin{proof}
We use the notation from above. Let $Y$ be a subgroup of order $p$ in $S$ which
is fully $\F$-centralized. We need to show that $C_\F(Y)$ is constrained. For
that purpose fix a subgroup $E\in C_\F(Y)^{cr}$.

\textbf{Case: $Y \nleq Q$.} Then $C_{Q/Z}(Y)$ is of order $p$, so $C_Q(Y)$ is
elementary abelian (of order $p^2$). Hence $C_S(Y)=C_Q(Y)Y$ is abelian in this
case, and so $C_\F(Y)$ is constrained.

\smallskip

\textbf{Case: $Y \leq Q$ but $Y \nleq Z$.} Then $C_{S}(Y)$ is abelian when $p=
5$ as then $m = 1$ and $YZ/Z$ is its own orthogonal complement with
respect to the symplectic form on $Q/Z$. We may therefore assume $p \geq 7$.

We consider now the possibilities for $E$.

\textbf{Subcase: $E\not\leq Q$.} As $E\cap Q$ contains $YZ(S) \cong C_p \times
C_p$, we have then $|E|\geq p^3$ and so $E$ is not $\F$-conjugate into $W$.
Thus, it follows in this case that from \eqref{E:ParkerStrothGen} and
\eqref{E:QAutFSinvariant} that $E$ is also not $\F$-conjugate into $Q$. Hence,
again by \eqref{E:ParkerStrothGen}, every $\F$-automorphism of $E$ extends to
an element of $\Aut_\F(S)$. In particular, every element of $\Aut_{C_\F(Y)}(E)$
extends to an element of $\Aut_{C_\F(Y)}(C_S(Y))$. As $E\in C_\F(Y)^{cr}$, this
implies $E=C_S(Y)$. It follows from  \eqref{E:QAutFSinvariant} that $C_Q(Y)$ is
invariant under $\Aut_{C_\F(Y)}(E)=\Aut_{C_\F(Y)}(C_S(Y))$.

\textbf{Subcase: $E\leq Q$.} As $W\cong C_p\times C_p\cong ZY\leq E$, it
follows now from \eqref{E:W}and \eqref{E:QAutFSinvariant} that $E$ is not
$\gen{\Aut_\F(S), \Aut_\F(Q)}$-conjugate into $W$.  So by
\eqref{E:ParkerStrothGen}, every morphism in a decomposition of
$\alpha\in\Aut_{C_\F(Y)}(E)$ lies in $\Aut_\F(Q)$ or $\Aut_\F(S)$.  Hence,
using again \eqref{E:QAutFSinvariant}, we conclude that $\alpha$ extends to $Q$
and thus to an element of $\Aut_{C_\F(Y)}(C_Q(Y))$. So $C_Q(Y) \leq E$ since
$\alpha$ was chosen arbitrarily and $E\in C_\F(Y)^{cr}$. Our assumption yields
$E\leq Q\cap C_S(Y)=C_Q(Y)$ and so $E=C_Q(Y)$.

\smallskip

In the case $Z\neq Y\leq Q$, we have thus shown that $C_S(Y)$ and $C_Q(Y)$ are
the only candidates for subgroups which are centric and radical in $C_\F(Y)$,
and that $C_Q(Y)$ is $\Aut_{C_\F(Y)}(C_S(Y))$-invariant. Thus, by Alperin's
fusion theorem, $C_Q(Y)$ is normal in $C_\F(Y)$. As $p\geq 7$ and so $m\geq 3$,
it follows from the construction of $Q$ and $P$ that $C_Q(Y)\neq Z_2(S)$ and so
$C_Q(Y)$ is self-centralizing in $C_S(Y)$ (see \cite[Lemma~2.2(i),
Lemma~2.3(iii)]{ParkerStroth2015}). Therefore, $C_\F(Y)$ is constrained.

\smallskip

\textbf{Case: $Y = Z$.} Assume first that $E$ is
$\gen{\Aut_\F(S),\Aut_\F(Q)}$-conjugate into $W$. Note that $E$ has order at
least $p^2$. Hence, $E$ is in fact $\gen{\Aut_\F(S),\Aut_\F(Q)}$-conjugate to
$W\cong C_p\times C_p$. We may therefore assume that $E = W$. By \eqref{E:W}
and \eqref{E:QAutFSinvariant}, $W$ is not $\Aut_\F(S)$-conjugate into $Q$ and
so, by \eqref{E:ParkerStrothGen}, every $\F$-conjugate of $W$ is
$\Aut_\F(S)$-conjugate to $W$. In particular, $W$ is fully $\F$-normalized. As
$\Aut_\F(W)\cong SL_2(p)$, we have $\Aut_{C_\F(Z)}(W)=C_{\Aut_\F(W)}(Z)=
N_{\Aut_\F(W)}(\Aut_S(W))$ and hence every element of $\Aut_{C_\F(Z)}(W)$
extends by the saturation axioms to an element of $\Aut_\F(N_S(W))$. As
$W\not\leq Q$, it follows from \eqref{E:ParkerStrothGen} and
\eqref{E:QAutFSinvariant} that every element of $\Aut_\F(N_S(W))$ extends to an
element of $\Aut_\F(S)$. This contradicts the assumption that $W=E\in
C_\F(Y)^{cr}=C_\F(Z)^{cr}$. So $E$ is not
$\gen{\Aut_\F(S),\Aut_\F(Q)}$-conjugate into $W$. Using again
\eqref{E:ParkerStrothGen} and \eqref{E:QAutFSinvariant} one sees now that
$\alpha\in\Aut_{C_\F(Y)}(E)$ extends to an element of $\Aut_\F(S)$ if
$E\not\leq Q$, and to an element of $\Aut_\F(Q)$ if $E\leq Q$. As $E\in
C_\F(Y)^{cr}$, it follows that $E\in\{Q,S\}$. As $E$ was arbitrary, it follows
from \eqref{E:QAutFSinvariant} that $Q\unlhd C_\F(Y)=C_\F(Z)$. As $C_S(Q)\leq
Q$, it follows that $C_\F(Y)$ is constrained.

Thus in all cases, $C_\F(Y)$ is constrained. We conclude that the Parker-Stroth
systems are of characteristic $p$-type and therefore have a punctured group.
\end{proof}

\section{Punctured groups over $p_+^{1+2}$}\label{S:RV}

The main purpose of this section is to illustrate that there can be several
punctured groups associated to the same fusion system, and that the nerves of
such punctured groups (regarded as transporter systems) might not be homotopy
equivalent to the nerve of the centric linking system.  Indeed, working in the
language of localities, we will see that there can be several punctured groups
extending the centric linking locality. This is the case even though we
consider examples of fusion systems of characteristic $p$-type, and so in each
case, the subcentric linking locality exists as the ``canonical'' punctured
group extending the centric linking locality. On the other hand, we will see
that in many cases, the subcentric linking locality is indeed the only
$p^\prime$-reduced punctured group over a given fusion system. Thus,
``interesting'' punctured groups seem still somewhat rare.

\smallskip

More concretely, we will look at fusion systems over a $p$-group $S$ which is
isomorphic to $p_+^{1+2}$. Here $p_+^{1+2}$ denotes the extraspecial group of
order $p^3$ and exponent $p$ if $p$ is an odd prime, and (using a somewhat
non-standard notation) we write $p_+^{1+2}$ for the dihedral group of order $8$
if $p=2$. Note that every subgroup of order at least $p^2$ is self-centralizing
in $S$ and thus centric in every fusion system over $S$. Thus, if $\F$ is a
saturated fusion system over $S$ with centric linking locality $(\L,\Delta,S)$,
we just need to add the cyclic groups of order $p$ as objects to obtain a
punctured group. We will again use Chermak's iterative procedure, which gives a
way of expanding a locality by adding one $\F$-conjugacy class of new objects
at the time. If all subgroups of order $p$ are $\F$-conjugate, we thus only
need to complete one step to obtain a punctured group. Conversely, we will see
in this situation that a punctured group extending the centric linking locality
is uniquely determined up to a rigid isomorphism by the normalizer of an
element of order $p$. Therefore, we will restrict attention to this particular
case. More precisely, we will assume the following hypothesis.

\begin{hypothesis}\label{H:extraspecialp3}
Assume that $p$ is a prime and $S$ is a $p$-group such that $S\cong p_+^{1+2}$
(meaning here $S\cong D_8$ if $p=2$). Set $Z:=Z(S)$. Let $\F$ be a saturated
fusion system over $S$ such that all subgroups of $S$ of order $p$ are
$\F$-conjugate. 
\end{hypothesis}

It turns out that there is a fusion system $\F$ fulfilling
Hypothesis~\ref{H:extraspecialp3} if and only if $p\in\{2,3,5,7\}$; for odd $p$
this can by seen from the classification theorem by Ruiz and Viruel
\cite{RuizViruel2004}. More precisely, we obtain the following lemma.

\begin{lem}\label{L:extraspecialFusList}
Hypothesis~\ref{H:extraspecialp3} holds if and only if we are in one of the following cases:
\begin{itemize}
 \item $p=2$ and $\F$ is realized by $A_6$.
 \item $p=3$ and $\F$ is realized by $^2\!F_4(2)'$ or $J_4$.
 \item $p=5$ and $\F$ is realized by $Th$.
 \item $p=7$ and $\F$ is one of the three exotic fusion systems discovered by
 Ruiz and Viruel \cite{RuizViruel2004}.
\end{itemize}
\end{lem}

\begin{proof}
Suppose first that $p=2$ so that $S\cong D_8$. Then there are precisely two
elementary abelian subgroups of order $4$, which are the only candidates for
$\F$-essential subgroups. Indeed, all involutions are $\F$-conjugate if and
only if both of them are essential, in which case $\F$ is the $2$-fusion system
of $A_6$.

\smallskip

Suppose now that $p$ is odd. In this case the claim follows essentially from
the list provided by Ruiz-Viruel \cite[Table~1.1]{RuizViruel2004}. To see this
it should be noted that all elements of order $p$ are $\F$-conjugate if and
only if all of the $p+1$ elementary abelian subgroups of $S$ of order $p^2$ are
in the set $\F^{ec}$-rad of $\F$-centric $\F$-radical elementary abelian
subgroups.  
\end{proof}

Assume  Hypothesis~\ref{H:extraspecialp3}. One easily observes that the
$2$-fusion system of $A_6$ is of characteristic $2$-type. Therefore, it follows
from Lemma~\ref{L:rvext} that the fusion  system $\F$ is always of
characteristic $p$-type and thus the associated subcentric linking locality is
a punctured group. As discussed in Remark~\ref{R:pprimereduced}, this leads to
a host of examples for punctured groups $\L^+$ over $\F$ which are modulo a
partial normal $p^\prime$-subgroup isomorphic to a subcentric linking locality
over $\F$. One can ask whether there are more examples. Indeed, the next
theorem tells us that this is the case if and only if $p=3$. For
$p\in\{5,7\}$ our two theorems below depend on the classification of finite
simple groups.

\begin{thm}\label{T:p=3interesting}
Under Hypothesis~\ref{H:extraspecialp3} there exists a punctured group
$(\L^+,\Delta^+,S)$ over $\F$ such that $\L^+/O_{p^\prime}(\L^+)$ is not a
subcentric linking locality if and only if $p=3$.
\end{thm}

It seems that for $p=3$, the number of $3^\prime$-reduced punctured groups over
$\F$ is probably also severely limited. However, since we don't want to get
into complicated and lengthy combinatorial arguments, we will not attempt to
classify them all. Instead, we will prove the following theorem, which leads
already to the construction of interesting examples. 

\begin{thm}\label{T:esclass}
Assume Hypothesis~\ref{H:extraspecialp3}. Suppose that $\L^+$ is a punctured
group over $\F$ such that $\L^+|_{\F^c}$ is a centric linking locality over $\F$.
Then $\L^+$ is $p^\prime$-reduced. Moreover, up to a rigid isomorphism, $\L^+$
is uniquely determined by the isomorphism type of $N_{\L^+}(Z)$, and one of the
following holds.  
\begin{itemize}
\item[(a)] $\L^+$ is the subcentric linking system for $\F$; or 
\item[(b)] $p = 3$, $\F$ is the $3$-fusion system of the Tits group ${
}^2\!F_4(2)'$ and $N_{\L^+}(Z) \cong 3S_6$; or
\item[(c)] $p = 3$, $\F$ is the $3$-fusion system of $Ru$ and of $J_4$, and
$N_{\L^+}(Z) \cong 3\#\Aut(A_6)$ or an extension of $3L_3(4)$ by a field or
graph automorphism. 
\end{itemize}
Conversely, each of the cases listed in (a)-(c) occurs in an example for $\L^+$.
\end{thm}

Here the notation $A\# B$ is as in \cite[p.261]{GLS3}, namely it describes a
group $X$ with normal subgroup $N \cong A$ and quotient $X/N \cong B$, and such
that $N \nleq Z(X)$ and $X$ does not split over $N$. In the case
$3\#\Aut(A_6)$, this is the unique extension of the quasisimple group $3A_6$ by
$\Out(3A_6) \cong C_2 \times C_2$.

Before beginning the proof, we make some remarks. The $3$-fusion systems of
$Ru$ and $J_4$ are isomorphic. For $G = Ru$ and $S$ a Sylow $3$-subgroup of
$G$, one has $N_G(Z(S)) \cong 3\#\Aut(A_6)$ \cite[Table~5.3r]{GLS3}, so the
punctured group $\L^+$ in Theorem~\ref{T:esclass}(c) is the punctured group of
$Ru$ at the prime $3$ (for example, since our theorem tells us that $\L^+$ is
uniquely determined by the isomorphism type of $N_{\L^+}(Z)$). Using the
classification of finite simple groups, this can be shown to be the only
punctured group in (b) or (c) that is isomorphic to the punctured group of a
finite group.  For example, when $G = J_4$, one has $N_G(Z(S)) \cong
(6M_{22})\cdot 2$.  The $3$-fusion system of $6M_{22}$ is constrained and
isomorphic to that of $3M_{21} = 3L_3(4)$ and also that of $3M_{10}=3(A_6.2)$,
where the extension $A_6.2$ is non-split  (see \cite[Table~5.3c]{GLS3}).  If we
are in the situation of Theorem~\ref{T:esclass}(c) and $N_{\L^+}(Z(S))$ is an
extension of $3L_3(4)$ by a field automorphism, then $N_{\L^+}(Z(S))$ is a
section of $N_G(Z(S))$. Also, for $G = { }^2F_4(2)'$, the normalizer in $G$ of
a subgroup of order $3$ is solvable \cite[Proposition~1.2]{Malle1989}. 

By Lemma~\ref{L:extraspecialFusList}, for $p\in\{2,5,7\}$, there are
also saturated fusion systems over $S$, in which all subgroups of order $p$ are
conjugate. Moreover, for $p=5$, the only such fusion system is the fusion
system of the Thompson sporadic group. It should be noted here that the
Thompson group is of local characteristic $5$, and thus its punctured group is
just the subcentric linking locality. The three exotic fusion systems at
the prime $7$, which were discovered by Ruiz and Viruel, are of characteristic
$7$-type. As our theorem shows, for each of these fusion systems, the
subcentric linking locality is the only associated punctured group extending
the centric linking locality.

\smallskip

We will now start to
prove Theorem~\ref{T:p=3interesting} and Theorem~\ref{T:esclass} in a series of
lemmas. If  Hypothesis~\ref{H:extraspecialp3} holds and $\L^+$ is a punctured group over
$\F$, then $M_0:=N_{\L^+}(Z)$ is a finite group containing $S$ as a Sylow
$p$-subgroup. Moreover, $Z$ is normal in $M_0$. These properties are preserved
if we replace $M_0$ by $M:=M_0/O_{p^\prime}(M_0)$ and identify $S$ with its
image in $M$. Moreover, we have $O_{p^\prime}(M)=1$. We analyze the structure
of such a finite group $M$ in the following lemma. Most of our arguments are
elementary. However, for $p\geq 5$, we need the classification of finite simple
groups in the form of knowledge about the Schur multipliers of finite simple
groups to show in case (b) that $p=3$.

\begin{lem}\label{OneComponent}
Let $M$ be a finite group with a Sylow $p$-subgroup $S\cong p_+^{1+2}$. Assume
that $Z:=Z(S)$ is normal in $M$ and $O_{p^\prime}(M)=1$. Then one of the
following holds.
\begin{itemize}
\item[(a)] $S \norm M$ and $C_M(S)\leq S$, or
\item[(b)] $p=3$, $S\leq F^*(M)$, and $F^*(M)$ is quasisimple with
$Z(F^*(M)) = Z$. 
\end{itemize}
\end{lem}

\begin{proof}
Assume first that $S\norm M$. In this case we have $[S,E(M)]=1$ and thus $S\cap E(M)\leq Z(E(M))$. 
So by \cite[33.12]{AschbacherFGT}, $E(M)$ is a
$p^\prime$-group. Since we assume $O_{p^\prime}(M)=1$, this implies $E(M)=1$
and $F^*(M)=O_p(M)=S$. Therefore (a) holds. 

Thus, for the remainder of the proof, we will assume that $S$ is not normal in
$M$, and we will show (b). First we prove
\begin{eqnarray}
E(M) \neq 1.
\label{E:E(M)neq1}
\end{eqnarray}
Suppose $E(M) = 1$ and set $P = O_p(M)$. Note that $Z \leq P$. As
$O_{p^\prime}(M)=1$, we have $P=F^*(M)$, so $C_M(P) \leq P$ and $P \neq Z$. As
we assume that $S$ is not normal in $M$, we have moreover $P\neq S$. If $P$ is
elementary abelian of order $p^2$, then $M/P$ acts on $P$ and normalizes $Z$,
thus it embeds into a Borel subgroup of $GL_2(p)$. If $p=2$ and $P$ is cyclic
of order $4$ then $\Aut(P)$ is a $2$-group.  So $S$ is in any case normal in
$M$ and this contradicts our assumption. Thus \eqref{E:E(M)neq1} holds.

We can now show that
\begin{eqnarray}
\text{$p$ divides $|Z(K)|$ for some component $K$ of $M$.}
\label{E:pmidZ(K)}
\end{eqnarray}
First note that $p$ divides $|K|$ for each component $K$ of $M$. For otherwise,
if $p$ doesn't divide $|K|$ for some $K$, then $1 < K \leq O_{p'}(E(M))\leq
O_{p^\prime}(M)=1$, a contradiction. 

Supposing \eqref{E:pmidZ(K)} is false, $Z(E(M))$ is a $p^\prime$-group  and
thus by assumption trivial. Hence, $E(M)$ is a direct product of simple groups.
Since $Z$ is normal in $M$, $[Z,E(M)]=1$ and thus $Z\cap E(M)=1$. As the
$p$-rank of $M$ is two and $p$ divides $|K|$ for each component $K$, there can
be at most one component, call it $J$, which is then simple and normal in $M$.
As $p$ divides $|J|$ and $J$ is normal in $M$, it follows that $S \cap J \neq
1$. But then $[S \cap J, S] \leq J \cap Z = 1$ and so $S \cap J = Z$ is normal
in $J$, a contradiction.  Thus, \eqref{E:pmidZ(K)} holds.

Next we will show that
\begin{eqnarray}\label{E:almostb}
\text{$K=F^*(M)$ is quasisimple with $S\leq K$ and $Z(K)=Z$.}
\end{eqnarray}

To prove this fix a component $K$ of $M$ such that $p$ divides $|Z(K)|$.  Then
$p$ divides $|K|/|Z(K)|$ by \cite[33.12]{AschbacherFGT}. If $S$ is not a
subgroup of $K$, then $K/Z(K)$ is a perfect group with cyclic Sylow
$p$-subgroups, so $Z(K)$ is a $p'$-group by \cite[33.14]{AschbacherFGT}, a
contradiction. Therefore $S \leq K$. If there were a component $L$ of $M$
different from $K$, then we would have $[S\cap L,L]\leq [K,L]=1$, i.e. $L$
would have a central Sylow $p$-subgroup. However, we have seen above that $p$
divides the order of each component, so we would get a contradiction to
\cite[33.12]{AschbacherFGT}.  Hence, $K=F^*(G)$ is the unique component of $M$.
Note that $O_{p^\prime}(Z(K))\leq O_{p^\prime}(M)=1$ and thus $Z(K)$ is a
$p$-group. Since $[Z,K]=1$, this implies $Z=Z(K)$. Thus \eqref{E:almostb}
holds. 

To prove (b), it remains to show that $p=3$. Assume first that $p=2$ so that
$S\cong D_8$. Then $\Aut(S)$ is a $2$-group and thus $N_K(S)=SC_K(S)$. Hence,
with $\bar{K} = K/Z$, we have $N_{\bar{K}}(\bar{S})=C_{\bar{K}}(\bar{S})$.
Therefore, $\bar{K}$ has a normal $p$-complement by Burnside's Theorem (see
e.g. \cite[7.2.1]{KurzweilStellmacher2004}), a contradiction which establishes
$p\neq 2$. 

For $p \geq 5$, we appeal to the account of the Schur multipliers of the finite
simple groups given in Chapter~6 of \cite{GLS3} to conclude that, by the
classification of the finite simple groups, $K/Z(K)$ must be isomorphic to
$L_m(q)$ with $p$ dividing $(m, q-1)$, or to $U_m(q)$ with $p$ dividing $(m,
q+1)$. But each group of this form has Sylow $p$-subgroups of order at least
$p^4$, a contradiction.  
\end{proof}

\begin{lem}\label{L:p=3interesting}
Assume Hypothesis~\ref{H:extraspecialp3} and let $(\L^+,\Delta^+,S)$ be a
punctured group over $\F$. Then the following hold:
\begin{itemize}
\item [(a)] If $P\in\Delta^+$ with $|P|\geq p^2$, then $N_{\L^+}(P)$ is Sylow
$p$-constrained and thus $p$-constrained.
\item [(b)] If $p\neq 3$ then, upon identifying $S$ with its image in
$\L^+/O_{p^\prime}(\L^+)$, the triple $(\L^+/O_{p^\prime}(\L^+),\Delta^+,S)$ is
a subcentric linking locality over $\F$.  
\end{itemize}
\end{lem}

\begin{proof}
If $P\in\Delta^+$ with $|P|\geq p^2$, then $S=N_S(P)$ is a Sylow $p$-subgroup
of $N_{\L^+}(P)$. As $P$ is normal in $N_{\L^+}(P)$ and $C_S(P)\leq P$, it
follows that $N_{\L^+}(P)$ is Sylow $p$-constrained. Thus (a) holds by
Lemma~\ref{L:pconstrained}.

\smallskip

Assume now $p\neq 3$. As all subgroups of order $p$ are by assumption
$\F$-conjugate, we have by Lemma~\ref{L:LocalitiesProp}(b) and
Lemma~\ref{L:NFPNLP}(a) that $N_{\L^+}(P)\cong M:=N_{\L^+}(Z)$ for every
$P\in\Delta^+$ with $|P|=p$. Moreover, by Lemma~\ref{OneComponent},
$M/O_{p^\prime}(M)$ has a normal Sylow $p$-subgroup and is thus in particular
Sylow $p$-constrained. Hence, using (a) and Lemma~\ref{L:pconstrained}, we can
conclude that $N_{\L^+}(P)$ is $p$-constrained for every $P\in\Delta^+$.  Therefore, by
Proposition~\ref{P:GetLocalityObjectiveCharp}, the triple
$(\L^+/O_{p^\prime}(\L^+),\Delta^+,S)$ is a locality over $\F$ of objective
characteristic $p$. Since $\Delta^+=\F^s$ by Lemma~\ref{L:rvext}, part (b)
follows.  
\end{proof}

Note that Lemma~\ref{L:p=3interesting} proves one direction of
Theorem~\ref{T:p=3interesting}, whereas the other direction would follow from
Lemma~\ref{L:extraspecialFusList} and Theorem~\ref{T:esclass}.
Therefore, we will focus now on the proof of Theorem~\ref{T:esclass} and thus
consider punctured groups which restrict to the centric linking system. If
$\L^+$ is such a punctured group, then we will apply Lemma~\ref{OneComponent}
to $N_{\L^+}(Z)$. In order to do this, we need the following two lemmas.

\begin{lem}\label{L:OpprimeM}
Let $M$ be a finite group with a Sylow $p$-subgroup $S\cong p_+^{1+2}$. Assume
that $Z:=Z(S)$ is normal in $M$ and $C_M(V)\leq V$ for every subgroup $V$ of
$S$ of order at least $p^2$. Then $O_{p^\prime}(M)=1$.
\end{lem}

\begin{proof}
Set $U = O_{p'}(M)$. As $Z$ is normal in $M$, it centralizes $U$. So $\bar{S} =
S/Z$ acts on $U$. Let $x \in S - Z(S)$. Then setting $V = \gen{x, z}$, the
centralizer $C_M(V)$ contains the $p'$-group $C_U(\bar{x})$. So our hypothesis
implies $C_{U}(\bar{x}) = 1$. Hence, by
\cite[8.3.4]{KurzweilStellmacher2004}(b), $U=\langle C_U(\bar{x}):\bar{x}\in
\bar{S}^\#\rangle=1$. 
\end{proof}

\begin{lem}\label{L:MProp}
Assume Hypothesis~\ref{H:extraspecialp3} and let $(\L^+,\Delta^+,S)$ be a
punctured group over $\F$ such that $\L^+|_{\F^c}$ is a centric linking
locality over $\F$. If we set $M:=N_{\L^+}(Z)$ the following conditions hold:
\begin{itemize}
\item[(a)] $S$ is a Sylow $p$-subgroup of $M$ and $Z$ is normal in $M$,
\item[(b)] $\F_S(M) =N_\F(Z)=N_\F(S)$, and
\item[(c)] $C_M(V) \leq V$ for each subgroup $V$ of $S$ of order $p^2$.
\end{itemize}
\end{lem}

\begin{proof}
Property (a) is clearly true. Moreover, by Lemma~\ref{L:NFPNLP}(b) and
Lemma~\ref{L:rvext}, we have $\F_S(M)=N_\F(Z)=N_\F(S)$, so (b) holds. Set
$\Delta=\F^c$. By assumption $\L:=\L^+|_{\Delta}$ is a centric linking
locality. So by \cite[Proposition~1(d)]{Henke2019}, we have
$C_\L(V)\subseteq V$ for every $V\in\Delta$. Hence, for every subgroup
$V\in\Delta$, we have $C_M(V)\subseteq C_{\L^+}(V)=C_\L(V)\subseteq V$, where
the equality follows from the definition of $\L=\L^+|_\Delta$. As every
subgroup of $S$ of order at least $p^2$ contains its centralizer in $S$, each
such subgroup is $\F$-centric. Therefore (c) holds.
\end{proof}

\begin{lem}\label{L:MDetermine}
Assume Hypothesis~\ref{H:extraspecialp3} and let $(\L^+,\Delta^+,S)$ be a
punctured group over $\F$ such that $\L:=\L^+|_{\F^c}$ is a centric linking
locality over $\F$. Set $M:=N_{\L^+}(Z)$. Then one of the following conditions
holds:
\begin{itemize}
\item [(a)] $S\unlhd M$, the group $M$ is a model for $N_\F(Z)=N_\F(S)$, and
$(\L^+,\Delta^+,S)$ is a subcentric linking locality over $\F$.
\item[(b)] $p = 3$, $\F$ is the $3$-fusion system of the Tits group ${
}^2\!F_4(2)'$ and $M \cong 3S_6$; or
\item[(c)] $p = 3$, $\F$ is the $3$-fusion system of $Ru$ and of $J_4$, and $M
\cong 3\#\Aut(A_6)$ or an extension of $3L_3(4)$ by a field or graph
automorphism. 
\end{itemize}
Moreover, in either of the cases (b) and (c), $N_{\Out(S)}(\Out_\F(S))$ is a
Sylow $2$-subgroup of $\Out(S)\cong GL_2(3)$, and every element of
$N_{\Aut(S)}(\Aut_\F(S))$ extends to an automorphism of $M$. 
\end{lem}

\begin{proof}
Set $\Delta=\F^c$. By Lemma~\ref{L:OpprimeM} and Lemma~\ref{L:MProp}, we have
$O_{p^\prime}(M)=1$, $\F_S(M)=N_\F(Z)=N_\F(S)$ and $C_M(S)\leq S$. In
particular, if $S\unlhd M$, then $M$ is a model for $N_\F(Z)=N_\F(S)$. For any
$P\in\Delta$, the group $N_{\L^+}(P)=N_\L(P)$ is of characteristic $p$. As
$\Delta^+=\Delta\cup Z^\F$, if $S\unlhd M$, the punctured group
$(\L^+,\Delta^+,S)$ is of objective characteristic $p$ and thus (a) holds. 

So assume now that $S$ is not normal in $M$. By Lemma~\ref{OneComponent}, we
have then $p=3$, $K:=F^*(M)$ is quasisimple, $S\leq K$ and $Z=Z(K)$.  Set
$\ov{M}:=M/Z$ and $G:=\ov{K}$. Let $1\neq \ov{x}\in\ov{S}$. Then the preimage
$V$ of $\<\ov{x}\>$ in $S$ has order at least $3^2$. Thus, by
Lemma~\ref{L:MProp}(c), we have $C_M(V)\leq V$. A $3^\prime$-element in the
preimage of $C_{G}(\ov{x})=C_{G}(\ov{V})$ in $K$ acts trivially on $\ov{V}$ and
$Z=Z(K)$. Thus, it is contained in $C_M(V)\leq V$ and therefore trivial. Hence,
we have
\begin{eqnarray}\label{E:Covx}
C_G(\ov{x})=\ov{S}\mbox{ for every }1\neq \ov{x}\in\ov{S}.  
\end{eqnarray}
Notice also that $G$ is a simple group with Sylow $3$-subgroup $\ov{S}$, which
is elementary abelian of order $3^2$. Moreover, $\Aut_G(\ov{S})$ is contained
in a Sylow $2$-subgroup of $\Aut(\ov{S})\cong GL_2(3)$, and such a Sylow
$2$-group is semidihedral of order $16$. In particular, if $\Aut_G(\ov{S})$ has
$2$-rank at least $2$, then $\Aut_G(\ov{S})$ contains a conjugate of every
involution in $\Aut(\ov{S})$, which is impossible because of \eqref{E:Covx}.
Hence, $\Aut_G(\ov{S})$ has $2$-rank one, and is thus either cyclic of order at
most $8$ or quaternion of order $8$ (and certainly nontrivial by
\cite[7.2.1]{KurzweilStellmacher2004}). By a result of Smith and Tyrer
\cite{SmithTyrer1973}, $\Aut_G(\ov{S})$ is not cyclic of order $2$. Using
\eqref{E:Covx}, it follows from  \cite[Theorem~13.3]{Higman1968} that $G\cong
L_2(9)\cong A_6$ if $\Aut_G(\ov{S})$ is cyclic of order $4$, and from a result
of  Fletcher \cite[Lemma~1]{Fletcher1971} that $G\cong L_3(4)$ (and thus
$\Aut_G(\ov{S})$ is quaternion) if $\Aut_G(\ov{S})$ is of order $8$. 

It follows from Lemma~\ref{L:MProp}(b) that $\Aut_M(S)=\Aut_\F(S)$. Since
$C_M(S)=Z$ and $C_G(\ov{S})=\ov{S}$ by \eqref{E:Covx}, we have
$\Aut_G(\ov{S})\cong N_G(\ov{S})/C_G(\ov{S})=\ov{N_K(S)}/\ov{S}\cong
\Aut_K(S)/\Inn(S)=\Out_K(S)$. Hence, \[\Out_G(\ov{S})\cong \Aut_G(\ov{S})\cong
\Out_K(S)\leq \Out_M(S)=\Out_\F(S).\] As $p=3$, it follows from
Lemma~\ref{L:extraspecialFusList} that $\F$ is the $3$-fusion system of the
Tits group or the $3$-fusion system of $J_4$.

Consider first the case that $\F$ is the $3$-fusion system of the Tits group ${
}^2\!F_4(2)$, which has $\Out_\F(S) \cong D_8$. Then $\Out_G(\ov{S})$ cannot be
quaternion, i.e. we have $\Out_G(\ov{S})\cong C_4$ and $G = A_6$. So conclusion
(b) of the lemma holds, as $S_6$ is the only two-fold extension of $A_6$ whose
Sylow $3$-normalizer has dihedral Sylow $2$-subgroups. By
\cite[Lemma~3.1]{RuizViruel2004}, we have $\Out(S)\cong GL_2(3)$. It follows
from the structure of this group that $N_{\Out(S)}(\Out_\F(S))\cong SD_{16}$ is
a Sylow $2$-subgroup of $\Out(S)$. As $M\cong 3S_6$ has an outer automorphism
group of order $2$, it follows that every element of $N_{\Aut(S)}(\Aut_\F(S))$
extends to an automorphism of $M$. 

Assume now that $\F$ is the $3$-fusion system of $J_4$, so that $\Out_\F(S)
\cong SD_{16}$.  An extension of $3A_6$ with this data must be $3\#\Aut(A_6)$.
Suppose now $\Aut_G(\ov{S})\cong Q_8$ and $G\cong L_3(4)$. Then $\ov{M}$ must
be a two-fold extension of $L_3(4)$. However, a graph-field automorphism
centralizes a Sylow $3$-subgroup, and so $M$ must be an extension of $L_3(4)$
by a field or a graph automorphism. Hence, (c) holds in this case. If (c)
holds, then $\Out_M(S)=\Out_\F(S)\cong SD_{16}$ is always a self-normalizing
Sylow $2$-subgroup in $\Out(S)\cong GL_2(3)$. In particular, every element of
$N_{\Aut(S)}(\Aut_\F(S))$ extends to an inner automorphism of $M$. This proves
the assertion.
\end{proof}

Note that the previous lemma shows basically that, for any punctured group
$(\L^+,\Delta^+,S)$ over $\F$ which restricts to a centric linking locality,
one of the conclusions (a)-(c) in Theorem~\ref{T:esclass} holds. To give a
complete proof of  Theorem~\ref{T:esclass}, we will also need to show that each
of these cases actually occurs in an example. To construct the examples, we
will need the following two lemmas. The reader might want to recall the
definition of $\L_\Delta(M)$ from Example~\ref{E:LDeltaM}

\begin{lem}\label{L:GetL+Help}
Let $M$ be a finite group isomorphic to $3S_6$ or $3\#\Aut(A_6)$ or an
extension of $3L_3(4)$ by a field or graph automorphism. Let $S$ be a Sylow
$3$-subgroup of $M$. Then $S\cong 3_+^{1+2}$ and, writing $\Delta$ for all
subgroups of $S$ of order $3^2$, we have $\L_\Delta(M)=N_M(S)$. Moreover,
$\F_S(M)=\F_S(N_M(S))$. 
\end{lem}
\begin{proof}
It is well-known that $M$ has in all cases a Sylow $3$-subgroup isomorphic to
$3_+^{1+2}$. By definition of $\L_\Delta(M)$, clearly
$N_M(S)\subseteq\L_\Delta(M)$. Moreover, if $g\in\L_\Delta(M)$, then there
exists $P\in\Delta$ such that $P^g\leq S$.   Note that $Z:=Z(S)\unlhd M$ and
$\ov{M}:=M/Z$ has a normal subgroup $\ov{K}$ isomorphic to $A_6$ or $L_3(4)$.
Denote by $K$ the preimage of $\ov{K}$ in $M$. Then $S\leq K$ and by a Frattini
argument, $M=KN_M(S)$. Hence we can write $g=kh$ with $k\in K$ and $h\in
N_M(S)$. In order to prove that $g\in N_M(S)$ and thus $\L_\Delta(M)\subseteq
N_M(S)$, it is sufficient to show that $k\in N_M(S)$. Note that
$P^k=(P^g)^{h^{-1}}\leq S$. As $\ov{S}$ is abelian, fusion in $\ov{K}$ is
controlled by $N_{\ov{K}}(\ov{S})$. So there exists $x\in K$ such that
$\ov{kx^{-1}}\in C_{\ov{K}}(\ov{P})$. As $\ov{K}\cong A_6$ of $L_3(4)$ and
$\ov{P}$ is a non-trivial $3$-subgroup of $\ov{K}$, one sees that
$C_{\ov{K}}(\ov{P})=\ov{S}$. Hence $kx^{-1}\in S$ and $k\in N_M(S)$. This shows
$\L_\Delta(M)=N_M(S)$. By Alperin's fusion theorem, we have
$\F_S(M)=\F_S(\L_\Delta(M))=\F_S(N_M(S))$. 
\end{proof}

\begin{lem}\label{L:NLZgroup}
Assume Hypothesis~\ref{H:extraspecialp3}. If $(\L,\Delta,S)$ is a centric
linking locality over $\F$, then $N_\L(Z) = N_\L(S)$. In particular, $N_\L(Z)$
is a group which is a model for $N_\F(S)$. 
\end{lem}

\begin{proof}
By Lemma~\ref{L:rvext}, we have $N_\F(Z)=N_\F(S)$. So $Z\unlhd S$ is a fully
$\F$-normalized subgroup such that every proper overgroup of $Z$ is in $\Delta$
and $O_p(N_\F(Z))=S\in\Delta$. Hence, by \cite[Lemma~7.1]{Henke2019},
$N_\L(Z)$ is a subgroup of $\L$ which is a model for $N_\F(Z)=N_\F(S)$. Since
$N_\L(S)\subseteq N_\L(Z)$ is by Lemma~\ref{L:NFPNLP}(b) a model for $N_\F(S)$,
and a model for a constrained fusion system is by
\cite[Theorem~III.5.10]{AschbacherKessarOliver2011} unique up to isomorphism,
it follows that $N_\L(Z)=N_\L(S)$.  
\end{proof}

We are now in a position to complete the proof of Theorem~\ref{T:esclass}.

\begin{proof}[Proof of Theorem~\ref{T:esclass}] 
Assume Hypothesis~\ref{H:extraspecialp3}. By Lemma~\ref{L:MDetermine}, for
every punctured group $(\L^+,\Delta^+,S)$ over $\F$ which restricts to a
centric linking locality, one of the cases (a)-(c) of Theorem~\ref{T:esclass}
holds. It remains to show that each of these cases actually occurs in an
example and that moreover the isomorphism type of $N_{\L^+}(Z)$ determines
$(\L^+,\Delta^+,S)$ uniquely up to a rigid isomorphism. 

By Lemma~\ref{L:rvext}, we have $N_\F(Z)=N_\F(S)$ and $\F^s$ is the set of
non-trivial subgroups of $S$. Hence, the subcentric linking locality
$(\L^s,\F^s,S)$ over $\F$ is always a punctured group over $S$. Moreover, it
follows from Lemma~\ref{L:NFPNLP}(b) that $N_{\L^s}(Z)$ is a model for
$N_\F(Z)=N_\F(S)$ and thus $S$ is normal in $N_{\L^s}(Z)$ by
\cite[Theorem~III.5.10]{AschbacherKessarOliver2011}. So case (a) of
Theorem~\ref{T:esclass} occurs in an example. Moreover, if $(\L^*,\Delta^+,S)$
is a punctured group such that $\L^*|_{\Delta}$ is a centric linking locality
and $N_{\L^*}(Z)\cong N_{\L^s}(Z)$, then $N_{\L^*}(Z)$ has a normal Sylow
$p$-subgroup and is thus by Lemma~\ref{L:MDetermine} a subcentric linking
locality. Hence, by Theorem~\ref{T:subcentric}, $(\L^*,\Delta^+,S)$ is rigidly
isomorphic to $(\L^s,\F^s,S)$. 

We are now reduced to the case that $p=3$ and we are looking at punctured
groups in which the normalizer of $Z$ does not have a normal Sylow
$3$-subgroup. So assume now $p=3$.  By
Lemma~\ref{L:extraspecialFusList}, $\F$ is the $3$-fusion system of the
Tits group or of $J_4$. Let $M$ always be a finite group containing $S$ as a
Sylow $3$-subgroup and assume that one of the following holds: 
\begin{itemize} \item[(b')] $\F$ is the $3$-fusion system of the Tits group ${
}^2\!F_4(2)'$ and $M \cong 3S_6$; or 
\item[(c')] $\F$ is the $3$-fusion system of $J_4$, and
$M \cong 3\#\Aut(A_6)$ or an extension of $3L_3(4)$ by a field or graph
automorphism.  
\end{itemize} 
In either case, one checks that $C_M(S)\leq S$.  Moreover, if (b') holds, then
$\Out_\F(S)\cong D_8$ and $N_M(S)\cong 3_+^{1+2}:D_8$. As $\Out(S)\cong
GL_2(3)$ has Sylow $2$-subgroups isomorphic to $SD_{16}$ and moreover,
$SD_{16}$ has a unique subgroup isomorphic to $D_8$, it follows that
$\Out_M(S)$ and $\Out_\F(S)$ are conjugate in $\Out(S)$.  Similarly, if (c')
holds, then $\Out_\F(S)\cong SD_{16}$ and $\Out_M(S)$ are both Sylow
$2$-sugroups of $\Out(S)$ and thus conjugate in $\Out(S)$. Hence, $N_M(S)$ is
always isomorphic to a model for $N_\F(S)$ and, replacing $M$ by a suitable
isomorphic group, we can and will always assume that $N_M(S)$ is a model for
$N_\F(S)$. We have then in particular that $N_\F(S)=\F_S(N_M(S))$. 

Pick now a centric linking locality $(\L,\Delta,S)$ over $S$. By
Lemma~\ref{L:NLZgroup}, $N_\L(Z)$ is a model for $N_\F(S)$. Hence, by the model
theorem \cite[Theorem~III.5.10(c)]{AschbacherKessarOliver2011}, there exists a
group isomorphism $\lambda\colon N_\L(Z)\rightarrow N_M(S)$ which restricts to
the identity on $S$. By Lemma~\ref{L:GetL+Help}, we have $N_M(S)=\L_\Delta(M)$
and $\F_S(M)=\F_S(N_M(S))=N_\F(S)=N_\F(Z)$. Note that $N_M(S)$ and
$\L_\Delta(M)$ are actually equal as partial groups and the group isomorphism
$\lambda$ can be interpreted as a rigid isomorphism from $N_\L(Z)$ to
$\L_\Delta(M)$. So Hypothesis~5.3 in \cite{Chermak2013} holds with $Z$ in place
of $T$. Since $\Delta=\F^c$ is the set of all subgroups of $S$ of order at
least $3^2$ and as all subgroups of $S$ of order $3$ are $\F$-conjugate, the
set $\Delta^+$ of non-identity subgroups of $S$ equals $\Delta\cup Z^\F$. So by
\cite[Theorem~5.14]{Chermak2013}, there exists a punctured group
$(\L^+(\lambda),\Delta^+,S)$ over $\F$ with $N_{\L^+(\lambda)}(Z)\cong M$. Thus
we have shown that all the cases listed in (a)-(c) of Theorem~\ref{T:esclass}
occur in an example.

Let now $(\L^*,\Delta^+,S)$ be any punctured group over $\F$ such that
$\L':=\L^*|_{\Delta}$ is a centric linking locality and $N_{\L^*}(Z)\cong M$.
Pick a group isomorphism $\phi\colon M\rightarrow M^*:=N_{\L^*}(Z)$ such that
$S^\phi=S$. Then $\phi|_S$ is an automorphism of $S$ with
$(\phi|_S)^{-1}\Aut_M(S)\phi|_S=\Aut_{M^*}(S)$. Recall that $\F_S(M)=N_\F(S)$,
Moreover, by Lemma~\ref{L:NFPNLP}(b),  we have $\F_S(M^*)=N_\F(Z)=N_\F(S)$.
Hence, $\Aut_M(S)=\Aut_\F(S)=\Aut_{M^*}(S)$ and $\phi|_S\in
N_{\Aut(S)}(\Aut_\F(S))$. So by Lemma~\ref{L:MDetermine}, there exists
$\psi\in\Aut(M)$ such that $\psi|_S=\phi|_S$. Then $\mu:=\psi^{-1}\phi$ is an
isomorphism from $M$ to $M^*=N_{\L^*}(Z)$ which restricts to the identity on
$S=N_S(Z)$. Moreover, by Theorem~\ref{T:centric}, there exists a rigid
isomorphism $\beta\colon \L\rightarrow \L'$. Therefore by
\cite[Theorem~5.15(a)]{Chermak2013}, there exists a rigid isomorphism from
$(\L^+(\lambda),\Delta^+,S)$ to $(\L^*,\Delta^+,S)$. This shows that a
punctured group $(\L^+,\Delta^+,S)$ over $\F$, which restricts to a centric
linking locality, is up to a rigid isomorphism uniquely determined by the
isomorphism type of $N_{\L^+}(Z)$. 
\end{proof}

\begin{proof}[Proof of Theorem~\ref{T:p=3interesting}]
Assume Hypothesis~\ref{H:extraspecialp3}. If $p\neq 3$, then it follows from
Lemma~\ref{L:p=3interesting} that $\L^+/O_{p^\prime}(\L^+)$ is a subcentric
linking locality for every every punctured group $(\L^+,\Delta^+,S)$ over $S$.
On the other hand, if $p=3$, then Theorem~\ref{T:esclass} together with
Lemma~\ref{L:extraspecialFusList} gives the existence of a punctured group
$(\L^+,\Delta^+,S)$ over $\F$ such that $O_{p^\prime}(\L^+)=1$ and
$N_{\L^+}(Z)$ is not of characteristic $p$, i.e.  such that
$\L^+/O_{p^\prime}(\L^+)$ is not a subcentric linking locality.  
\end{proof}

\appendix

\section{Notation and background on groups of Lie type}\label{S:lie}

We record here some generalities on algebraic groups and finite groups of Lie
type which are needed in Section~\ref{S:sol}.  Our main references are
\cite{Carter1972}, \cite{GLS3}, and \cite{BrotoMollerOliver2019}, since these
references contain proofs for all of the background lemmas we need. 

Fix a prime $p$ and a semisimple algebraic group $\ol G$ over $\ol{\FF}_p$. Let
$\ol T$ be a maximal torus of $\ol G$, $W = N_{\ol G}(\ol T)/\ol T$ the Weyl
group, and let $X(\ol T) = \Hom(\ol T, \ol{\FF}_p^\times)$ be the character
group.  Let $\ol X_\alpha = \{x_\alpha(\lambda) \mid \lambda \in \ol{\FF}_p\}$
denote a root subgroup, namely a closed $\ol T$-invariant subgroup isomorphic
$\ol{\FF}_p$.  The root subgroups are indexed by the roots of $\ol T$, the
characters $\alpha \in X(\ol T)$ with $x_{\alpha}(\lambda)^t =
x_{\alpha}(\alpha(t)\lambda)$ for each $t \in \ol T$.  The character group
$X(\ol T)$ is written additively: for each $\alpha, \beta \in X(\ol T)$ and
each $t \in \ol T$, we write $(\alpha+\beta)(t) = \alpha(t)\beta(t)$. For each
$n \in N_{\ol G}(\ol T)$, $\alpha \in X(\ol T)$, and $t \in \ol T$ we write $({
}^n\alpha)(t) = \alpha(t^n)$ for the induced action of $N_{\ol G}(\ol T)$
action on $X(\ol T)$. 

Let $\Sigma(\ol{T})$ be the set of $\ol T$-roots $\alpha \in X(\ol{T})$, and
let $V = \RR \otimes_{\ZZ} X(\ol T)$ be the associated real inner product space
with $W$-invariant inner product $( , )$.  We regard $X(\ol T)$ as a subset of
$V$, and write $w_\alpha \in W$ for the reflection in the hyperplane
$\alpha^\perp$.

For each root $\alpha \in \Sigma(\ol T)$ and each $\lambda \in
\ol{\FF}_p^\times$, let $n_{\alpha}(\lambda), h_\alpha(\lambda) \in \gen{\ol
X_\alpha, \ol X_{-\alpha}}$ be the images of the elements
$[\begin{smallmatrix}0&-\lambda^{-1}\\\lambda&0\end{smallmatrix}]$,
$[\begin{smallmatrix}\lambda&0\\0&\lambda^{-1}\end{smallmatrix}]$ under the
homomorphism $SL_2(\ol{\FF}_p) \to G$ which sends
$[\begin{smallmatrix}1&0\\u&1\end{smallmatrix}]$ to $x_{\alpha}(u)$ and
$[\begin{smallmatrix}1&v\\0&1\end{smallmatrix}]$ to $x_{-\alpha}(v)$.  Thus
\begin{eqnarray}
\label{E:nalphahalpha}
n_\alpha(\lambda) = x_{\alpha}(\lambda)x_{-\alpha}(-\lambda^{-1})x_{\alpha}(\lambda) \quad \text{ and } \quad
h_{\alpha}(\lambda) = n_{\alpha}(1)^{-1}n_{\alpha}(\lambda),
\end{eqnarray}
and $n_\alpha(1)$ represents $w_\alpha$ for each $\alpha \in \Sigma$.  We
assume throughout that parametrizations of the root groups have been chosen so
that the Chevalley relations of \cite[1.12.1]{GLS3} hold. 

Although $\Sigma(\ol T)$ is defined in terms of characters of the maximal torus
$\ol T$, it will be convenient to identify $\Sigma(\ol T)$ with an abstract
root system $\Sigma$ inside some standard Euclidean space $\RR^{l}$, $( , )$,
via a $W$-equivariant bijection which preserves sums of roots
\cite[1.9.5]{GLS3}. We'll write also $V$ for this Euclidean space. The symbol
$\Pi$ denotes a fixed but arbitrary base of $\Sigma$.

The maps $h_\beta \colon \ol\FF_p^{\times} \to \ol T$, defined above for each
$\beta \in \Sigma$, are algebraic homomorphisms lying in the group of
cocharacters $X^\vee(\ol T) := \Hom(\ol\FF_p^\times, \ol T)$. Composition
induces a $W$-invariant perfect pairing $X(\ol T) \otimes_{\ZZ} X^\vee(\ol T)
\to \ZZ$ defined by $\alpha \otimes h \mapsto \gen{\alpha,h}$, where
$\gen{\alpha,h}$ is the unique integer such that $\alpha(h(\lambda)) =
\lambda^{\gen{\alpha,h}}$ for each $\lambda \in \ol\FF_p^\times$. Since
$\Sigma$ contains a basis of $V$, we can identify $V^*$ with $\RR \otimes_{\ZZ}
X^{\vee}(\ol T)$, and view $X^{\vee}(\ol T) \subseteq V^*$ via this 
pairing. Under the identification of $V$ with $V^*$ via $v \mapsto (-,v)$,
for each $\beta \in \Sigma$ there is $\beta^\vee \in V$ such that $(-,
\beta^\vee) = \gen{-, h_\beta}$ in $V^*$, namely the unique element such that
$(\beta,\beta^\vee) = 2$ and such that $w_\beta$ is reflection in the
hyperplane $\ker((-,\beta^\vee))$.  Thus, when viewed in $V$ in this way (as
opposed to in the dual space $V^*$),
$\beta^\vee = 2\beta/(\beta,\beta)$ is the abstract coroot corresponding to
$\beta$. Write $\Sigma^\vee = \{\beta^\vee \mid \beta \in \Sigma\} \subseteq
V$ for the dual root system of $\Sigma$. 

If we set $\gen{\alpha,\beta} = (\alpha,\beta^\vee) = 2(\alpha,\beta)/(\beta,
\beta)$ for each pair of roots $\alpha, \beta \in \Sigma$, then 
\begin{eqnarray}
\label{E:pairing}
\gen{\alpha,\beta} = \gen{\alpha, h_{\beta}}
\end{eqnarray}
where the first is computed in $\Sigma$, and the second is the pairing
discussed above. Equivalently,
\begin{eqnarray}
\label{E:pairingequiv}
x_{\alpha}(\mu)^{h_{\beta}(\lambda)} = x_{\alpha}(\lambda^{\gen{\alpha,\beta}}\mu)
\end{eqnarray}
for each $\alpha,\beta \in \Sigma$, each $\mu \in \ol\FF_p$, and each $\lambda
\in \ol\FF_p^\times$. 

Additional Chevalley relations we need are
\begin{align}
\label{E:xn}
x_{\alpha}(\lambda)^{n_{\beta}(1)} &= x_{w_{\beta}(\alpha)}(c_{\alpha,\beta}\lambda),\\
\label{E:hn}
h_{\alpha}(\lambda)^{n_{\beta}(1)} &= h_{w_{\beta}(\alpha)}(\lambda),\\
\label{E:nn}
n_{\alpha}(\lambda)^{n_{\beta}(1)} &= n_{w_{\beta}(\alpha)}(c_{\alpha,\beta}\lambda),\\
\label{E:n2}
n_{\alpha}(1)^2 &= h_{\alpha}(-1).
\end{align}
where 
\[
w_\beta(\alpha) = \alpha-\gen{\alpha,\beta}\beta,
\]
is the usual reflection in the hyperplane $\beta^\perp$, and where the
$c_{\alpha,\beta} \in \{\pm 1\}$, in the notation of
\cite[Theorem~1.12.1]{GLS3}, are certain signs which depend on the choice of
the Chevalley generators. This notation is related to the signs
$\eta_{\alpha,\beta}$ in \cite[Chapter~6]{Carter1972} by $c_{\alpha,\beta} =
\eta_{\beta,\alpha}$. 

Important tools for determining the signs $c_{\alpha,\beta}$ in certain cases
are proved in \cite[Propositions~6.4.2, 6.4.3]{Carter1972}, and we record
several of those results here.

\begin{lem}\label{L:signs}
Let $\alpha, \beta \in \Sigma$ be linearly independent roots. 
\begin{itemize}
\item[(1)] $c_{\alpha,\alpha} = -1$ and $c_{-\alpha, \alpha} = -1$. 
\item[(2)] $c_{-\alpha, \beta} = c_{\alpha, \beta}$. 
\item[(3)] $c_{\alpha,\beta}c_{w_\beta(\alpha), \beta} = (-1)^{\gen{\alpha,\beta}}$.
\item[(4)] If the $\beta$-root string through $\alpha$ is of the form 
\[
\alpha-s\beta, \dots, \alpha, \dots, \alpha+s\beta
\]
for some $s \geq 0$, that is, if $\alpha$ and $\beta$ are orthogonal, then
$c_{\alpha,\beta} = (-1)^s$. 
\end{itemize}
\end{lem}
\begin{proof}
The first three listed properties are proved in Proposition~6.4.3 of
\cite{Carter1972}. By the proof of that proposition, there are signs
$\epsilon_i \in \{\pm 1\}$ such that $c_{\alpha,\beta} = (-1)^s\frac{\epsilon_0
\cdots \epsilon_{s-1}}{\epsilon_0 \cdots \epsilon_{r-1}}$, whenever the
$\beta$-root string through $\alpha$ is of the form $\alpha-s\beta, \dots,
\alpha, \dots, \alpha+r\beta$. When $\alpha$ and $\beta$ are orthogonal, we
have $r-s = \gen{\alpha,\beta} = 0$, and hence $c_{\alpha,\beta} = (-1)^s$. 
\end{proof}

\begin{lem}
\label{L:liebasic}
The following hold. 
\begin{itemize}
\item[(1)] For each $\alpha, \beta \in \Sigma$, we have 
\[
\alpha(h_\beta(\lambda)) = \lambda^{\gen{\alpha, \beta}}.
\]
\item[(2)] The maximal torus $\ol T$ is generated by the $h_{\alpha}(\lambda)$ for
$\alpha \in \Sigma$ and $\lambda \in \ol\FF_{p}^\times$. If $\ol G$ is simply
connected, and $\lambda_{\alpha} \in \FF_p^\times$ are such that $\prod_{\alpha
\in \Pi} h_\alpha(\lambda_\alpha) = 1$, then $\lambda_\alpha = 1$ for all
$\alpha \in \Pi$. Thus,
\[
\ol T = \prod_{\alpha \in \Pi} h_{\alpha}(\ol \FF_p^{\times}), 
\]
and $h_\alpha$ is injective for each $\alpha$. 
\item[(3)] If $\beta, \alpha_1,\dots, \alpha_k \in \Sigma$ and $n_1,\dots,n_k \in
\ZZ$ are such that $\beta^\vee = n_1\alpha_1^\vee + \cdots + n_k\alpha_k^\vee$,
then
\[
h_{\beta}(\lambda) = h_{\alpha_1}(\lambda^{n_1})\cdots h_{\alpha_k}(\lambda^{n_k}). 
\]
\item[(4)] Define 
\[
\Phi \colon \ZZ\Sigma^{\vee} \times \ol\FF_p^{\times} \longrightarrow \ol T \quad \text{ by } \quad \Phi(\alpha^\vee, \lambda) = h_{\alpha}(\lambda).
\]
Then $\Phi$ is bilinear and $\ZZ[W]$-equivariant. It induces a surjective
$\ZZ[W]$-module homomorphism $\ZZ\Sigma^{\vee} \otimes_{\ZZ}
\ol\FF_{p}^{\times} \to \ol T$ which is an isomorphism if $\ol G$ is of
universal type. 
\end{itemize}
\end{lem}
\begin{proof}
(1) is the statement in \eqref{E:pairing} and is part of
\cite[Remark~1.9.6]{GLS3}. We refer to
\cite[Lemma~2.4(c)]{BrotoMollerOliver2019} for a proof, which is based on the
treatment in Carter \cite[pp.97-100]{Carter1972}. Part (2) is proved in
\cite[Lemma~2.4(b)]{BrotoMollerOliver2019}, and part (3) is
\cite[Lemma~2.4(d)]{BrotoMollerOliver2019}. Finally, part (4) is proved in
\cite[Lemma~2.6]{BrotoMollerOliver2019}. 
\end{proof}

\begin{prop}\label{P:subtorus-cent}
For each subgroup $X \leq \ol T$, 
\[
C_{\ol G}(X) = C_{\ol G}(X)^\circ C_{N_{\ol G}(\ol T)}(X). 
\]
The connected component $C_{\ol G}(X)^\circ$ is generated by $\ol T$ and the
root groups $\ol X_{\alpha}$ for those roots $\alpha \in \Sigma$ whose kernel
contains $X$. In particular, if $X = \gen{h_{\beta}(\lambda)}$ for some $\beta
\in \Sigma$ and some $\lambda \in \ol \FF_p^\times$ having multiplicative order
$r$, then 
\[
C_{\ol G}(X)^{\circ} = \gen{\ol T, \ol X_{\alpha} \mid \alpha \in \Sigma,\,\, r \text{ divides } \gen{\alpha,\beta}}. 
\]
\end{prop}
\begin{proof}
See \cite[Proposition~2.5]{BrotoMollerOliver2019}, which is based on
\cite[Lemma~3.5.3]{Carter1985}. The referenced result covers all but the last
statement, which then follows from the previous parts and
Lemma~\ref{L:liebasic}(1), given the definition of $r$. 
\end{proof}

\begin{prop}\label{P:faithful}
Let $\ol G$ be a simply connected, simple algebraic group over $\ol{\FF}_p$,
let $\ol T$ be a maximal torus of $\ol G$, and let $T_{r} = \{t \in \ol T \mid
t^r = 1\}$ with $r > 1$ prime to $p$. Then one of the following holds. 
\begin{itemize}
\item[(1)] $C_{\ol G}(T_{r}) = \ol T$ and $N_{\ol G}(T_{r}) = N_{\ol G}(\ol T)$. 
\item[(2)] $r = 2$, $C_{\ol G}(T_{r}) = \ol T\gen{w_0}$ for some element $w_0 \in
N_{\ol G}(\ol T)$ inverting $\ol T$, and $N_{\ol G}(T_{r}) = N_{\ol G}(\ol T)$,
\item[(3)] $r = 2$, and $\ol G = Sp_{2n}(\ol{\FF}_p)$ for some $n \geq 1$. 
\end{itemize}
\end{prop}
\begin{proof}
By Lemma~\ref{L:liebasic}(2) and since $\ol G$ is simply connected, the torus
is direct product of the images of the coroots for fundamental roots:
\begin{eqnarray}
\label{E:dp-coroots}
\ol T = \prod_{\alpha \in \Pi} h_{\alpha}(\ol\FF_p^\times).
\end{eqnarray}
Thus, if $\lambda \in \ol{\FF}_p^\times$ is a fixed element of order $r$, then
$T_{r}$ is the direct product of $\gen{h_{\alpha}(\lambda)}$ as $\alpha$ ranges
over $\Pi$.

We first look at $C_{\ol G}(T_r)^\circ$, using
Proposition~\ref{P:subtorus-cent}.  By Lemma~\ref{L:liebasic}(1), $T_{r}$ is
contained in the kernel of a root $\beta$ if and only if
$\beta(h_{\alpha}(\lambda)) = \lambda^{\gen{\beta,\alpha}} = 1$ for all simple
roots $\alpha$, i.e., if $\gen{\beta,\alpha}$ is divisible by $r$ for each
fundamental root $\alpha$. Let $\Sigma_r$ be the set of all such roots $\beta$.
For each $\alpha \in \Pi$, the reflection $w_{\alpha}$ sends a root $\beta$ to
$\beta-\gen{\beta,\alpha}\alpha$. Hence $\beta \in \Sigma_r$ if and only if
$w_{\alpha}(\beta) \in \Sigma_r$ since $\gen{-,-}$ is linear in the first
component. Since the Weyl group is generated by $w_\alpha$, $\alpha \in \Pi$,
it follows that $\Sigma_r$ is invariant under the Weyl group. By
\cite[Lemma~10.4C]{Humphreys1972}, and since $\ol G$ is simple, $W$ is
transitive on all roots of a given length, and so either $\Sigma_r =
\varnothing$, or $\Sigma_r$ contains all long roots or all short ones.  Thus,
by \cite[Table~1]{Humphreys1972}, we conclude that either $\Sigma_r =
\varnothing$, or  $r = 2$, each root in $\Pi \cap \Sigma_r$ is long, and each
$\alpha \in \Pi$ not orthogonal to $\beta$ is short and has angle $\pi/4$ or
$3\pi/4$ with $\beta$. Now by inspection of the Dynkin diagrams corresponding
to irreducible root systems, we conclude that the latter is possible only if
$\Sigma = A_1 = C_1$, $C_2$, or $C_3$. Thus, either $C_{\ol G}(T_{r})^\circ =
\ol T$ or (3) holds. 

So we may assume that $C_{\ol G}(T_r)^\circ = \ol T$. Now $N_{\ol G}(\ol T)
\leq N_{\ol G}(T_r)$ since $T_{r}$ is characteristic in $T$. As $C_{\ol
G}(T_r)^\circ = \ol T$, also $\ol T$ is normalized by $N_{\ol G}(T_{r})$, so
$N_{\ol G}(\ol T) = N_{\ol G}(T_r)$. For $r \geq 3$, it follows from
\cite[Lemma~2.7]{BrotoMollerOliver2019} that $C_{N_{\ol G}(\ol T)}(T_{r}) = \ol
T$, completing the proof of (1) in this case.

Assume now that $r = 2$ and (1) does not hold.  Let $B := C_{W}(T_2) \leq W =
N_{\ol G}(\ol T)/\ol T$. To complete the proof, we need to show $B =
\gen{-1_V}$ or else (3) holds. Here we argue as in Case 1 of the proof of
\cite[Proposition~5.13]{BrotoMollerOliver2019}. 

Let $\Lambda = \ZZ\Sigma^\vee$ be the lattice of coroots, and fix $\lambda \in
\ol{\FF}_p^\times$ of order $4$. The map $\Phi_\lambda \colon \Lambda \to \ol
T$ defined by $\Phi_{\lambda}(\alpha^\vee) = h_{\alpha}(\lambda)$ is a
$W$-equivariant homomorphism by Lemma~\ref{L:liebasic}(3). Since $\ol G$ is
simply connected, this homomorphism has kernel $4\Lambda$, image $T_4$, and it
identifies $\Lambda/2\Lambda$ with $T_2$, by Lemma~\ref{L:liebasic}(2). 

Since $B$ acts on $T_4$ and centralizes $T_2$, we have $[T_4,B] \leq T_2 \leq
C_{\ol T}(B)$, so $B$ acts quadratically on $T_4$. Since $B$ acts faithfully on
$T_4$ by (1), it follows that $B$ is a $2$-group.  

Assume that $B \neq \gen{-1_V}$.  If $B$ is of $2$-rank $1$ with center
$\gen{-1_V}$ then by assumption there is some $b \in B$ with $b^2 = -1_V$.  In
this case, $b$ endows $V$ with the structure of a complex vector space, and so
$b$ does not centralize $\Lambda/2\Lambda$, a contradiction.  Thus, there is an
involution $b \in B$ which is not $-1_V$. Let $V = V_+ \oplus V_-$ be the
decomposition of $V$ into the sum of the eigenspaces for $b$, and set
$\Lambda_{\pm} = \Lambda \cap V_{\pm}$.  Fix $v \in \Lambda$, and write $v =
v_+ + v_-$ with $v_{\pm} \in V_{\pm}$. Then $2v_{-} = v - v^b = [v,b] \in V_-
\cap 2\Lambda = 2\Lambda_-$. So $v_- \in \Lambda_{-}$, and then $v_+ \in
\Lambda_+$. This shows that $\Lambda = \Lambda_+ \oplus \Lambda_-$ with
$\Lambda_{\pm} \neq 0$. The hypotheses of
\cite[Lemma~2.8]{BrotoMollerOliver2019} thus hold, and so $G =
Sp_{2n}(\ol{\FF}_p)$ for some $n \geq 2$ by that lemma. 
\end{proof}

\bibliographystyle{amsalpha}{}
\bibliography{/home/justin/work/math/research/mybib}

\end{document}